\documentclass[a4paper,10pt,reqno]{amsart}
\usepackage{amsmath}
\usepackage{amssymb}
\usepackage{amsthm}
\usepackage[all]{xy}
\usepackage{booktabs}
\usepackage{multirow}
\usepackage{hyperref}

\newlength{\realsidemargin}
\newlength{\sidemargin}
\theoremstyle{definition}
\newtheorem{Def}{Definition}[section]
\newtheorem{rem}[Def]{Remark}

\newtheorem{ex}[Def]{Example}

\newtheorem{Not}[Def]{Notations}

\theoremstyle{plain}
\newtheorem{Prop}[Def]{Proposition}
\newtheorem{Lem}[Def]{Lemma}
\newtheorem{Thm}[Def]{Theorem}
\newtheorem{Cor}[Def]{Corollary}

\newcommand{\SwapSymbols}[1]{
	\expandafter\let\expandafter\temporarysymbol\csname #1\endcsname
	\expandafter\let\csname #1\expandafter\endcsname\csname var#1\endcsname
	\expandafter\let\csname var#1\endcsname\temporarysymbol
}

\SwapSymbols{Gamma}
\SwapSymbols{Delta}
\SwapSymbols{Theta}
\SwapSymbols{Lambda}
\SwapSymbols{Xi}
\SwapSymbols{Pi}
\SwapSymbols{Sigma}
\SwapSymbols{Upsilon}
\SwapSymbols{Phi}
\SwapSymbols{Psi}
\SwapSymbols{Omega}

\def\mbZ{\mathbb{Z}}

\def\mbC{\mathbb{C}}

\def\mcA{\mathcal{A}}

\def\mcF{\mathcal{F}}
\def\mcG{\mathcal{G}}

\def\mcL{\mathcal{L}}

\def\mcS{\mathcal{S}}
\def\mcX{\mathcal{X}}
\def\mcY{\mathcal{Y}}

\def\mfa{\mathfrak{a}}

\def\mfp{\mathfrak{p}}
\def\mfq{\mathfrak{q}}

\let\originalleft\left
\let\originalright\right
\renewcommand{\left}{\mathopen{}\mathclose\bgroup\originalleft}
\renewcommand{\right}{\aftergroup\egroup\originalright}

\makeatletter

\def\Mod{\mathop{\operator@font Mod}\nolimits}
\def\mod{\mathop{\operator@font mod}\nolimits}
\def\GrMod{\mathop{\operator@font GrMod}\nolimits}
\def\grmod{\mathop{\operator@font grmod}\nolimits}
\def\Noeth{\mathop{\operator@font Noeth}\nolimits}
\def\noeth{\mathop{\operator@font noeth}\nolimits}
\def\QCoh{\mathop{\operator@font QCoh}\nolimits}
\def\coh{\mathop{\operator@font coh}\nolimits}
\def\Proj{\mathop{\operator@font Proj}\nolimits}
\def\Hom{\mathop{\operator@font Hom}\nolimits}
\def\End{\mathop{\operator@font End}\nolimits}
\def\Ext{\mathop{\operator@font Ext}\nolimits}
\def\Ker{\mathop{\operator@font Ker}\nolimits}
\def\Im{\mathop{\operator@font Im}\nolimits}
\def\Cok{\mathop{\operator@font Cok}\nolimits}

\def\RHom{\mathop{\operator@font \mathbf{R}Hom}\nolimits}

\def\Spec{\mathop{\operator@font Spec}\nolimits}
\def\Max{\mathop{\operator@font Max}\nolimits}
\def\Supp{\mathop{\operator@font Supp}\nolimits}
\def\Ass{\mathop{\operator@font Ass}\nolimits}
\def\ASpec{\mathop{\operator@font ASpec}\nolimits}
\def\ASupp{\mathop{\operator@font ASupp}\nolimits}
\def\AAss{\mathop{\operator@font AAss}\nolimits}
\def\asupp{\mathop{\operator@font asupp}\nolimits}

\def\projdim{\mathop{\operator@font proj.dim}\nolimits}
\def\injdim{\mathop{\operator@font inj.dim}\nolimits}
\def\gldim{\mathop{\operator@font gl.dim}\nolimits}

\def\Zg{\mathop{\operator@font Zg}\nolimits}
\def\Ann{\mathop{\operator@font Ann}\nolimits}

\renewcommand{\p@enumii}{}

\makeatother

\title{Specialization orders on atom spectra of Grothendieck categories}
\subjclass[2010]{18E15 (Primary), 16D90, 16G30, 13C05 (Secondary)}
\keywords{Atom spectrum; Grothendieck category; partially ordered set; colored quiver}

\author{Ryo Kanda}
\thanks{The author is a Research Fellow of Japan Society for the Promotion of Science. This work is supported by Grant-in-Aid for JSPS Fellows 25$\cdot$249.}
\address{Graduate School of Mathematics, Nagoya University, Furo-cho, Chikusa-ku, Nagoya-shi, Aichi-ken, 464-8602, Japan}
\email{kanda.ryo@a.mbox.nagoya-u.ac.jp}
\begin{document}

%%%%%%%%%%%%%%%%%%%%%%%%%%%%%%%%%%%%%%%%%%%%%%%%%%%%%%%%%%%%%%%%%%%%%%%%%%%%%%%%
\begin{abstract}
	We introduce systematic methods to construct Grothendieck categories from colored quivers and develop a theory of the specialization orders on the atom spectra of Grothendieck categories. We show that any partially ordered set is realized as the atom spectrum of some Grothendieck category, which is an analog of Hochster's result in commutative ring theory. We also show that there exists a Grothendieck category which has empty atom spectrum but has nonempty injective spectrum.
\end{abstract}
%%%%%%%%%%%%%%%%%%%%%%%%%%%%%%%%%%%%%%%%%%%%%%%%%%%%%%%%%%%%%%%%%%%%%%%%%%%%%%%%

\maketitle
\tableofcontents

%%%%%%%%%%%%%%%%%%%%%%%%%%%%%%%%%%%%%%%%%%%%%%%%%%%%%%%%%%%%%%%%%%%%%%%%%%%%%%%%
\section{Introduction}
\label{sec:Introduction}
%%%%%%%%%%%%%%%%%%%%%%%%%%%%%%%%%%%%%%%%%%%%%%%%%%%%%%%%%%%%%%%%%%%%%%%%%%%%%%%%

The aim of this paper is to provide systematic methods to construct Grothendieck categories with certain structures and to establish a theory of the specialization orders on the spectra of Grothendieck categories. There are important Grothendieck categories appearing in representation theory of rings and algebraic geometry: the category $\Mod R$ of (right) modules over a ring $R$, the category $\QCoh X$ of quasi-coherent sheaves on a scheme $X$ (\cite[Lemma 2.1.7]{Conrad}), and the category of quasi-coherent sheaves on a noncommutative projective scheme introduced by Verevkin \cite{Verevkin} and Artin and Zhang \cite{ArtinZhang}. Furthermore, by using the Gabriel--Popescu embedding (\cite[Proposition]{PopescoGabriel}; Theorem \ref{Thm:GabrielPopescuEmbedding}), it is shown that any Grothendieck category can be obtained as the quotient category of the category of modules over some ring by some localizing subcategory. In this sense, the notion of Grothendieck category is ubiquitous.

In commutative ring theory, Hochster characterized the topological spaces appearing as the prime spectra of commutative rings with Zariski topologies (\cite[Theorem 6 and Proposition 10]{Hochster}; Theorem \ref{Thm:HochsterTheorem}). Speed \cite{Speed} pointed out that Hochster's result gives the following characterization of the partially ordered sets appearing as the prime spectra of commutative rings.

\begin{Thm}[{Hochster \cite[Proposition 10]{Hochster}} and {Speed \cite[Corollary 1]{Speed}}; Corollary \ref{Cor:SpectralPosetIsInverseLimitOfFinitePoset}]\label{Thm:IntroSpectralPosetIsInverseLimitOfFinitePoset}
	Let $P$ be a partially ordered set. Then $P$ is isomorphic to the prime spectrum of some commutative ring with the inclusion relation if and only if $P$ is an inverse limit of finite partially ordered sets in the category of partially ordered sets.
\end{Thm}

We show a theorem of the same type for Grothendieck categories. In \cite{Kanda1} and \cite{Kanda2}, we investigated Grothendieck categories by using the \emph{atom spectrum} $\ASpec\mcA$ of a Grothendieck category $\mcA$ (Definition \ref{Def:AtomSpectrum}). It is the set of equivalence classes of monoform objects, which generalizes the prime spectrum of a commutative ring: for a commutative ring $R$, there exists a canonical bijection between $\ASpec(\Mod R)$ and the prime spectrum $\Spec R$ (Proposition \ref{Prop:AtomSpectrumForCommutativeRing}). In this paper, we discuss two structures of the atom spectrum of a Grothendieck category: the topology and the partial order. The topology on the atom spectrum is not a direct analog of the Zariski topology in commutative ring theory. In the case of a commutative ring $R$, the open subsets of $\ASpec(\Mod R)$ correspond to the specialization-closed subsets of $\Spec R$ (Proposition \ref{Prop:TopologyOnAtomSpectrumOfCommutativeRing}). On the other hand, the partial order on the atom spectrum is a generalization of the inclusion relation between prime ideals of a commutative ring (Proposition \ref{Prop:PartialOrderBetweenAtomOfCommutativeRing}). Further correspondences between notions of $\ASpec\mcA$ and those of $\Spec R$ are collected in Table \ref{tb:CorrespondingNotions}.

\begin{table}\label{tb:CorrespondingNotions}
	\caption{Corresponding notions on $\ASpec\mcA$ and $\Spec R$}
	\begin{tabular}{cc}
		\toprule
		Grothendieck category $\mcA$ & Commutative ring $R$ \\
		\midrule
		Atom spectrum $\ASpec\mcA$ & Prime spectrum $\Spec R$ \\
		Atom $\alpha$ in $\mcA$ & Prime ideal $\mfp$ of $R$ \\
		Maximal atoms in $\mcA$ & Maximal ideals of $R$ \\
		Open points in $\ASpec\mcA$ & Maximal ideals of $R$ \\
		Minimal atoms in $\mcA$ & Minimal prime ideals of $R$ \\
		($=$Closed points in $\ASpec\mcA$) & \\
		Associated atoms $\AAss M$ & Associated primes $\Ass M$ \\
		Atom support $\ASupp M$ & Support $\Supp M$ \\
		Open subsets of $\ASpec\mcA$ & Specialization-closed subsets of $\Spec R$ \\
		Closure $\overline{\{\alpha\}}$ of an atom $\alpha$ & $\{\mfq\in\Spec R\mid\mfq\subset\mfp\}$ for a prime ideal $\mfp$ \\
		Generic point in $\ASpec\mcA$ & Unique maximal ideal of $R$ \\
		\bottomrule
	\end{tabular}
\end{table}

Recall that a topological space $X$ is called \emph{Kolmogorov} (also called a $T_{0}$\emph{-space}) if for any distinct points $x_{1}$ and $x_{2}$ in $X$, there exists an open subset of $X$ containing exactly one of them. A topological space $X$ is called \emph{Alexandroff} if the intersection of any family of open subsets of $X$ is also open. For a Kolmogorov space $X$, we can regard $X$ as a partially ordered set by the \emph{specialization order} on $X$. It is well known that this correspondence gives a bijection between homeomorphism classes of Kolmogorov Alexandroff spaces and isomorphism classes of partially ordered sets (Proposition \ref{Prop:BijectionBetweenKolmogorovAlexandroffSpacesAndPartiallyOrderedSets} (\ref{item:PropBijectionBetweenKolmogorovAlexandroffSpacesAndPartiallyOrderedSets})).

For a Grothendieck category $\mcA$, we show that $\ASpec\mcA$ is a Kolmogorov space (Proposition \ref{Prop:AtomSpectrumIsKolmogorovSpace}). We regard $\ASpec\mcA$ as a partially ordered set together with the specialization order. For a commutative ring $R$, the topological space $\ASpec(\Mod R)$ is Alexandroff. Hence the topology and the partial order can be recovered from each other. In general, since the topological space $\ASpec\mcA$ is not necessarily Alexandroff (Example \ref{ex:AtomSpectrumOfCategoryOfGradedModules}), the topology may have more information than the partial order. In this paper, we mainly focus on the partial order on the atom spectrum since the topology has a more complicated structure. Our main result is the following theorem, which contains a complete characterization of the partially ordered sets appearing as the atom spectra of Grothendieck categories.

\begin{Thm}[Theorem \ref{Thm:PosetOfGrothendieckCategory}]\label{Thm:IntroPosetOfGrothendieckCategory}\leavevmode
	\begin{enumerate}
		\item\label{item:ThmIntroTopologicalSpaceOfGrothendieckCategory} Any Kolmogorov Alexandroff space is homeomorphic to the atom spectrum of some Grothendieck category.
		\item\label{item:ThmIntroPosetOfGrothendieckCategory} Any partially ordered set is isomorphic to the atom spectrum of some Grothendieck category as a partially ordered set.
	\end{enumerate}
\end{Thm}

The characterization of the topological spaces appearing as the atom spectra of Grothendieck categories remains open.

Kaplansky's question, asking which partially ordered sets appear as the prime spectra of commutative \emph{noetherian} rings, is still unsolved (see \cite{WiegandWiegand} for the details). For \emph{finite} partially ordered sets, we have the following result.

\begin{Thm}[{de Souza Doering and Lequain \cite[Theorem B]{deSouzaDoeringLequain}}; Theorem \ref{Thm:FinitePosetOfCommutativeNoetherianRing}]\label{Thm:IntroFinitePosetOfCommutativeNoetherianRing}
	Let $P$ be a \emph{finite} partially ordered set. Then $P$ is isomorphic to the prime spectrum of some commutative \emph{noetherian} ring with the inclusion relation if and only if there does not exist a chain of the form $x<y<z$ in $P$.
\end{Thm}

Since Theorem \ref{Thm:IntroSpectralPosetIsInverseLimitOfFinitePoset} implies that any finite partially ordered set appears as the prime spectrum of some commutative ring not necessarily noetherian, the answer to Kaplansky's question is expected to be quite different from Theorem \ref{Thm:IntroSpectralPosetIsInverseLimitOfFinitePoset}.

The partially ordered sets appearing as the atom spectra of locally noetherian Grothendieck categories satisfy the ascending chain condition similarly to the case of commutative rings (Proposition \ref{Prop:ExistenceOfMaximalAtom}). On the other hand, we show that the locally noetherianness does not restrict the structure of partial orders in the case of finite partially ordered sets.

\begin{Thm}[Corollary \ref{Cor:FinitePosetOfLocallyNoetherianGrothendieckCategory}]\label{Thm:IntroPosetOfLocallyNoetherianGrothendieckCategory}\leavevmode
	\begin{enumerate}
		\item\label{item:ThmIntroTopologicalSpaceOfLocallyNoetherianGrothendieckCategory} Any finite Kolmogorov space is homeomorphic to the atom spectrum of some locally noetherian Grothendieck category.
		\item\label{item:ThmIntroPosetOfLocallyNoetherianGrothendieckCategory} Any finite partially ordered set is isomorphic to the atom spectrum of some locally noetherian Grothendieck category as a partially ordered set.
	\end{enumerate}
\end{Thm}

In order to show Theorem \ref{Thm:IntroPosetOfGrothendieckCategory} and Theorem \ref{Thm:IntroPosetOfLocallyNoetherianGrothendieckCategory}, we introduce a construction of Grothendieck categories from colored quivers (Definition \ref{Def:GrothendieckCategoryAssociatedWithColoredQuiver}). By using this, we also construct the following Grothendieck category.

\begin{Thm}[Theorem \ref{Thm:GrothendieckCategoryWithNoAtom}]\label{Thm:IntroGrothendieckCategoryWithNoAtom}
	There exists a Grothendieck category $\mcA$ which has empty atom spectrum but has nonempty injective spectrum, that is, $\mcA$ has no atom but has at least one indecomposable injective object.
\end{Thm}

In section \ref{sec:AtomSpectrum}, we recall fundamental results on the atom spectrum of an abelian category. An element of the atom spectrum is called an \emph{atom}. In section \ref{sec:TopologicalPropertiesOfAtomSpectra}, we show some topological properties of the atom spectrum. In section \ref{sec:PartialOrdersOnAtomSpectra}, we introduce a partial order on the atom spectrum and investigate the maximality and the minimality of atoms. We show the existence of a maximal atom and that of a minimal atom under some conditions of noetherianness. In section \ref{sec:AtomSpectraOfQuotientCategories}, we investigate the atom spectrum of the quotient category of a Grothendieck category by a localizing subcategory. In section \ref{sec:Localization}, we introduce the localization of a Grothendieck category at an atom. This is a reformulation of the localization of a Grothendieck category written in \cite{Popescu}. We see that this operation is closely related to the partial order on the atom spectrum. In section \ref{sec:ConstructionOfGrothendieckCategories}, we introduce a construction of Grothendieck categories from colored quivers and show Theorem \ref{Thm:IntroPosetOfGrothendieckCategory} and Theorem \ref{Thm:IntroPosetOfLocallyNoetherianGrothendieckCategory}. In section \ref{sec:ExamplesOfGrothendieckCategories}, we construct several Grothendieck categories which have remarkable structures. We show Theorem \ref{Thm:IntroGrothendieckCategoryWithNoAtom} as one of these constructions.

\begin{Not}\label{Not:Notation}
	Any ring is supposed to be an associative ring with an identity element. Any module over a ring is supposed to be a right module. For a ring $R$, denote by $\Mod R$ the category of right $R$-modules and by $\mod R$ the category of finitely generated right $R$-modules. Any field is supposed to be commutative.
	
	Any graded ring is supposed to be positively graded, and any graded module is $\mbZ$-graded. For a graded ring $T$, denote by $\GrMod T$ the category of graded right $T$-modules. For a graded right $T$-module $M=\bigoplus_{i\in\mbZ}M_{i}$ and $j\in\mbZ$, denote by $M(j)$ the graded right $T$-module which is isomorphic to $M$ as a nongraded right $T$-module and has the grading defined by $M(j)_{i}=M_{i+j}$.
\end{Not}

%%%%%%%%%%%%%%%%%%%%%%%%%%%%%%%%%%%%%%%%%%%%%%%%%%%%%%%%%%%%%%%%%%%%%%%%%%%%%%%%
\section{Atom spectrum}
\label{sec:AtomSpectrum}
%%%%%%%%%%%%%%%%%%%%%%%%%%%%%%%%%%%%%%%%%%%%%%%%%%%%%%%%%%%%%%%%%%%%%%%%%%%%%%%%

In this section, we recall some fundamental results on the atom spectrum $\ASpec\mcA$ of an abelian category $\mcA$, especially in the case where $\mcA$ is a Grothendieck category.

We recall the definitions of main subjects in this paper.

\begin{Def}\label{Def:GrothendieckCategory}\leavevmode
	\begin{enumerate}
		\item\label{item:DefGrothendieckCategory} An abelian category $\mcA$ is called a \emph{Grothendieck category} if $\mcA$ has exact direct limits and a generator.
		\item\label{item:DefLocallyNoetherianGrothendieckCategory} A Grothendieck category $\mcA$ is called \emph{locally noetherian} if there exists a generating set of $\mcA$ consisting of noetherian objects.
	\end{enumerate}
\end{Def}

\begin{rem}\label{rem:ExactDirectLimits}
	It is well known that the existence of exact direct limits in an abelian category $\mcA$ is equivalent to that $\mcA$ satisfies both of the following conditions.
	\begin{enumerate}
		\item $\mcA$ has arbitrary direct sums.
		\item Let $M$ be an object in $\mcA$, $N$ a subobject of $M$, and $\mcL=\{L_{\lambda}\}_{\lambda\in\Lambda}$ a family of subobjects of $M$ such that any finite subfamily of $\mcL$ has an upper bound in $\mcL$. Then we have
		\begin{equation*}
			\left(\sum_{\lambda\in\Lambda}L_{\lambda}\right)\cap N=\sum_{\lambda\in\Lambda}(L_{\lambda}\cap N).
		\end{equation*}
	\end{enumerate}
\end{rem}

We say that a Grothendieck category is \emph{nonzero} if it has a nonzero object.

In \cite[Theorem 2.9]{Mitchell}, it is shown that any object $M$ in a Grothendieck category $\mcA$ has its injective envelope $E(M)$ in $\mcA$.

For a Grothendieck category $\mcA$, denote by $\Noeth\mcA$ the full subcategory of $\mcA$ consisting of objects which are the sum of their noetherian subobjects.

\begin{Prop}\label{Prop:LocallyNoetherianGrothendieckCategoryFromGrothendieckCategory}
	For any Grothendieck category $\mcA$, the full subcategory $\Noeth\mcA$ of $\mcA$ is closed under subobjects, quotient objects, and arbitrary direct sums. In particular, $\Noeth\mcA$ is a locally noetherian Grothendieck category.
\end{Prop}

\begin{proof}
	It is immediate that $\Noeth\mcA$ is closed under quotient objects and direct sums. Let $M$ be an object in $\mcA$ belonging to $\Noeth\mcA$, $L$ a subobject of $M$, and $\{M_{\lambda}\}_{\lambda\in\Lambda}$ the family of all the noetherian subobjects of $M$. Then by the definition of the Grothendieck categories, we have
	\begin{equation*}
		L=\left(\sum_{\lambda\in\Lambda}M_{\lambda}\right)\cap L=\sum_{\lambda\in\Lambda}(M_{\lambda}\cap L).
	\end{equation*}
	Since $M_{\lambda}\cap L$ is a noetherian subobject of $L$ for any $\lambda\in\Lambda$, we deduce that $L$ belongs to $\Noeth\mcA$.
\end{proof}

Monoform objects are used in order to define the atom spectrum of $\mcA$.

\begin{Def}\label{Def:MonoformObject}
	Let $\mcA$ be an abelian category.
	\begin{enumerate}
		\item\label{item:DefMonoformObject} An object $H$ in $\mcA$ is called \emph{monoform} if for any nonzero subobject $L$ of $H$, there exists no common nonzero subobject of $H$ and $H/L$, that is, there does not exist a nonzero subobject of $H$ which is isomorphic to a subobject of $H/L$.
		\item\label{item:DefAtomEquivalence} We say that monoform objects $H_{1}$ and $H_{2}$ in $\mcA$ are \emph{atom-equivalent} if there exists a common nonzero subobject of $H_{1}$ and $H_{2}$.
	\end{enumerate}
\end{Def}

Recall that a nonzero object $U$ in an abelian category $\mcA$ is called \emph{uniform} if for any nonzero subobjects $L_{1}$ and $L_{2}$ of $U$, we have $L_{1}\cap L_{2}\neq 0$. It is easy to see that any nonzero subobject of a uniform object in $\mcA$ is also uniform. A similar result holds for monoform objects.

\begin{Prop}\label{Prop:PropertyOfMonoformObject}
	Let $\mcA$ be an abelian category.
	\begin{enumerate}
		\item\label{item:PropSubobjectOfMonoformObjectIsMonoform} Any nonzero subobject of a monoform object in $\mcA$ is also monoform.
		\item\label{item:PropMonoformObjectIsUniform} Any monoform object in $\mcA$ is uniform.
	\end{enumerate}
\end{Prop}

\begin{proof}
	(\ref{item:PropSubobjectOfMonoformObjectIsMonoform}) \cite[Proposition 2.2]{Kanda1}.
	
	(\ref{item:PropMonoformObjectIsUniform}) \cite[Proposition 2.6]{Kanda1}.
\end{proof}

As in \cite[Proposition 2.8]{Kanda1}, the atom equivalence is an equivalence relation between monoform objects. Hence we can consider the quotient class by the atom equivalence.

\begin{Def}\label{Def:AtomSpectrum}
	Let $\mcA$ be an abelian category. Denote by $\ASpec\mcA$ the quotient class of the class of monoform objects by the atom equivalence, and call it the \emph{atom spectrum} of $\mcA$. We call an element of $\ASpec\mcA$ an \emph{atom} in $\mcA$. The equivalence class of a monoform object $H$ in $\mcA$ is denoted by $\overline{H}$.
\end{Def}

The notion of atoms was introduced by Storrer \cite{Storrer} in the case of the categories of modules over rings. Atoms in abelian categories play a central role throughout this paper.

Note that $\ASpec\mcA$ is not necessarily a set in general. However, in the case where $\mcA$ is a Grothendieck category, we can show that it is a set.

\begin{Prop}\label{Prop:AtomSpectrumAndGeneratingSet}\leavevmode
	\begin{enumerate}
		\item\label{item:PropAtomSpectrumAndGeneratingSet} Let $\mcA$ be an abelian category with a generating set $\mcG$. Then for any atom $\alpha$ in $\mcA$, there exist an object $G$ in $\mcA$ belonging to $\mcG$ and a subobject $L$ of $G$ such that $G/L$ is a monoform object in $\mcA$ satisfying $\overline{G/L}=\alpha$.
		\item\label{item:PropAtomSpectrumIsNotNecessarilySet} If $\mcA$ is a Grothendieck category, then $\ASpec\mcA$ is a set.
	\end{enumerate}
\end{Prop}

\begin{proof}
	(\ref{item:PropAtomSpectrumAndGeneratingSet}) Let $H$ be a monoform object in $\mcA$ such that $\overline{H}=\alpha$. There exist an object $G$ in $\mcA$ belonging to $\mcG$ and a nonzero morphism $f\colon G\to H$. Then $G/\Ker f$ is isomorphic to the nonzero subobject $\Im f$ of $H$. By Proposition \ref{Prop:PropertyOfMonoformObject} (\ref{item:PropSubobjectOfMonoformObjectIsMonoform}), $G/\Ker f$ is a monoform object in $\mcA$ which is atom-equivalent to $H$.
	
	(\ref{item:PropAtomSpectrumIsNotNecessarilySet}) This follows from (\ref{item:PropAtomSpectrumAndGeneratingSet}) since the collection of quotient objects of an object in $\mcA$ is a set as in \cite[Proposition IV.6.6]{Stenstrom}.
\end{proof}

By using these facts, we obtain a description of the atom spectra of the categories of modules over rings and the categories of graded modules over graded rings.

\begin{Cor}\label{Cor:AtomIsRepresentedByOneGeneratedModule}\leavevmode
	\begin{enumerate}
		\item\label{item:CorAtomIsRepresentedByOneGeneratedModule} Let $R$ be a ring and $\alpha$ an atom in $\Mod R$. Then there exists a right ideal $J$ of $R$ such that $R/J$ is a monoform object in $\Mod R$ satisfying $\overline{R/J}=\alpha$. In particular, $\ASpec(\Mod R)$ is a set.
		\item\label{item:CorAtomIsRepresentedByOneGeneratedGradedModule} Let $T$ be a graded ring and $\alpha$ an atom in $\GrMod T$. Then there exist a homogeneous right ideal $J$ of $T$ and $i\in\mbZ$ such that $(T/J)(i)$ is a monoform object in $\GrMod T$ satisfying $\overline{(T/J)(i)}=\alpha$. In particular, $\ASpec(\GrMod T)$ is a set.
	\end{enumerate}
\end{Cor}

\begin{proof}
	These follow from Proposition \ref{Prop:AtomSpectrumAndGeneratingSet} since $\Mod R$ is a Grothendieck category with a generator $R$, and $\GrMod T$ is a Grothendieck category with a generating set $\{T(i)\mid i\in\mbZ\}$.
\end{proof}

The following result shows that the atom spectrum of an abelian category is a generalization of the prime spectrum of a commutative ring.

\begin{Prop}[{\cite[p.\ 631]{Storrer}}]\label{Prop:AtomSpectrumForCommutativeRing}
	Let $R$ be a commutative ring. Then the map $\mfp\mapsto\overline{R/\mfp}$ is a bijection between $\Spec R$ and $\ASpec(\Mod R)$.
\end{Prop}

We can consider the following collections of atoms associated to an object in an abelian categories.

\begin{Def}\label{Def:AtomSupportAndAssociatedAtoms}
	Let $\mcA$ be an abelian category and $M$ an object in $\mcA$.
	\begin{enumerate}
		\item\label{item:DefAtomSupport} Define a subclass $\ASupp M$ of $\ASpec\mcA$ by
		\begin{equation*}
			\ASupp M=\{\alpha\in\ASpec\mcA\mid\alpha=\overline{H}\text{ for a monoform subquotient }H\text{ of }M\},
		\end{equation*}
		and call it the \emph{atom support} of $M$.
		\item\label{item:DefAssociatedAtoms} Define a subclass $\AAss M$ of $\ASpec\mcA$ by
		\begin{equation*}
			\AAss M=\{\alpha\in\ASpec\mcA\mid\alpha=\overline{H}\text{ for a monoform subobject }H\text{ of }M\},
		\end{equation*}
		and call an element of it an \emph{associated atom} of $M$.
	\end{enumerate}
\end{Def}

We show that these notions are generalizations of supports and associated primes in commutative ring theory in Proposition \ref{Prop:AtomSupportAndAssociatedAtomsForCommutativeRing}. These notions satisfy the following fundamental properties.

\begin{Prop}\label{Prop:AtomSupportAndAssociatedAtomsAndShortExactSequence}
	Let $\mcA$ be an abelian category and $0\to L\to M\to N\to 0$ an exact sequence in $\mcA$.
	\begin{enumerate}
		\item\label{item:PropAtomSupportAndShortExactSequence} $\ASupp M=\ASupp L\cup\ASupp N$.
		\item\label{item:PropAssociatedAtomsAndShortExactSequence} \textnormal{(\cite[Proposition 3.1]{Storrer})} $\AAss L\subset\AAss M\subset\AAss L\cup\AAss N$.
	\end{enumerate}
\end{Prop}

\begin{proof}
	(\ref{item:PropAtomSupportAndShortExactSequence}) \cite[Proposition 3.3]{Kanda1}.
	
	(\ref{item:PropAssociatedAtomsAndShortExactSequence}) \cite[Proposition 3.5]{Kanda1}.
\end{proof}

\begin{Prop}\label{Prop:AtomSupportAndAssociatedAtomsAndDirectSum}
	Let $\mcA$ be a Grothendieck category and $\{M_{\lambda}\}_{\lambda\in\Lambda}$ a family of objects in $\mcA$.
	\begin{enumerate}
		\item\label{item:PropAtomSupportAndDirectSum} $\ASupp\bigoplus_{\lambda\in\Lambda}M_{\lambda}=\bigcup_{\lambda\in\Lambda}\ASupp M_{\lambda}$.
		\item\label{item:PropAssociatedAtomsAndDirectSum} \textnormal{(\cite[Proposition 3.1]{Storrer})} $\AAss\bigoplus_{\lambda\in\Lambda}M_{\lambda}=\bigcup_{\lambda\in\Lambda}\AAss M_{\lambda}$.
	\end{enumerate}
\end{Prop}

\begin{proof}
	(\ref{item:PropAtomSupportAndDirectSum}) It is straightforward to see that
	\begin{equation*}
		\bigcup_{\lambda\in\Lambda}\ASupp M_{\lambda}\subset\ASupp\bigoplus_{\lambda\in\Lambda}M_{\lambda}.
	\end{equation*}
	For any $\alpha\in\ASupp\bigoplus_{\lambda\in\Lambda}M_{\lambda}$, there exists a monoform subobject $H$ of a quotient object $N$ of $\bigoplus_{\lambda\in\Lambda}M_{\lambda}$ such that $\overline{H}=\alpha$. Denote by $N_{\lambda}$ the image of $M_{\lambda}$ by the canonical epimorphism $\bigoplus_{\lambda\in\Lambda}M_{\lambda}\twoheadrightarrow N$. Let $\mcS$ be the set of finite subsets of $\Lambda$. By the definition of Grothendieck categories, we have
	\begin{equation*}
		\sum_{\Lambda'\in\mcS}\left(\left(\sum_{\lambda\in\Lambda'}N_{\lambda}\right)\cap H\right)=\left(\sum_{\Lambda'\in\mcS}\sum_{\lambda\in\Lambda'}N_{\lambda}\right)\cap H=N\cap H=H.
	\end{equation*}
	Hence there exists $\Lambda'\in\mcS$ such that $(\sum_{\lambda\in\Lambda'}N_{\lambda})\cap H\neq 0$. By Proposition \ref{Prop:PropertyOfMonoformObject} (\ref{item:PropSubobjectOfMonoformObjectIsMonoform}), $(\sum_{\lambda\in\Lambda'}N_{\lambda})\cap H$ is a monoform subobject of $N$ which is atom-equivalent to $H$. Since we have the diagram
	\begin{equation*}
		\left(\sum_{\lambda\in\Lambda'}N_{\lambda}\right)\cap H\hookrightarrow\sum_{\lambda\in\Lambda'}N_{\lambda}\twoheadleftarrow\bigoplus_{\lambda\in\Lambda'}M_{\lambda},
	\end{equation*}
	it holds that
	\begin{equation*}
		\alpha=\overline{H}\in\ASupp\bigoplus_{\lambda\in\Lambda'}M_{\lambda}=\bigcup_{\lambda\in\Lambda'}\ASupp M_{\lambda}\subset\bigcup_{\lambda\in\Lambda}\ASupp M_{\lambda}
	\end{equation*}
	by Proposition \ref{Prop:AtomSupportAndAssociatedAtomsAndShortExactSequence} (\ref{item:PropAtomSupportAndShortExactSequence}). Hence the statement follows.
	
	(\ref{item:PropAssociatedAtomsAndDirectSum}) This can be shown similarly by using Proposition \ref{Prop:AtomSupportAndAssociatedAtomsAndShortExactSequence} (\ref{item:PropAssociatedAtomsAndShortExactSequence}).
\end{proof}

For a commutative ring $R$, we identify $\ASpec(\Mod R)$ with $\Spec R$ by the correspondence in Proposition \ref{Prop:AtomSpectrumForCommutativeRing}. In this case, we show that atom supports and associated atoms coincide with supports and associated primes.

\begin{Prop}\label{Prop:AtomSupportAndAssociatedAtomsForCommutativeRing}
	Let $R$ be a commutative ring and $M$ an $R$-module.
	\begin{enumerate}
		\item\label{item:PropAtomSupportForCommutativeRing} $\ASupp M=\Supp M$.
		\item\label{item:PropAssociatedAtomsForCommutativeRing} $\AAss M=\Ass M$.
	\end{enumerate}
\end{Prop}

\begin{proof}
	Let $\mfp$ be a prime ideal of $R$. Since any nonzero submodule of $R/\mfp$ contains a submodule isomorphic to $R/\mfp$, we have $\overline{R/\mfp}\in\AAss M$ if and only if there exists a monomorphism $R/\mfp\hookrightarrow M$. Hence $\AAss M=\Ass M$.
	
	Assume that $\overline{R/\mfp}\in\ASupp M$. Then $R/\mfp$ is a subquotient of $M$. Since $(R/\mfp)_{\mfp}$ is a nonzero subquotient of $M_{\mfp}$, we have $\mfp\in\Supp M$.
	
	Conversely, assume that $\mfp\in\Supp M$. Since we have $M=\sum_{x\in M}xR$, there exists $y\in M$ such that $\mfp\in\Supp yR$. Let $\mfa=\Ann_{R}(y)$. Then it holds that $yR\cong R/\mfa$. Since we have
	\begin{eqnarray*}
		\Supp yR=\Supp\frac{R}{\mfa}=\{\mfq\in\Spec R\mid\mfa\subset\mfq\},
	\end{eqnarray*}
	it follows that $\mfa\subset\mfp$. We have the diagram
	\begin{eqnarray*}
		\frac{R}{\mfp}\twoheadleftarrow\frac{R}{\mfa}\cong yR\hookrightarrow M.
	\end{eqnarray*}
	Therefore $\overline{R/\mfp}\in\ASupp M$.
\end{proof}

By using atom supports, we show a property of monoform objects.

\begin{Prop}\label{Prop:MonoformObjectAndAtomSupport}
	Let $\mcA$ be an abelian category and $H$ a monoform object in $\mcA$. Then for any nonzero subobject $L$ of $H$, we have $\overline{H}\notin\ASupp(H/L)$.
\end{Prop}

\begin{proof}
	If we have $\overline{H}\in\ASupp(H/L)$, then there exists a subobject $L'$ of $H$ with $L\subset L'$ such that there exists a common nonzero subobject of $H$ and $H/L'$. This is a contradiction.
\end{proof}

In commutative ring theory, it is well known that the number of associated primes of a finitely generated module over a commutative noetherian ring is finite. We can show this type of property for associated atoms.

\begin{Prop}\label{Prop:NumberOfAssociatedAtoms}
	Let $\mcA$ be an abelian category.
	\begin{enumerate}
		\item\label{item:PropNumberOfAssociatedAtomsOfUniformObject} \textnormal{(\cite[Lemma 3.4]{Storrer})} Let $U$ be a uniform object in $\mcA$. Then $\AAss U$ has at most one element. In particular, for any monoform object $H$ in $\mcA$, we have $\AAss H=\{\overline{H}\}$.
		\item\label{item:PropNumberOfAssociatedAtomsOfNoetherianObject} Let $M$ be a nonzero noetherian object in $\mcA$. Then $\AAss M$ is a nonempty finite set.
	\end{enumerate}
\end{Prop}

\begin{proof}
	(\ref{item:PropNumberOfAssociatedAtomsOfUniformObject}) The former part holds since any monoform subobjects $H_{1}$ and $H_{2}$ of $U$ have the common nonzero subobject $H_{1}\cap H_{2}$. The latter part follows from Proposition \ref{Prop:PropertyOfMonoformObject} (\ref{item:PropMonoformObjectIsUniform}).
	
	(\ref{item:PropNumberOfAssociatedAtomsOfNoetherianObject}) \cite[Remark 3.6]{Kanda1}.
\end{proof}

For an abelian category $\mcA$, a subobject $L$ of an object $M$ in $\mcA$ is called \emph{essential} if for any nonzero subobject $L'$ of $M$, we have $L\cap L'\neq 0$. The following fact is also a generalization of that in commutative ring theory.

\begin{Prop}\label{Prop:AssociatedAtomsOfEssentialSubobject}
	Let $\mcA$ be an abelian category, $M$ an object in $\mcA$, and $L$ an essential subobject of $M$. Then we have $\AAss L=\AAss M$.
\end{Prop}

\begin{proof}
	$\AAss L\subset\AAss M$ is obvious. $\AAss L\supset\AAss M$ follows straightforwardly from Proposition \ref{Prop:PropertyOfMonoformObject} (\ref{item:PropSubobjectOfMonoformObjectIsMonoform}).
\end{proof}

%%%%%%%%%%%%%%%%%%%%%%%%%%%%%%%%%%%%%%%%%%%%%%%%%%%%%%%%%%%%%%%%%%%%%%%%%%%%%%%%
\section{Topological properties of atom spectra}
\label{sec:TopologicalPropertiesOfAtomSpectra}
%%%%%%%%%%%%%%%%%%%%%%%%%%%%%%%%%%%%%%%%%%%%%%%%%%%%%%%%%%%%%%%%%%%%%%%%%%%%%%%%

In this section, we show several topological properties of the atom spectrum of an abelian category. These are used in section \ref{sec:PartialOrdersOnAtomSpectra} in order to consider a partial order on the atom spectrum.

We recall the definition of the topology on the atom spectrum.

\begin{Def}\label{Def:TopologyOnAtomSpectrum}
	Let $\mcA$ be an abelian category. We say that a subclass $\Phi$ of $\ASpec\mcA$ is \emph{open} if for any $\alpha\in\Phi$, there exists a monoform object $H$ in $\mcA$ satisfying $\overline{H}=\alpha$ and $\ASupp H\subset\Phi$.
\end{Def}

As in \cite[Proposition 3.8]{Kanda1}, the family of all the open subclasses of $\ASpec\mcA$ satisfies the axioms of topology. This topology is described as follows.

\begin{Prop}\label{Prop:OpenBasisOfAtomSpectrum}
	Let $\mcA$ be an abelian category. Then the family
	\begin{equation*}
		\{\ASupp M\mid M\in\mcA\}
	\end{equation*}
	is an open basis of $\ASpec\mcA$. If $\mcA$ is a Grothendieck category, then this is the set of all the open subsets of $\ASpec\mcA$.
	
	Moreover, if $\mcA$ is a locally noetherian Grothendieck category, then the family
	\begin{equation*}
		\{\ASupp M\mid M\text{ is a noetherian object in }\mcA\}
	\end{equation*}
	is an open basis of $\ASpec\mcA$.
\end{Prop}

\begin{proof}
	This follows from \cite[Proposition 3.9]{Kanda1}, \cite[Proposition 5.6]{Kanda1}, and \cite[Proposition 5.3]{Kanda1}.
\end{proof}

Note that the atom spectrum of a locally noetherian Grothendieck category is homeomorphic to the \emph{Ziegler spectrum} described in \cite{Herzog} (or in \cite{Krause} as the \emph{spectrum}) as in \cite[Theorem 5.9]{Kanda1}.

For a commutative ring $R$, we can describe the topology on $\ASpec(\Mod R)$ defined above in terms of prime ideals of $R$. Recall that a subset $\Phi$ of $\Spec R$ is called \emph{closed under specialization} if for any prime ideals $\mfp$ and $\mfq$ of $R$ with $\mfp\subset\mfq$, the condition $\mfp\in\Phi$ implies $\mfq\in\Phi$.

\begin{Prop}[{\cite[Proposition 7.2]{Kanda1}}]\label{Prop:TopologyOnAtomSpectrumOfCommutativeRing}
	Let $R$ be a commutative ring and $\Phi$ a subset of $\Spec R$. Then $\Phi$ is an open subset of $\ASpec(\Mod R)$ if and only if $\Phi$ is closed under specialization.
\end{Prop}

A topological space $X$ is called \emph{Alexandroff} if the intersection of any family of open subsets of $X$ is also open (or equivalently, if the union of any family of closed subsets of $X$ is also closed). For any commutative ring $R$, the topological space $\ASpec(\Mod R)$ is Alexandroff by Proposition \ref{Prop:TopologyOnAtomSpectrumOfCommutativeRing}. However, the atom spectrum of an abelian category is not necessarily an Alexandroff space.

\begin{ex}[{\cite[Example 4.7]{Pappacena}}]\label{ex:AtomSpectrumOfCategoryOfGradedModules}
	Let $k$ be a field. We regard $k[x]$ as a graded ring with $\deg x=1$ and consider the locally noetherian Grothendieck category $\GrMod k[x]$ of graded $k[x]$-modules. Let $H=k[x]$ and $S=k[x]/(x)$ as graded $k[x]$-modules. Then these are monoform objects in $\GrMod k[x]$, and we have
	\begin{equation*}
		\ASpec(\GrMod k[x])=\{\overline{H}\}\cup\{\overline{S(j)}\mid j\in\mbZ\}.
	\end{equation*}
	Since
	\begin{equation*}
		\ASupp H(i)=\{\overline{H}\}\cup\{\overline{S(j)}\mid j\in\mbZ,\ j\leq i\},
	\end{equation*}
	we have
	\begin{equation*}
		\bigcap_{i\in\mbZ}\ASupp H(i)=\{\overline{H}\},
	\end{equation*}
	where $\ASupp H(i)$ is an open subset of $\ASpec(\GrMod k[x])$ for any $i\in\mbZ$. Since any nonzero subobject of $H$ has a simple subquotient, the subset $\{\overline{H}\}$ of $\ASpec(\GrMod k[x])$ is not open. Therefore $\ASpec(\GrMod k[x])$ is not an Alexandroff space.
\end{ex}

Recall that a topological space $X$ is called a \emph{Kolmogorov space} (or a $T_{0}$\emph{-space}) if for any distinct points $x_{1}$ and $x_{2}$ in $X$, there exists an open subset of $X$ containing exactly one of them.

\begin{Prop}\label{Prop:AtomSpectrumIsKolmogorovSpace}
	Let $\mcA$ be an abelian category. Then $\ASpec\mcA$ is a Kolmogorov space.
\end{Prop}

\begin{proof}
	Assume that there exist distinct atoms $\alpha_{1}$ and $\alpha_{2}$ such that no open subclass of $\ASpec\mcA$ contains exactly one of them. Let $H$ be a monoform object such that $\overline{H}=\alpha_{1}$. Then $\ASupp H$ is an open subclass of $\ASpec\mcA$ containing $\alpha_{1}$. By the assumption, we have $\alpha_{2}\in\ASupp H$, and hence there exists a subobject $L$ of $H$ such that $\alpha_{2}\in\AAss(H/L)$. Since $\AAss H=\{\alpha_{1}\}$ by Proposition \ref{Prop:NumberOfAssociatedAtoms} (\ref{item:PropNumberOfAssociatedAtomsOfUniformObject}), we have $L\neq 0$. Similarly, $\alpha_{2}\in\ASupp(H/L)$ implies that there exists a subobject $L'$ of $H$ satisfying $L\subset L'$ and $\alpha_{1}\in\AAss(H/L')$. This contradicts Proposition \ref{Prop:MonoformObjectAndAtomSupport}.
\end{proof}

A topological space $X$ is called \emph{compact} if any open cover of $X$ admits a finite subcover. In the case of locally noetherian Grothendieck categories, open compact subsets of $\ASpec\mcA$ are characterized as follows.

\begin{Prop}[{Herzog \cite[Corollary 3.9]{Herzog} and Krause \cite[Corollary 4.5]{Krause}}]\label{Prop:CharacterizationOfOpenCompactSubset}
	Let $\mcA$ be a locally noetherian Grothendieck category. Then a subset $\Phi$ of $\ASpec\mcA$ is open and compact if and only if there exists a noetherian object $M$ in $\mcA$ such that $\Phi=\ASupp M$.
\end{Prop}

\begin{proof}
	Herzog \cite{Herzog} and Krause \cite{Krause} showed more general results for locally coherent Grothendieck category by using a classification of Serre subcategories. We give a direct proof in the case of locally noetherian Grothendieck categories for the convenience of the reader.
	
	Assume that $\Phi$ is open and compact. By Proposition \ref{Prop:OpenBasisOfAtomSpectrum}, there exists a family $\{L_{\lambda}\}_{\lambda\in\Lambda}$ of noetherian objects in $\mcA$ such that $\Phi=\bigcup_{\lambda\in\Lambda}\ASupp L_{\lambda}$, where $\ASupp L_{\lambda}$ is open for any $\lambda\in\Lambda$. Since $\Phi$ is compact, there exist $\lambda_{1},\ldots,\lambda_{n}\in\Lambda$ such that $\Phi=\bigcup_{i=1}^{n}\ASupp L_{\lambda_{i}}$. By Proposition \ref{Prop:AtomSupportAndAssociatedAtomsAndShortExactSequence} (\ref{item:PropAtomSupportAndShortExactSequence}), we have $\Phi=\ASupp\bigoplus_{i=1}^{n}L_{\lambda_{i}}$, where $\bigoplus_{i=1}^{n}L_{\lambda_{i}}$ is noetherian.
	
	Conversely, assume that $\Phi=\ASupp M$ for a noetherian object $M$ in $\mcA$. Then $\Phi$ is open in $\ASpec\mcA$. We can assume that $M$ is nonzero. Let $\{\Phi_{\lambda}\}_{\lambda\in\Lambda}$ be an open cover of $\Phi$. By Proposition \ref{Prop:NumberOfAssociatedAtoms} (\ref{item:PropNumberOfAssociatedAtomsOfNoetherianObject}), there exists a monoform subobject $H$ of $M$. Since we have $\overline{H}\in\ASupp M=\bigcup_{\lambda\in\Lambda}\Phi_{\lambda}$, there exists $\lambda_{1}\in\Lambda$ such that $\overline{H}\in\Phi_{\lambda_{1}}$. Since $\Phi_{\lambda_{1}}$ is an open subset of $\ASpec\mcA$, there exists a monoform object $H'$ in $\mcA$ satisfying $\ASupp H'\subset\Phi_{\lambda_{1}}$ and $\overline{H'}=\overline{H}$. Hence there exists a nonzero subobject $H_{1}$ of $H$ which is isomorphic to a subobject of $H'$. Let $M_{0}=M$ and $M_{1}=M_{0}/H_{1}$. Since $M$ is noetherian, inductively we obtain objects $M_{0},\ldots,M_{n}$ in $\mcA$, monoform objects $H_{1},\ldots,H_{n}$ in $\mcA$, and $\lambda_{1},\ldots,\lambda_{n}\in\Lambda$ with $\ASupp H_{i}\subset\Phi_{\lambda_{i}}$ and an exact sequence
	\begin{equation*}
		0\to H_{i}\to M_{i-1}\to M_{i}\to 0
	\end{equation*}
	for each $i=1,\ldots,n$, where $M_{0}=M$ and $M_{n}=0$. Then we have
	\begin{equation*}
		\Phi=\ASupp M=\bigcup_{i=1}^{n}\ASupp H_{i}\subset\bigcup_{i=1}^{n}\Phi_{\lambda_{i}}
	\end{equation*}
	This shows that $\Phi$ is compact.
\end{proof}

The discreteness of the atom spectrum is characterized as follows.

\begin{Prop}\label{Prop:CharacterizationOfOpenAtomAndDiscreteAtomSpectrum}
	Let $\mcA$ be a locally noetherian Grothendieck category.
	\begin{enumerate}
		\item\label{item:PropCharacterizationOfOpenAtom} For any atom $\alpha$ in $\mcA$, the subset $\{\alpha\}$ of $\ASpec\mcA$ is open if and only if there exists a simple object $S$ in $\mcA$ such that $\overline{S}=\alpha$.
		\item\label{item:PropCharacterizationOfDiscreteAtomSupport} For any noetherian object $M$ in $\mcA$, the subset $\ASupp M$ of $\ASpec\mcA$ with the induced topology is discrete if and only if $M$ has finite length.
		\item\label{item:PropCharacterizationOfDiscreteAtomSpectrum} $\ASpec\mcA$ is a discrete topological space if and only if any noetherian object in $\mcA$ has finite length.
	\end{enumerate}
\end{Prop}

\begin{proof}
	(\ref{item:PropCharacterizationOfOpenAtom}) Since any nonzero object in $\mcA$ has a simple subquotient, any nonempty open subset of $\ASpec\mcA$ contains an atom represented by a simple object in $\mcA$. For a simple object $S$ in $\mcA$, it is clear that $\{\overline{S}\}$ is an open subset of $\mcA$.
	
	(\ref{item:PropCharacterizationOfDiscreteAtomSupport}) Assume that $\ASupp M$ is discrete. Then by (\ref{item:PropCharacterizationOfOpenAtom}), any element of $\ASupp M$ is represented by a simple object in $\mcA$. By Proposition \ref{Prop:NumberOfAssociatedAtoms} (\ref{item:PropNumberOfAssociatedAtomsOfNoetherianObject}), any quotient object of $M$ has a simple subobject. Therefore $M$ has finite length.
	
	Conversely, assume that $M$ has finite length. Then by Proposition \ref{Prop:AtomSupportAndAssociatedAtomsAndShortExactSequence} (\ref{item:PropAtomSupportAndShortExactSequence}), any element of $\ASupp M$ is represented by a simple object in $\mcA$, and hence $\ASupp M$ is discrete.
	
	(\ref{item:PropCharacterizationOfDiscreteAtomSpectrum}) Assume that $\ASpec\mcA$ is discrete. Then by (\ref{item:PropCharacterizationOfDiscreteAtomSupport}), any noetherian object in $\mcA$ has finite length.
	
	Conversely, assume that any noetherian object in $\mcA$ has finite length. For any atom $\alpha$ in $\mcA$, there exists a noetherian monoform object $H$ in $\mcA$ such that $\overline{H}=\alpha$. Since $H$ has finite length by the assumption, there exists a simple subobject $S$ of $H$, and hence $\alpha=\overline{S}$. By (\ref{item:PropCharacterizationOfOpenAtom}), the topological space $\ASpec\mcA$ is discrete.
\end{proof}

%%%%%%%%%%%%%%%%%%%%%%%%%%%%%%%%%%%%%%%%%%%%%%%%%%%%%%%%%%%%%%%%%%%%%%%%%%%%%%%%
\section{Partial orders on atom spectra}
\label{sec:PartialOrdersOnAtomSpectra}
%%%%%%%%%%%%%%%%%%%%%%%%%%%%%%%%%%%%%%%%%%%%%%%%%%%%%%%%%%%%%%%%%%%%%%%%%%%%%%%%

In this section, we define a partial order on the atom spectrum of an abelian category. It is a generalization of the inclusion relation between prime ideals of a commutative ring. We show the existence of a maximal atom and that of a minimal atom under some conditions of noetherianness.

For a Kolmogorov space $X$, the partial order called \emph{specialization order} $\preceq$ on $X$ is defined as follows: for any $x,y\in X$, we have $x\preceq y$ if and only if $x$ belongs to the topological closure $\overline{\{y\}}$ of $\{y\}$ in $X$. On the other hand, each partially ordered set $P$ can be regarded as a topological space as follows: a subset $\Phi$ of $P$ is open if and only if for any $p,q\in P$ with $p\preceq q$, the condition $p\in\Phi$ implies $q\in\Phi$. The following results are well known.

\begin{Prop}\label{Prop:BijectionBetweenKolmogorovAlexandroffSpacesAndPartiallyOrderedSets}\leavevmode
	\begin{enumerate}
		\item\label{item:PropBijectionBetweenKolmogorovAlexandroffSpacesAndPartiallyOrderedSets} The map sending a Kolmogorov Alexandroff space $X$ to the partially ordered set $(X,\preceq)$, where $\preceq$ is the specialization order on $X$, is a bijection between homeomorphism classes of Kolmogorov Alexandroff spaces and isomorphism classes of partially ordered sets. The inverse map sends a partially ordered set $P$ to the topological space $P$ defined above.
		\item\label{item:PropBijectionBetweenFiniteKolmogorovSpacesAndFinitePartiallyOrderedSets} The maps in (\ref{item:PropBijectionBetweenKolmogorovAlexandroffSpacesAndPartiallyOrderedSets}) induce a bijection between homeomorphism classes of finite Kolmogorov spaces and isomorphism classes of finite partially ordered sets.
	\end{enumerate}
\end{Prop}

\begin{proof}
	(\ref{item:PropBijectionBetweenKolmogorovAlexandroffSpacesAndPartiallyOrderedSets}) This can be shown straightforwardly.
	
	(\ref{item:PropBijectionBetweenFiniteKolmogorovSpacesAndFinitePartiallyOrderedSets}) This is obvious since any finite topological space is Alexandroff.
\end{proof}

For an abelian category $\mcA$, the topological space $\ASpec\mcA$ is Kolmogorov by Proposition \ref{Prop:AtomSpectrumIsKolmogorovSpace}. We regard $\ASpec\mcA$ as a partially ordered set together with the specialization order $\leq$ on $\ASpec\mcA$. This partial order is described explicitly as follows.

\begin{Prop}\label{Prop:TopologicalCharacterizationOfPartialOrder}
	Let $\mcA$ be an abelian category, and let $\alpha$ and $\beta$ be atoms in $\mcA$. Then the following assertions are equivalent.
	\begin{enumerate}
		\item $\alpha\leq\beta$, that is, $\alpha\in\overline{\{\beta\}}$.
		\item Any open subclass $\Phi$ of $\ASpec\mcA$ containing $\alpha$ also contains $\beta$. In other words, the atom $\beta$ belongs to the intersection of all the open subclasses containing $\alpha$.
		\item For any object $M$ in $\mcA$ satisfying $\alpha\in\ASupp M$, we have $\beta\in\ASupp M$.
		\item For any monoform object $H$ in $\mcA$ satisfying $\overline{H}=\alpha$, we have $\beta\in\ASupp H$.
	\end{enumerate}
\end{Prop}

\begin{proof}
	This follows from the definition of the topology on $\ASpec\mcA$ and Proposition \ref{Prop:OpenBasisOfAtomSpectrum}.
\end{proof}

The following result shows that the partial order defined above is a generalization of the inclusion relation between prime ideals of a commutative ring.

\begin{Prop}\label{Prop:PartialOrderBetweenAtomOfCommutativeRing}
	Let $R$ be a commutative ring, and let $\mfp$ and $\mfq$ be prime ideals of $R$. Then $\overline{R/\mfp}\leq\overline{R/\mfq}$ if and only if $\mfp\subset\mfq$.
\end{Prop}

\begin{proof}
	This follows from Proposition \ref{Prop:TopologyOnAtomSpectrumOfCommutativeRing} and Proposition \ref{Prop:TopologicalCharacterizationOfPartialOrder}.
\end{proof}

Let $\mcA$ be an abelian category. We say that an atom is \emph{maximal} (resp.\ \emph{minimal}) in a subclass $\Phi$ of $\ASpec\mcA$ if it is maximal (resp.\ minimal) in $\Phi$ with respect to the partial order $\leq$. Such an atom is characterized in terms of the topology by the following proposition.

\begin{Prop}\label{Prop:TopologicalCharacterizationOfMaximalAtomAndMinimalAtom}
	Let $X$ be a Kolmogorov space with the specialization order $\preceq$. Let $x$ be an element of $X$ and $\Phi$ a subclass of $X$ containing $x$.
	\begin{enumerate}
		\item\label{item:PropTopologicalCharacterizationOfMaximalAtom} $x$ is maximal in $\Phi$ if and only if $\{x\}$ is the intersection of some family of open subclasses of $\Phi$.
		\item\label{item:PropTopologicalCharacterizationOfMinimalAtom} $x$ is minimal in $\Phi$ if and only if $\{x\}$ is a closed subclass of $\Phi$.
	\end{enumerate}
\end{Prop}

\begin{proof}
	This can be shown straightforwardly.
\end{proof}

\begin{rem}\label{rem:MinimalAtomIsNotNecessarilySimpleAtom}
	For a commutative ring $R$, it is clear that a prime ideal $\mfp$ of $R$ is maximal if and only if $R/\mfp$ is a simple $R$-module.
	
	Let $S$ be a simple object in an abelian category $\mcA$. Since $\{\overline{S}\}$ is an open subclass of $\ASpec\mcA$, the atom $\overline{S}$ is maximal by Proposition \ref{Prop:TopologicalCharacterizationOfMaximalAtomAndMinimalAtom} (\ref{item:PropTopologicalCharacterizationOfMaximalAtom}). However, a maximal atom is not necessarily represented by a simple object in general. In Example \ref{ex:AtomSpectrumOfCategoryOfGradedModules}, $\overline{H}$ is maximal in $\ASpec\mcA$ but is not represented by a simple object in $\GrMod k[x]$.
\end{rem}

In the case of a locally noetherian Grothendieck category, the existence of a maximal atom follows from the next proposition. We say that a partially ordered set $P$ satisfies the \emph{ascending chain condition} if any ascending chain $p_{0}\leq p_{1}\leq\cdots$ in $P$ eventually stabilizes. This is equivalent to that any subset of $P$ has a maximal element. The \emph{descending chain condition} is defined similarly.

\begin{Prop}\label{Prop:ExistenceOfMaximalAtom}
	Let $\mcA$ be a locally noetherian Grothendieck category. Then $\ASpec\mcA$ satisfies the ascending chain condition.
\end{Prop}

\begin{proof}
	Let $\alpha_{0}\leq\alpha_{1}\leq\cdots$ be an ascending chain of atoms in $\mcA$, and let $H$ be a noetherian monoform object such that $\overline{H}=\alpha_{0}$. Since $\ASupp H$ is an open subset of $\ASpec\mcA$, we have $\alpha_{1}\in\ASupp H$ by Proposition \ref{Prop:TopologicalCharacterizationOfPartialOrder}. Then there exists a subobject $L_{1}$ of $H$ such that $\alpha_{1}\in\AAss(H/L_{1})$. We have $\alpha_{1}\in\ASupp(H/L_{1})$. Inductively we obtain an ascending chain $0=L_{0}\subset L_{1}\subset L_{2}\subset\cdots$ of subobjects of $H$ such that $\alpha_{i}\in\AAss(H/L_{i})$ for any $i\in\mbZ_{\geq 0}$. Since $H$ is noetherian, there exists $n\in\mbZ_{\geq 0}$ such that $L_{n}=L_{n+1}=\cdots$. Then by $\{\alpha_{n},\alpha_{n+1},\ldots\}\subset\AAss(H/L_{n})$ and Proposition \ref{Prop:NumberOfAssociatedAtoms} (\ref{item:PropNumberOfAssociatedAtomsOfNoetherianObject}), the set $\{\alpha_{n},\alpha_{n+1},\ldots\}$ is finite. Since $\leq$ is a partial order on $\ASpec\mcA$, this implies that the chain $\alpha_{0}\leq\alpha_{1}\leq\cdots$ eventually stabilizes.
\end{proof}

We show the existence of minimal elements of certain subsets of the atom spectrum of a locally noetherian Grothendieck category.

\begin{Prop}\label{Prop:ExistenceOfMinimalAtomInOpenCompactSubset}
	Let $\mcA$ be a locally noetherian Grothendieck category, $M$ a noetherian object in $\mcA$, and $\alpha$ an atom in $\mcA$ belonging to $\ASupp M$. Then there exists a minimal element $\beta$ of $\ASupp M$ such that $\beta\leq\alpha$.
\end{Prop}

\begin{proof}
	By Proposition \ref{Prop:CharacterizationOfOpenCompactSubset}, the subset $\ASupp M$ of $\ASpec\mcA$ is compact. Hence we can show that there exists a minimal closed subset $\Psi$ of $\ASupp M$ contained by $\overline{\{\alpha\}}\cap\ASupp M$ by using Zorn's lemma. Since $\ASupp M$ is a Kolmogorov space by Proposition \ref{Prop:AtomSpectrumIsKolmogorovSpace}, the set $\Psi$ consists of an atom $\beta$ in $\mcA$ such that $\{\beta\}$ is a closed subset of $\ASupp M$. By Proposition \ref{Prop:TopologicalCharacterizationOfMaximalAtomAndMinimalAtom} (\ref{item:PropTopologicalCharacterizationOfMinimalAtom}), the atom $\beta$ is minimal in $\ASupp M$. By the definition of the specialization order, we have $\beta\leq\alpha$.
\end{proof}

If a locally noetherian Grothendieck category $\mcA$ has a noetherian object $M$ such that $\ASpec\mcA=\ASupp M$, the existence of a minimal atom in $\mcA$ follows from Proposition \ref{Prop:ExistenceOfMinimalAtomInOpenCompactSubset}. In Proposition \ref{Prop:LocallyNoetherianGrothendieckCategoryWithNoMinimalAtom}, we show that there exists a nonzero locally noetherian Grothendieck category having no minimal atom. Furthermore, it is shown in Proposition \ref{Prop:LocallyNoetherianGrothendieckCategoryWithNoDescendingChainCondition} that the descending chain condition on $\ASpec\mcA$ does not necessarily hold even in the case where $\ASpec\mcA=\ASupp M$.

%%%%%%%%%%%%%%%%%%%%%%%%%%%%%%%%%%%%%%%%%%%%%%%%%%%%%%%%%%%%%%%%%%%%%%%%%%%%%%%%
\section{Atom spectra of quotient categories}
\label{sec:AtomSpectraOfQuotientCategories}
%%%%%%%%%%%%%%%%%%%%%%%%%%%%%%%%%%%%%%%%%%%%%%%%%%%%%%%%%%%%%%%%%%%%%%%%%%%%%%%%

In this section, we describe the atom spectrum of the quotient category of a Grothendieck category by a localizing subcategory. We show an analogous result on the atom spectrum to the results \cite[Proposition 3.6]{Herzog} and \cite[Corollary 4.4]{Krause} on the Ziegler spectrum. Results in this section is used in section \ref{sec:Localization} in order to investigate the localization of Grothendieck category at an atom.

We recall the definition of Serre subcategories and the quotient categories by them.

\begin{Def}\label{Def:SerreSubcategory}
	Let $\mcA$ be an abelian category. A full subcategory $\mcX$ of $\mcA$ is called a \emph{Serre subcategory} of $\mcA$ if the following assertion holds: for any exact sequence
	\begin{equation*}
		0\to L\to M\to N\to 0
	\end{equation*}
	in $\mcA$, the object $M$ belongs to $\mcX$ if and only if $L$ and $N$ belong to $\mcX$.
\end{Def}

\begin{Def}\label{Def:QuotientCategory}
	Let $\mcA$ be an abelian category and $\mcX$ a Serre subcategory of $\mcA$. Then the \emph{quotient category} $\mcA/\mcX$ of $\mcA$ by $\mcX$ is defined as follows.
	\begin{enumerate}
		\item The objects of $\mcA/\mcX$ are those of $\mcA$.
		\item For objects $M$ and $N$ in $\mcA$,
		\begin{equation*}
			\Hom_{\mcA/\mcX}(M,N)=\varinjlim_{(M',N')\in\mcS_{M,N}}\Hom_{\mcA}(M',\frac{N}{N'}),
		\end{equation*}
		where $\mcS_{M,N}$ is the directed set defined by
		\begin{equation*}
			\mcS_{M,N}=\biggl\{(M',N')\biggm| M'\subset M,\ N'\subset N\text{ with }\frac{M}{M'},N'\in\mcX\biggr\}
		\end{equation*}
		and the relation: for $(M',N'),(M'',N'')\in\mcS_{M,N}$, $(M',N')\leq (M'',N'')$ if we have $M'\supset M''$ and $N'\subset N''$.
		\item Let $L$, $M$, and $N$ be objects in $\mcA$. For $[f]\in\Hom_{\mcA/\mcX}(L,M)$ and $[g]\in\Hom_{\mcA/\mcX}(M,N)$ represented by $f\colon L'\to M/M'$ and $g\colon M''\to N/N''$ in $\mcA$, where $(L',M')\in\mcS_{L,M}$ and $(M'',N'')\in\mcS_{M,N}$, the composite $[g][f]$ is represented by the composite
		\begin{equation*}
			L''\xrightarrow{f''}\frac{M'+M''}{M'}\cong\frac{M''}{M'\cap M''}\xrightarrow{g'}\frac{N}{N'},
		\end{equation*}
		where
		\begin{equation*}
			L''=f^{-1}\left(\frac{M'+M''}{M'}\right),\ \frac{N'}{N''}=g(M'\cap M''),
		\end{equation*}
		and $f''$ and $g'$ are the induced morphisms from $f$ and $g$, respectively.
	\end{enumerate}
\end{Def}

With the notation in Definition \ref{Def:QuotientCategory}, we can define a canonical additive functor by the correspondence $M\mapsto M$ for each object $M$ in $\mcA$ and the canonical map $\Hom_{\mcA}(M,N)\to\Hom_{\mcA/\mcX}(M,N)$ for each objects $M$ and $N$ in $\mcA$.

\begin{Def}\label{Def:LocalizingSubcategory}
	Let $\mcA$ be an abelian category. A Serre subcategory $\mcX$ of $\mcA$ is called a \emph{localizing subcategory} of $\mcA$ if the canonical functor $\mcA\to\mcA/\mcX$ has a right adjoint functor.
\end{Def}

It is known that localizing subcategories of a Grothendieck category are characterized as follows.

\begin{Prop}\label{Prop:CharacterizationOfLocalizingSubcategory}
	Let $\mcA$ be a Grothendieck category and $\mcX$ a Serre subcategory of $\mcA$. Then the following assertions are equivalent.
	\begin{enumerate}
		\item\label{item:PropBeingLocalizingSubcategory} $\mcX$ is a localizing subcategory of $\mcA$.
		\item\label{item:PropClosedUnderMaximalSubobject} For any object $M$ in $\mcA$, the set of subobjects of $M$ belonging to $\mcX$ has a maximal element.
		\item\label{item:PropClosedUnderGreatestSubobject} For any object $M$ in $\mcA$, the set of subobjects of $M$ belonging to $\mcX$ has a largest element.
		\item\label{item:PropClosedUnderDirectSum} $\mcX$ is closed under arbitrary direct sums.
	\end{enumerate}
\end{Prop}

\begin{proof}
	This follows from \cite[Proposition 4.5.2]{Popescu} and \cite[Proposition 4.6.3]{Popescu}. Note that $\mcX$ is closed under finite sums of subobjects since it is closed under quotient objects and extensions.
\end{proof}

By Proposition \ref{Prop:CharacterizationOfLocalizingSubcategory}, the localizing subcategories of a Grothendieck category $\mcA$ is the full subcategories closed under subobjects, quotient objects, extensions, and arbitrary direct sums. Hence for any full subcategory $\mcX$ of a Grothendieck category $\mcA$, there exists a smallest localizing subcategory $\mcY$ of $\mcA$ containing $\mcX$. In the case where $\mcX$ is closed under subobjects and quotient objects, such $\mcY$ is described as follows.

\begin{Prop}\label{Prop:LocalizingSubcategoryGeneratedBySubcategory}
	Let $\mcA$ be a Grothendieck category, $\mcX$ a full subcategory of $\mcA$ closed under subobjects and quotient objects, and $\mcY$ the smallest localizing subcategory of $\mcA$ containing $\mcX$. Then an object $M$ in $\mcA$ belongs to $\mcY$ if and only if any nonzero quotient object of $M$ has a nonzero subobject belonging to $\mcX$.
\end{Prop}

\begin{proof}
	This is shown in a similar way to the proof of \cite[Proposition 4.6.10]{Popescu}.
\end{proof}

We also use the following wider class of subcategories.

\begin{Def}\label{Def:PrelocalizingSubcategory}
	Let $\mcA$ be a Grothendieck category. A full subcategory $\mcX$ of $\mcA$ is called a \emph{prelocalizing subcategory} if it is closed under subobjects, quotient objects, and arbitrary direct sums.
\end{Def}

Note that any prelocalizing subcategory of a Grothendieck category is also a Grothendieck category. The smallest prelocalizing subcategory containing a given full subcategory is described as follows.

\begin{Prop}\label{Prop:PrelocalizingSubcategoryGeneratedBySubcategory}
	Let $\mcA$ be a Grothendieck category, $\mcX$ a full subcategory of $\mcA$, and $\mcY$ the smallest prelocalizing subcategory of $\mcA$ containing $\mcX$. Then an object $X$ in $\mcA$ belongs to $\mcY$ if and only if $X$ is a subquotient of the direct sum of some family of objects belonging to $\mcX$.
\end{Prop}

\begin{proof}
	It is enough to show that the full subcategory consisting of all the subquotients of direct sums of objects in $\mcX$ is closed under arbitrary direct sums. Let $\{M_{\lambda}\}_{\lambda\in\Lambda}$ be a family of objects in $\mcX$, and let $L_{\lambda}\subset L'_{\lambda}\subset M_{\lambda}$ be subobjects of $M_{\lambda}$ for each $\lambda\in\Lambda$. Then we have the diagram
	\begin{equation*}
		\bigoplus_{\lambda\in\Lambda}\frac{L'_{\lambda}}{L_{\lambda}}\twoheadleftarrow\bigoplus_{\lambda\in\Lambda}L'_{\lambda}\hookrightarrow\bigoplus_{\lambda\in\Lambda}M_{\lambda}
	\end{equation*}
	since direct sum is exact. Hence the claim follows.
\end{proof}

For a ring $R$ and a right $R$-module $M$, the smallest prelocalizing subcategory containing $M$ was investigated by Wisbauer \cite{Wisbauer} from the viewpoint of representation theory of rings.

The localizing subcategories of a locally coherent Grothendieck category were classified in \cite{Herzog} and \cite{Krause}. In \cite{Kanda1}, we stated this result for locally noetherian Grothendieck categories in terms of the atom spectrum by using the following maps.

\begin{Def}\label{Def:AtomSupportOfSubcategory}
	Let $\mcA$ be an abelian category.
	\begin{enumerate}
		\item\label{item:DefAtomSupportOfSubcategory} For a full subcategory $\mcX$ of $\mcA$, define a subclass $\ASupp\mcX$ of $\ASpec\mcA$ by
		\begin{equation*}
			\ASupp\mcX=\bigcup_{M\in\mcX}\ASupp M.
		\end{equation*}
		\item\label{item:DefInverseImageOfAtomSupport} For a subclass $\Phi$ of $\ASpec\mcA$, define a full subcategory $\ASupp^{-1}\Phi$ by
		\begin{equation*}
			\ASupp^{-1}\Phi=\{M\in\mcA\mid\ASupp M\subset\Phi\}.
		\end{equation*}
	\end{enumerate}
\end{Def}

We can show the following lemma for arbitrary abelian categories.

\begin{Lem}\label{Lem:OpenSubclassOfAtomSpectrumAndAtomSupport}
	Let $\mcA$ be an abelian category and $\Phi$ an open subclass of $\ASpec\mcA$. Then we have $\ASupp(\ASupp^{-1}\Phi)=\Phi$.
\end{Lem}

\begin{proof}
	This follows from the proof of \cite[Theorem 4.3]{Kanda1}.
\end{proof}

\begin{Thm}[{\cite[Theorem 3.8]{Herzog}, \cite[Corollary 4.3]{Krause}, and \cite[Theorem 5.5]{Kanda1}}]\label{Thm:LocalizingSubcategoryAndOpenSubset}
	Let $\mcA$ be a locally noetherian Grothendieck category. Then the map $\mcX\mapsto\ASupp\mcX$ is a bijection between localizing subcategories of $\mcA$ and open subsets of $\ASpec\mcA$. The inverse map is given by $\Phi\mapsto\ASupp^{-1}\Phi$.
\end{Thm}

Note that Theorem \ref{Thm:LocalizingSubcategoryAndOpenSubset} does not necessarily hold for a general Grothendieck category. Indeed, a nonzero Grothendieck category with no atom, which we construct in Theorem \ref{Thm:GrothendieckCategoryWithNoAtom}, is a counter-example.

The atom support of a localizing subcategory is also used in order to describe the atom spectra of quotient categories. We recall several fundamental properties of the quotient categories of Grothendieck categories by localizing subcategories.

\begin{Thm}\label{Thm:PropertyOfQuotientCategory}
	Let $\mcA$ be a Grothendieck category, $\mcX$ a localizing subcategory of $\mcA$, $F\colon\mcA\to\mcA/\mcX$ the canonical functor, and $G\colon\mcA/\mcX\to\mcA$ its right adjoint functor.
	\begin{enumerate}
		\item\label{item:ThmLocalizingSubcategoryAndQuotientSubcategoryAreGrothendieckCategory} $\mcX$ and $\mcA/\mcX$ are Grothendieck categories. If $\mcA$ is a locally noetherian Grothendieck category, so are $\mcX$ and $\mcA/\mcX$.
		\item\label{item:ThmPropertyOfCanonicalFunctorForQuotientCategory} The functor $F$ is dense and exact. The functor $G$ is fully faithful. The composite $FG\colon\mcA/\mcX\to\mcA/\mcX$ is isomorphic to the identity functor on $\mcA/\mcX$.
		\item\label{item:ThmDescriptionOfCompositeFunctorForQuotientCategory} Let $M$ be an object in $\mcA$. Then the object $GF(M)$ is described as follows.
		\begin{enumerate}
			\item Let $L$ be the largest subobject of $M$ belonging to $\mcX$ and $N=M/L$.
			\item Let $\widetilde{M}/N$ be the largest subobject of $E(N)/N$ belonging to $\mcX$, where $E(N)$ is the injective envelope of $N$.
		\end{enumerate}
		Then $\widetilde{M}$ is isomorphic to $GF(M)$. In particular, for any object $M'$ in $\mcA/\mcX$, any nonzero subobject of $G(M')$ does not belong to $\mcX$.
		\item\label{item:ThmCharacterizationOfZeroObjectInQuotientCategory} An object $M$ in $\mcA$ satisfies $F(M)=0$ if and only if $M$ belongs to $\mcX$.
	\end{enumerate}
\end{Thm}

\begin{proof}
	(\ref{item:ThmLocalizingSubcategoryAndQuotientSubcategoryAreGrothendieckCategory}) \cite[Proposition III.9]{Gabriel} and Corollary 1 of it.
	
	(\ref{item:ThmPropertyOfCanonicalFunctorForQuotientCategory}) \cite[Theorem 4.3.8]{Popescu} and \cite[Proposition 4.4.3 (1)]{Popescu}.
	
	(\ref{item:ThmDescriptionOfCompositeFunctorForQuotientCategory}) This follows from the proof of \cite[Theorem 4.4.5]{Popescu}.
	
	(\ref{item:ThmCharacterizationOfZeroObjectInQuotientCategory}) \cite[Lemma 4.3.4]{Popescu}.
\end{proof}

The following result gives a description of the atom spectra of (pre)localizing subcategories.

\begin{Prop}\label{Prop:AtomSpectraOfLocalizingSubcategory}
	Let $\mcA$ be an abelian category and $\mcX$ a full subcategory of $\mcA$ closed under subobjects, quotient objects, and finite direct sums. Then $\ASpec\mcX$ is homeomorphic to the open subclass $\ASupp\mcX$ of $\ASpec\mcA$ with the induced topology.
\end{Prop}

\begin{proof}
	Note that $\mcX$ is an abelian category. Since $\mcX$ is closed under subobjects and quotient objects, the canonical functor $\mcX\to\mcA$ sends any monoform object in $\mcX$ to a monoform object in $\mcA$, and this correspondence induces an injective map $\ASpec\mcX\to\ASpec\mcA$. It is obvious that the image of this map is $\ASupp\mcX$ and that this map sends any open subclass of $\ASpec\mcX$ to an open subclass of $\ASupp\mcX$. Let $\Phi$ be an open subclass of $\ASupp\mcX$. For any $\alpha\in\Phi$, there exists a monoform object in $\mcA$ satisfying $\overline{H}=\alpha$ and $\ASupp H\subset\Phi$. Since $\alpha\in\ASupp\mcX$, and $\mcX$ is closed under subobjects and quotient objects, there exists a monoform object $H'$ in $\mcX$ such that $\overline{H'}=\alpha$. We have a common nonzero subobject $H''$ of $H$ and $H'$. Then by Proposition \ref{Prop:PropertyOfMonoformObject} (\ref{item:PropSubobjectOfMonoformObjectIsMonoform}), $H''$ is a monoform object in $\mcX$ satisfying $\overline{H''}=\alpha$ and $\ASupp H''\subset\Phi$. This shows that the map $\ASpec\mcX\to\ASupp\mcX$ is a homeomorphism.
\end{proof}

We describe the atom spectra of quotient categories in Theorem \ref{Thm:AtomSpectraOfQuotientCategory}. We start with the following lemma for the proof of it.

\begin{Lem}\label{Lem:QuotientObjectInQuotientCategory}
	Let $\mcA$ be a Grothendieck category, $\mcX$ a localizing subcategory of $\mcA$, $F\colon\mcA\to\mcA/\mcX$ the canonical functor, and $G\colon\mcA/\mcX\to\mcA$ its right adjoint functor. For any object $M'$ in $\mcA/\mcX$ and any subobject $L'$ of $M'$, we have an exact sequence
	\begin{equation*}
		0\to\frac{G(M')}{G(L')}\to G\left(\frac{M'}{L'}\right)\to N\to 0,
	\end{equation*}
	in $\mcA$ such that $N$ is an object in $\mcA$ belonging to $\mcX$, and $G(M')/G(L')$ is essential as a subobject of $G(M'/L')$.
\end{Lem}

\begin{proof}
	Since $G$ is a left exact functor, we have the exact sequence
	\begin{equation*}
		0\to G(L')\to G(M')\to G\left(\frac{M'}{L'}\right),
	\end{equation*}
	and hence we have an exact sequence
	\begin{equation*}
		0\to\frac{G(M')}{G(L')}\to G\left(\frac{M'}{L'}\right)\to N\to 0,
	\end{equation*}
	where $N$ is an object in $\mcA$. By Theorem \ref{Thm:PropertyOfQuotientCategory} (\ref{item:ThmPropertyOfCanonicalFunctorForQuotientCategory}), we have
	\begin{equation*}
		F\left(\frac{G(M')}{G(L')}\right)\cong FG\left(\frac{M'}{L'}\right).
	\end{equation*}
	This implies $F(N)=0$. By Theorem \ref{Thm:PropertyOfQuotientCategory} (\ref{item:ThmCharacterizationOfZeroObjectInQuotientCategory}), the object $N$ belongs to $\mcX$.
	
	Let $M$ be a nonzero subobject of $G(M'/L')$. Denote the image of the composite $M\hookrightarrow G(M'/L')\to N$ by $B$. Then $B$ belongs to $\mcX$. We have the commutative diagram
	\begin{equation*}
		\newdir^{ (}{!/-5pt/@^{(}}
		\xymatrix{
			0\ar[r] & L\ar@{^{ (}->}[d]\ar[r] & M\ar@{^{ (}->}[d]\ar[r] & B\ar@{^{ (}->}[d]\ar[r] & 0 \\
			0\ar[r] & \dfrac{G(M')}{G(L')}\ar[r] & G\left(\dfrac{M'}{L'}\right)\ar[r] & N\ar[r] & 0\rlap{\,,}
		}
	\end{equation*}
	where $L$ is a subobject of $M$. If $L=0$, then $M$ is a nonzero subobject of $G(M'/L')$ which belongs to $\mcX$. This contradicts Theorem \ref{Thm:PropertyOfQuotientCategory} (\ref{item:ThmDescriptionOfCompositeFunctorForQuotientCategory}). Hence we have $L\neq 0$. This shows that the intersection of the subobjects $G(M')/G(L')$ and $M$ of $G(M'/L')$ is nonzero. Therefore $G(M')/G(L')$ is essential as a subobject of $G(M'/L')$.
\end{proof}

We show some results on monoform objects in quotient categories.

\begin{Lem}\label{Lem:MonoformObjectInQuotientCategory}
	Let $\mcA$ be a Grothendieck category, $\mcX$ a localizing subcategory of $\mcA$, $F\colon\mcA\to\mcA/\mcX$ the canonical functor, and $G\colon\mcA/\mcX\to\mcA$ its right adjoint functor.
	\begin{enumerate}
		\item\label{item:LemMonoformObjectInQuotientCategoryIsMonoformObject} For any monoform object $H'$ in $\mcA/\mcX$, the object $G(H')$ in $\mcA$ is monoform.
		\item\label{item:LemTorsionfreeMonoformObjectIsMonoformObjectInQuotientCategory} Let $H$ be a monoform object in $\mcA$ such that any nonzero subobject of $H$ does not belong to $\mcX$. Then $GF(H)$ is a monoform object in $\mcA$ which has a subobject isomorphic to $H$, and $F(H)$ is a monoform object in $\mcA/\mcX$.
	\end{enumerate}
\end{Lem}

\begin{proof}
	(\ref{item:LemMonoformObjectInQuotientCategoryIsMonoformObject}) Assume that $G(H')$ is not monoform. Then there exist a nonzero subobject $L$ of $G(H')$ and a common nonzero subobject $N$ of $G(H')$ and $G(H')/L$. By Theorem \ref{Thm:PropertyOfQuotientCategory} (\ref{item:ThmPropertyOfCanonicalFunctorForQuotientCategory}), $F$ is an exact functor, and $FG$ is isomorphic to the identity functor on $\mcA/\mcX$. Hence $F(L)$ is a subobject of $H'$, and $F(N)$ is a common subobject of $H'$ and $H'/F(L)$. By Theorem \ref{Thm:PropertyOfQuotientCategory} (\ref{item:ThmDescriptionOfCompositeFunctorForQuotientCategory}), neither $L$ nor $N$ belongs to $\mcX$, and hence $F(L)$ and $F(N)$ are nonzero by Theorem \ref{Thm:PropertyOfQuotientCategory} (\ref{item:ThmCharacterizationOfZeroObjectInQuotientCategory}). This contradicts the monoformness of $H'$.
	
	(\ref{item:LemTorsionfreeMonoformObjectIsMonoformObjectInQuotientCategory}) By Theorem \ref{Thm:PropertyOfQuotientCategory} (\ref{item:ThmDescriptionOfCompositeFunctorForQuotientCategory}), the object $H$ can be regarded as an essential subobject of $GF(H)$, and $GF(H)/H$ belongs to $\mcX$. Assume that $GF(H)$ is not monoform. Then there exist nonzero subobjects $L$ and $M$ of $GF(H)$ such that $M$ is isomorphic to a subobject of $GF(H)/L$. We have $L\cap H\neq 0$ and $M\cap H\neq 0$. By Proposition \ref{Prop:PropertyOfMonoformObject} (\ref{item:PropSubobjectOfMonoformObjectIsMonoform}), $M\cap H$ is a monoform object in $\mcA$ which is atom-equivalent to $H$. Hence we have $\overline{H}\in\ASupp(GF(H)/L)$. Since we have an exact sequence
	\begin{equation*}
		0\to\frac{H+L}{L}\to\frac{GF(H)}{L}\to\frac{GF(H)}{H+L}\to 0,
	\end{equation*}
	we have
	\begin{equation*}
		\ASupp\frac{GF(H)}{L}=\ASupp\frac{H+L}{L}\cup\ASupp\frac{GF(H)}{H+L}
	\end{equation*}
	by Proposition \ref{Prop:AtomSupportAndAssociatedAtomsAndShortExactSequence} (\ref{item:PropAtomSupportAndShortExactSequence}). By Proposition \ref{Prop:MonoformObjectAndAtomSupport}, we have
	\begin{equation*}
		\overline{H}\notin\ASupp\frac{H}{L\cap H}=\ASupp\frac{H+L}{L}.
	\end{equation*}
	Since $GF(H)/(H+L)$ is a quotient object of $GF(H)/H$, it belongs to $\mcX$. Hence we also have
	\begin{equation*}
		\overline{H}\notin\ASupp\frac{GF(H)}{H+L}.
	\end{equation*}
	This is a contradiction. Therefore $GF(H)$ is monoform.
	
	Assume that $F(H)$ is not a monoform object in $\mcA/\mcX$. Then there exist a nonzero subobject $L'$ of $F(H)$ and a nonzero subobject $M'/L'$ of $F(H)/L'$ which is isomorphic to a subobject $N'$ of $F(H)$. Since $G$ is a fully faithful left exact functor, we have the diagram
	\begin{equation*}
		GF(H)\supset G(N')\cong G\left(\frac{M'}{L'}\right)\subset G\left(\frac{F(H)}{L'}\right),
	\end{equation*}
	where $G(L')$ and $G(N')$ are nonzero. Since $G(N')$ is a monoform object in $\mcA$ which is atom-equivalent to $GF(H)$, we have
	\begin{equation*}
		\overline{GF(H)}=\overline{G(N')}\in\AAss G\left(\frac{F(H)}{L'}\right).
	\end{equation*}
	Since $GF(H)/G(L')$ can be regarded as an essential subobject of $G(F(H)/L')$ by Lemma \ref{Lem:QuotientObjectInQuotientCategory}, we have
	\begin{equation*}
		\AAss G\left(\frac{F(H)}{L'}\right)=\AAss\frac{GF(H)}{G(L')}
	\end{equation*}
	by Proposition \ref{Prop:AssociatedAtomsOfEssentialSubobject}. This contradicts the monoformness of $GF(H)$ by Proposition \ref{Prop:MonoformObjectAndAtomSupport}. Therefore $F(H)$ is a monoform object in $\mcA/\mcX$.
\end{proof}

We describe the atom spectra of the quotient categories as sets.

\begin{Lem}\label{Lem:AtomInQuotientCategory}
	Let $\mcA$ be a Grothendieck category and $\mcX$ a localizing subcategory of $\mcA$. Then there exists a bijection between $\ASpec(\mcA/\mcX)$ and $\ASpec\mcA\setminus\ASupp\mcX$.
\end{Lem}

\begin{proof}
	Let $F\colon\mcA\to\mcA/\mcX$ be the canonical functor and $G\colon\mcA/\mcX\to\mcA$ its right adjoint functor.
	
	For each monoform object $H$ in $\mcA$ satisfying $\overline{H}\in\ASpec\mcA\setminus\ASupp\mcX$, any nonzero subobject of $H$ does not belong to $\mcX$. Hence by Lemma \ref{Lem:MonoformObjectInQuotientCategory} (\ref{item:LemTorsionfreeMonoformObjectIsMonoformObjectInQuotientCategory}), $F(H)$ is a monoform object in $\mcA/\mcX$, and we obtain the atom $\overline{F(H)}$ in $\mcA/\mcX$.
	
	For each monoform object $H'$ in $\mcA/\mcX$, by Lemma \ref{Lem:MonoformObjectInQuotientCategory} (\ref{item:LemMonoformObjectInQuotientCategoryIsMonoformObject}), $G(H')$ is a monoform object in $\mcA$. By Theorem \ref{Thm:PropertyOfQuotientCategory} (\ref{item:ThmDescriptionOfCompositeFunctorForQuotientCategory}), any nonzero subobject of $G(H')$ does not belong to $\mcX$. Hence we obtain the atom $\overline{G(H')}$ in $\mcA$ belonging to $\ASpec\mcA\setminus\ASupp\mcX$.
	
	The above two operations induce maps
	\begin{equation*}
		\overline{F}\colon\ASpec\mcA\setminus\ASupp\mcX\to\ASpec(\mcA/\mcX)
	\end{equation*}
	and
	\begin{equation*}
		\overline{G}\colon\ASpec(\mcA/\mcX)\to\ASpec\mcA\setminus\ASupp\mcX.
	\end{equation*}
	Since $FG$ is isomorphic to the identity functor on $\mcA/\mcX$ by Theorem \ref{Thm:PropertyOfQuotientCategory} (\ref{item:ThmPropertyOfCanonicalFunctorForQuotientCategory}), the composite $\overline{F}\,\overline{G}$ is the identity map on $\ASpec(\mcA/\mcX)$. Since the above monoform object $H$ in $\mcA$ is a nonzero subobject of $GF(H)$, we have $\overline{G}\,\overline{F}(\overline{H})=\overline{GF(H)}=\overline{H}$.
\end{proof}

In order to show that the bijection in Lemma \ref{Lem:AtomInQuotientCategory} is a homeomorphism, we show the next results on atom supports and associated atoms in quotient categories. For a Grothendieck category $\mcA$ and a localizing subcategory $\mcX$ of $\mcA$, we identify $\ASpec(\mcA/\mcX)$ with $\ASpec\mcA\setminus\ASupp\mcX$ by the correspondence in Lemma \ref{Lem:AtomInQuotientCategory}.

\begin{Lem}\label{Lem:AtomSupportAndAssociatedAtomsForQuotientCategory}
	Let $\mcA$ be a Grothendieck category, $\mcX$ a localizing subcategory of $\mcA$, $F\colon\mcA\to\mcA/\mcX$ the canonical functor, and $G\colon\mcA/\mcX\to\mcA$ its right adjoint functor.
	\begin{enumerate}
		\item\label{item:LemAtomSupportAndAssociatedAtomsForQuotientCategory} For any object $M'$ in $\mcA/\mcX$, we have
		\begin{equation*}
			\ASupp M'=\ASupp G(M')\setminus\ASupp\mcX
		\end{equation*}
		and
		\begin{equation*}
			\AAss M'=\AAss G(M').
		\end{equation*}
		\item\label{item:LemAtomSupportAndQuotientCategory} For any object $M$ in $\mcA$, we have
		\begin{equation*}
			\ASupp GF(M)\setminus\ASupp\mcX=\ASupp M\setminus\ASupp\mcX
		\end{equation*}
		and
		\begin{equation*}
			\ASupp F(M)=\ASupp M\setminus\ASupp\mcX.
		\end{equation*}
	\end{enumerate}
\end{Lem}

\begin{proof}
	(\ref{item:LemAtomSupportAndAssociatedAtomsForQuotientCategory}) By Theorem \ref{Thm:PropertyOfQuotientCategory} (\ref{item:ThmPropertyOfCanonicalFunctorForQuotientCategory}) and Lemma \ref{Lem:MonoformObjectInQuotientCategory}, we have
	\begin{equation*}
		\AAss M'=\AAss G(M')\setminus\ASupp\mcX
	\end{equation*}
	and
	\begin{equation*}
		\ASupp M'\supset\ASupp G(M')\setminus\ASupp\mcX.
	\end{equation*}
	By Theorem \ref{Thm:PropertyOfQuotientCategory} (\ref{item:ThmDescriptionOfCompositeFunctorForQuotientCategory}), we have
	\begin{equation*}
		\AAss G(M')\cap\ASupp\mcX=\emptyset.
	\end{equation*}
	For any $\alpha\in\ASupp M'$, there exist a subobject $L'$ of $M'$ and a monoform subobject $H'$ of $M'/L'$ such that $\overline{H'}=\alpha$. By Lemma \ref{Lem:QuotientObjectInQuotientCategory} and Proposition \ref{Prop:AssociatedAtomsOfEssentialSubobject}, we have
	\begin{equation*}
		\overline{G(H')}\in\AAss G\left(\frac{M'}{L'}\right)=\AAss\frac{G(M')}{G(L')}\subset\ASupp G(M').
	\end{equation*}
	This shows that
	\begin{equation*}
		\ASupp M'\subset\ASupp G(M')\setminus\ASupp\mcX.
	\end{equation*}
	
	(\ref{item:LemAtomSupportAndQuotientCategory}) Let $L$ be the largest subobject of $M$ belonging to $\mcX$ and $N=M/L$. Take the largest subobject $\widetilde{M}/N$ of $E(N)/N$ belonging to $\mcX$. By Theorem \ref{Thm:PropertyOfQuotientCategory} (\ref{item:ThmDescriptionOfCompositeFunctorForQuotientCategory}), $GF(M)$ is isomorphic to $\widetilde{M}$. We have the exact sequences
	\begin{equation*}
		0\to L\to M\to N\to 0,
	\end{equation*}
	and
	\begin{equation*}
		0\to N\to\widetilde{M}\to\frac{\widetilde{M}}{N}\to 0.
	\end{equation*}
	By Proposition \ref{Prop:AtomSupportAndAssociatedAtomsAndShortExactSequence} (\ref{item:PropAtomSupportAndShortExactSequence}), we have
	\begin{align*}
		\ASupp GF(M)\setminus\ASupp\mcX
		&=\ASupp\widetilde{M}\setminus\ASupp\mcX\\
		&=\ASupp N\setminus\ASupp\mcX\\
		&=\ASupp M\setminus\ASupp\mcX.
	\end{align*}
	By (\ref{item:LemAtomSupportAndAssociatedAtomsForQuotientCategory}), we obtain
	\begin{equation*}
		\ASupp F(M)=\ASupp GF(M)\setminus\ASupp\mcX=\ASupp M\setminus\ASupp\mcX.\qedhere
	\end{equation*}
\end{proof}

\begin{Thm}\label{Thm:AtomSpectraOfQuotientCategory}
	Let $\mcA$ be a Grothendieck category and $\mcX$ a localizing subcategory of $\mcA$. Then $\ASpec(\mcA/\mcX)$ is homeomorphic to the closed subset $\ASpec\mcA\setminus\ASupp\mcX$ of $\ASpec\mcA$ with the induced topology.
\end{Thm}

\begin{proof}
	This follows from Lemma \ref{Lem:AtomInQuotientCategory} and Lemma \ref{Lem:AtomSupportAndAssociatedAtomsForQuotientCategory}.
\end{proof}

It is known that any Grothendieck category $\mcA$ can be obtained as the quotient category of the category of modules over some ring $R$ by some localizing subcategory. Since we have a fully faithful functor $\mcA\to\Mod R$, this result is called the \emph{Gabriel--Popescu embedding}.

\begin{Thm}[Gabriel and Popescu {\cite[Proposition]{PopescoGabriel}}]\label{Thm:GabrielPopescuEmbedding}
	Let $\mcA$ be a Grothendieck category, $G$ a generator of $\mcA$, and $R=\End_{\mcA}(G)$. Then there exists a localizing subcategory $\mcX$ of $\Mod R$ such that $\mcA$ is equivalent to $(\Mod R)/\mcX$.
\end{Thm}

Therefore we deduce the following result on the atom spectra of Grothendieck categories.

\begin{Cor}\label{Cor:AtomSpectrumOfGrothendieckCategoryIsClosedSubsetOfAtomSpectrumOfModuleCategory}
	For any Grothendieck category $\mcA$, there exists a ring $R$ such that $\ASpec\mcA$ is homeomorphic to some closed subset of $\ASpec(\Mod R)$.
\end{Cor}

\begin{proof}
	This follows from Theorem \ref{Thm:GabrielPopescuEmbedding} and Theorem \ref{Thm:AtomSpectraOfQuotientCategory}.
\end{proof}

In section \ref{sec:ConstructionOfGrothendieckCategories} and section \ref{sec:ExamplesOfGrothendieckCategories}, we construct several Grothendieck categories which have significant structures. Then by using Theorem \ref{Thm:GabrielPopescuEmbedding} and Corollary \ref{Cor:AtomSpectrumOfGrothendieckCategoryIsClosedSubsetOfAtomSpectrumOfModuleCategory}, we obtain some rings with some structures from the viewpoint of atoms.

%%%%%%%%%%%%%%%%%%%%%%%%%%%%%%%%%%%%%%%%%%%%%%%%%%%%%%%%%%%%%%%%%%%%%%%%%%%%%%%%
\section{Localization}
\label{sec:Localization}
%%%%%%%%%%%%%%%%%%%%%%%%%%%%%%%%%%%%%%%%%%%%%%%%%%%%%%%%%%%%%%%%%%%%%%%%%%%%%%%%

In this section, we investigate the localization of a Grothendieck category $\mcA$ at an atom $\alpha$ in $\mcA$. It is defined as the quotient category of $\mcA$ by a certain localizing subcategory. This contains a reformulation of localization by a prime localizing subcategory written in \cite[section 4.20]{Popescu}. We show that the partial order on the atom spectrum defined in section \ref{sec:PartialOrdersOnAtomSpectra} coincides with a partial order naturally defined between prime localizing subcategories. We show directly that the localization at an atom is a generalization of the localization of a commutative ring $R$ at a prime ideal of $\mfp$.

\begin{Def}\label{Def:LocalizationAtAtom}
	Let $\mcA$ be a Grothendieck category and $\alpha$ an atom in $\mcA$.
	\begin{enumerate}
		\item\label{item:DefLocalizingSubcategoryCorrespondingToAtom} Define an open subset $\Phi_{\alpha}$ of $\ASpec\mcA$ by $\Phi_{\alpha}=\ASpec\mcA\setminus\overline{\{\alpha\}}$. Define a localizing subcategory $\mcX_{\alpha}$ of $\mcA$ by $\mcX_{\alpha}=\ASupp^{-1}\Phi_{\alpha}$.
		\item\label{item:DefLocalizationAtAtom} Denote by $\mcA_{\alpha}$ the quotient category of $\mcA$ by $\mcX_{\alpha}$. The canonical functor $\mcA\to\mcA_{\alpha}$ is denoted by $(-)_{\alpha}$. We call the operation of obtaining $\mcA_{\alpha}$ from $\mcA$ the \emph{localization} at $\alpha$.
	\end{enumerate}
\end{Def}

Note that $\mcX_{\alpha}$ defined in Definition \ref{Def:LocalizationAtAtom} (\ref{item:DefLocalizingSubcategoryCorrespondingToAtom}) is a localizing subcategory of $\mcA$ by Proposition \ref{Prop:AtomSupportAndAssociatedAtomsAndShortExactSequence} (\ref{item:PropAtomSupportAndShortExactSequence}) and Proposition \ref{Prop:AtomSupportAndAssociatedAtomsAndDirectSum} (\ref{item:PropAtomSupportAndDirectSum}).

By using localization, we obtain a description of atom supports, which is known as the standard definition of supports in commutative ring theory.

\begin{Prop}\label{Prop:AtomSupportAndLocalization}
	Let $\mcA$ be a Grothendieck category. For any object $M$ in $\mcA$, we have
	\begin{equation*}
		\ASupp M=\{\alpha\in\ASpec\mcA\mid M_{\alpha}\neq 0\}.
	\end{equation*}
\end{Prop}

\begin{proof}
	By Theorem \ref{Thm:PropertyOfQuotientCategory} (\ref{item:ThmCharacterizationOfZeroObjectInQuotientCategory}), we have $M_{\alpha}=0$ if and only if $\ASupp M\cap\overline{\{\alpha\}}=\emptyset$. Since $\ASupp M$ is an open subset of $\ASpec\mcA$, we have $\ASupp M\cap\overline{\{\alpha\}}=\emptyset$ if and only if $\alpha\notin\ASupp M$.
\end{proof}

In fact, the localized category $\mcA_{\alpha}$ in Definition \ref{Def:LocalizationAtAtom} (\ref{item:DefLocalizationAtAtom}) is ``local'' in the sense of next definition. This is shown in Proposition \ref{Prop:AtomSpectrumOfLocalizedCategory} (\ref{item:PropAtomSpectrumOfLocalizedCategory}).

\begin{Def}\label{Def:LocalCategoryAndPrimeLocalizingSubcategory}\leavevmode
	\begin{enumerate}
		\item\label{item:DefLocalCategory} A Grothendieck category $\mcA$ is called \emph{local} if there exists a simple object $S$ in $\mcA$ such that $E(S)$ is a cogenerator of $\mcA$.
		\item\label{item:DefPrimeLocalizingSubcategory} A localizing subcategory $\mcX$ of a Grothendieck category $\mcA$ is called \emph{prime} if $\mcA/\mcX$ is a local Grothendieck category.
	\end{enumerate}
\end{Def}

Local Grothendieck categories have the following characterization in terms of atoms.

\begin{Prop}\label{Prop:CharacterizationOfLocalGrothendieckCategory}
	Let $\mcA$ be a Grothendieck category.
	\begin{enumerate}
		\item\label{item:PropCharacterizationOfLocalGrothendieckCategoryByUsingAtom} $\mcA$ is local if and only if there exists an atom $\alpha$ in $\mcA$ such that for any nonzero object $M$ in $\mcA$, we have $\alpha\in\ASupp M$.
		\item\label{item:PropSimpleObjectInLocalGrothendieckCategory} If $\mcA$ is local, then any two simple objects in $\mcA$ are isomorphic. In the case where $\mcA$ is a nonzero locally noetherian Grothendieck category, the converse also holds.
	\end{enumerate}
\end{Prop}

\begin{proof}
	(\ref{item:PropCharacterizationOfLocalGrothendieckCategoryByUsingAtom}) Assume that $\mcA$ is local. Then there exists a simple object $S$ in $\mcA$ such that $E(S)$ is a cogenerator of $\mcA$. For any object $M$ in $\mcA$, there exists a nonzero morphism $f\colon M\to E(S)$. Since $\Im f$ is a nonzero subobject of $E(S)$, we have $M\twoheadrightarrow\Im f\hookleftarrow S$. This shows that $\overline{S}\in\ASupp M$.
	
	Conversely, assume that there exists an atom $\alpha$ in $\mcA$ such that $\alpha\in\ASupp M$ for any nonzero object $M$ in $\mcA$. Let $H$ be a monoform object in $\mcA$ such that $\overline{H}=\alpha$. If $H$ is not simple, then there exists a nonzero proper subobject $L$ of $H$, and hence we have $\alpha\in\ASupp(H/L)$. This contradicts Proposition \ref{Prop:MonoformObjectAndAtomSupport}. Therefore $H$ is a simple object in $\mcA$. Denote it by $S$.
	
	 Let $N$ and $N'$ be objects in $\mcA$ and $g\colon N\to N'$ a nonzero morphism. Since $\Im g$ is nonzero, there exists a subobject $B$ of $\Im g$ which has a quotient object isomorphic to $S$. By the injectivity of $E(S)$, we obtain a nonzero morphism $\Im g\to E(S)$ such that the diagram
	 \begin{equation*}
	 	\newdir^{ (}{!/-5pt/@^{(}}
		\xymatrix{
			\Im g\ar[r] & E(S) \\
			B\ar@{^{ (}->}[u]\ar@{>>}[r] & S\ar@{^{ (}->}[u]
		}
	\end{equation*}
	commutes. We also obtain a morphism $h\colon N'\to E(S)$ such that the diagram
	\begin{equation*}
	 	\newdir^{ (}{!/-5pt/@^{(}}
		\xymatrix{
			N'\ar^-{h}[dr] & \\
			\Im g\ar@{^{ (}->}[u]\ar[r] & E(S)
		}
	\end{equation*}
	commutes. Then we have $hg\neq 0$. This shows that $E(S)$ is a cogenerator of $\mcA$.
	
	(\ref{item:PropSimpleObjectInLocalGrothendieckCategory}) The former part is shown in \cite[Proposition 4.20.1 (1)]{Popescu}.
	
	Assume that $\mcA$ is a nonzero locally noetherian Grothendieck category such that any two simple objects in $\mcA$ are isomorphic. Any nonzero object $M$ in $\mcA$ has a nonzero noetherian subobject $L$, and $L$ has a simple quotient object. Hence there exists a simple object $S$ in $\mcA$, and we have $\overline{S}\in\ASupp M$ for any nonzero object $M$ in $\mcA$. By (\ref{item:PropCharacterizationOfLocalGrothendieckCategoryByUsingAtom}), it is shown that $\mcA$ is local.
\end{proof}

The localness of a locally noetherian Grothendieck category can be stated in terms of the topology on the atom spectrum. Recall that a point $x$ in a topological space $X$ is called a \emph{generic point} if the closure of $\{x\}$ coincides with $X$.

\begin{Cor}\label{Cor:CharacterizationOfLocalGrothendieckCategoryUsingTopology}
	Let $\mcA$ be a locally noetherian Grothendieck category. Then the following assertions are equivalent.
	\begin{enumerate}
		\item\label{item:CorCategoryIsLocal} $\mcA$ is local.
		\item\label{item:CorExistenceOfGenericPoint} $\ASpec\mcA$ has a generic point.
		\item\label{item:CorExistenceOfLargestElement} $\ASpec\mcA$ has a largest element.
	\end{enumerate}
\end{Cor}

\begin{proof}
	If $\mcA$ is local, then (\ref{item:CorExistenceOfGenericPoint}) and (\ref{item:CorExistenceOfLargestElement}) follow from Proposition \ref{Prop:CharacterizationOfLocalGrothendieckCategory} (\ref{item:PropCharacterizationOfLocalGrothendieckCategoryByUsingAtom}), Proposition \ref{Prop:OpenBasisOfAtomSpectrum}, and the definition of the specialization order.
	
	If (\ref{item:CorExistenceOfGenericPoint}) or (\ref{item:CorExistenceOfLargestElement}) holds, then the localness of $\mcA$ follows from Proposition \ref{Prop:CharacterizationOfLocalGrothendieckCategory} (\ref{item:PropSimpleObjectInLocalGrothendieckCategory}) and Proposition \ref{Prop:CharacterizationOfOpenAtomAndDiscreteAtomSpectrum} (\ref{item:PropCharacterizationOfOpenAtom}).
\end{proof}

For a Grothendieck category $\mcA$, the atom spectra of localized categories of $\mcA$ are described by using the partial order on the atom spectrum of $\mcA$.

\begin{Prop}\label{Prop:AtomSpectrumOfLocalizedCategory}
	Let $\mcA$ be a Grothendieck category.
	\begin{enumerate}
		\item\label{item:PropAtomSpectrumOfLocalizedCategory} For any atom $\alpha$ in $\mcA$, the topological space $\ASpec\mcA_{\alpha}$ is homeomorphic to the closed subset
		\begin{equation*}
			\{\beta\in\ASpec\mcA\mid\beta\leq\alpha\}
		\end{equation*}
		of $\ASpec\mcA$ with the induced topology. In particular, the category $\mcA_{\alpha}$ is local.
		\item\label{item:PropCharacterizationOfLocalCategory} $\mcA$ is local if and only if there exists an atom $\alpha$ in $\mcA$ such that the canonical functor $\mcA\to\mcA_{\alpha}$ is an equivalence.
	\end{enumerate}
\end{Prop}

\begin{proof}
	(\ref{item:PropAtomSpectrumOfLocalizedCategory}) We obtain the description of $\ASpec\mcA_{\alpha}$ by Theorem \ref{Thm:AtomSpectraOfQuotientCategory} and the definition of the specialization order.
	
	For any nonzero object $M'$ in $\mcA_{\alpha}$, by Lemma \ref{Lem:AtomSupportAndAssociatedAtomsForQuotientCategory} (\ref{item:LemAtomSupportAndAssociatedAtomsForQuotientCategory}), we have
	\begin{equation*}
		\ASupp M'=\ASupp G(M')\cap\overline{\{\alpha\}}.
	\end{equation*}
	Hence $\alpha\in\ASupp M'$ by Theorem \ref{Thm:PropertyOfQuotientCategory} (\ref{item:ThmPropertyOfCanonicalFunctorForQuotientCategory}) and Proposition \ref{Prop:AtomSupportAndLocalization}. The localness of $\mcA_{\alpha}$ follows from Proposition \ref{Prop:CharacterizationOfLocalGrothendieckCategory} (\ref{item:PropCharacterizationOfLocalGrothendieckCategoryByUsingAtom}).
	
	(\ref{item:PropCharacterizationOfLocalCategory}) If $\mcA$ is equivalent to $\mcA_{\alpha}$ for some atom $\alpha$ in $\mcA$, then $\mcA$ is local by (\ref{item:PropAtomSpectrumOfLocalizedCategory}).
		
	Assume that $\mcA$ is local. Then by Proposition \ref{Prop:CharacterizationOfLocalGrothendieckCategory} (\ref{item:PropCharacterizationOfLocalGrothendieckCategoryByUsingAtom}), there exists an atom $\alpha$ in $\mcA$ such that $\alpha\in\ASupp M$ for any nonzero object $M$ in $\mcA$. This shows that $\mcX_{\alpha}$ consists of zero objects in $\mcA$. Hence the canonical functor $\mcA\to\mcA_{\alpha}$ is an equivalence.
\end{proof}

The minimality of atoms in locally noetherian Grothendieck categories is characterized as follows by using localization.

\begin{Prop}\label{Prop:CharacterizationOfMinimalAtomByLocalization}
	Let $\mcA$ be a locally noetherian Grothendieck category and $\alpha$ an atom in $\mcA$.
	\begin{enumerate}
		\item\label{item:PropCharacterizationOfMinimalAtomInAtomSupportByLocalization} Let $M$ be a noetherian object in $\mcA$ with $\alpha\in\ASupp M$. Then $\alpha$ is minimal in $\ASupp M$ if and only if $M_{\alpha}$ has finite length.
		\item\label{item:PropCharacterizationOfMinimalAtomByLocalization} $\alpha$ is minimal in $\ASpec\mcA$ if and only if any noetherian object in $\mcA_{\alpha}$ has finite length.
	\end{enumerate}
\end{Prop}

\begin{proof}
	(\ref{item:PropCharacterizationOfMinimalAtomInAtomSupportByLocalization}) By Theorem \ref{Thm:PropertyOfQuotientCategory} (\ref{item:ThmPropertyOfCanonicalFunctorForQuotientCategory}), it is straightforward to show that $M_{\alpha}$ is a noetherian object in $\mcA_{\alpha}$. Hence by Proposition \ref{Prop:CharacterizationOfOpenAtomAndDiscreteAtomSpectrum} (\ref{item:PropCharacterizationOfDiscreteAtomSupport}), $M_{\alpha}$ has finite length if and only if $\ASupp M_{\alpha}$ is discrete. By Lemma \ref{Lem:AtomSupportAndAssociatedAtomsForQuotientCategory} (\ref{item:LemAtomSupportAndQuotientCategory}) and Proposition \ref{Prop:AtomSpectrumOfLocalizedCategory} (\ref{item:PropAtomSpectrumOfLocalizedCategory}), we have
	\begin{equation*}
		\ASupp M_{\alpha}=\ASupp M\cap\{\beta\in\ASpec\mcA\mid\beta\leq\alpha\}.
	\end{equation*}
	Therefore by the definition of the specialization order, $\ASupp M_{\alpha}$ is discrete if and only if $\alpha$ is minimal in $\ASupp M$.
	
	(\ref{item:PropCharacterizationOfMinimalAtomByLocalization}) By Proposition \ref{Prop:CharacterizationOfOpenAtomAndDiscreteAtomSpectrum} (\ref{item:PropCharacterizationOfDiscreteAtomSpectrum}), any noetherian object in $\mcA_{\alpha}$ has finite length if and only if $\ASpec\mcA_{\alpha}$ is discrete. By Proposition \ref{Prop:AtomSpectrumOfLocalizedCategory} (\ref{item:PropAtomSpectrumOfLocalizedCategory}), it is equivalent to that $\alpha$ is minimal in $\ASpec\mcA$.
\end{proof}

For any Grothendieck category $\mcA$ and any atom $\alpha$ in $\mcA$, Proposition \ref{Prop:AtomSpectrumOfLocalizedCategory} (\ref{item:PropAtomSpectrumOfLocalizedCategory}) shows that $\mcX_{\alpha}$ is a prime localizing subcategory of $\mcA$. In fact, this correspondence gives a bijection between atoms in $\mcA$ and prime localizing subcategories of $\mcA$. The partial order on $\ASpec\mcA$ is induced by the opposite relation of the inclusion relation between prime localizing subcategories.

\begin{Thm}\label{Thm:BijectionBetweenAtomsAndPrimeLocalizingSubcategory}
	Let $\mcA$ be a Grothendieck category. Then the map $\alpha\mapsto\mcX_{\alpha}$ is a bijection between atoms in $\mcA$ and prime localizing subcategories of $\mcA$. For any atoms $\alpha$ and $\beta$ in $\mcA$, we have $\alpha\leq\beta$ if and only if $\mcX_{\alpha}\supset\mcX_{\beta}$.
\end{Thm}

\begin{proof}
	Let $\mcX$ be a prime localizing subcategory of $\mcA$ and $F\colon\mcA\to\mcA/\mcX$ the canonical functor. Let $\alpha$ be the atom in $\mcA$ corresponding to the unique simple object $S$ in $\mcA/\mcX$. Then we have $\alpha\notin\ASupp\mcX$. We show that $\mcX=\mcX_{\alpha}$. Since $\ASupp\mcX$ is an open subset of $\ASpec\mcA$, we have $\ASupp\mcX\subset\Phi_{\alpha}$. Hence we have
	\begin{equation*}
		\mcX\subset\ASupp^{-1}(\ASupp\mcX)\subset\ASupp^{-1}\Phi_{\alpha}=\mcX_{\alpha}.
	\end{equation*}
	For any object $M$ in $\mcA$ belonging to $\mcX_{\alpha}$, we have $\alpha\notin\ASupp M$. By Lemma \ref{Lem:AtomSupportAndAssociatedAtomsForQuotientCategory} (\ref{item:LemAtomSupportAndQuotientCategory}),
	\begin{equation*}
		\ASupp F(M)=\ASupp M\setminus\ASupp\mcX,
	\end{equation*}
	and hence $\overline{S}\notin\ASupp F(M)$. By the proof of Proposition \ref{Prop:CharacterizationOfLocalGrothendieckCategory} (\ref{item:PropCharacterizationOfLocalGrothendieckCategoryByUsingAtom}), it holds that $F(M)=0$, and hence $M$ belongs to $\mcX$. Therefore we have $\mcX=\mcX_{\alpha}$. This shows that the map $\alpha\mapsto\mcX_{\alpha}$ is surjective.
	
	Let $\alpha$ and $\beta$ be atoms in $\mcA$. Assume that $\alpha\leq\beta$. For any object $M$ in $\mcA$ belonging to $\mcX_{\beta}$, we have $\beta\notin\ASupp M$. By Proposition \ref{Prop:TopologicalCharacterizationOfPartialOrder}, we have $\alpha\notin\ASupp M$, and hence $M\in\mcX_{\alpha}$. This shows that $\mcX_{\alpha}\supset\mcX_{\beta}$.
	
	Conversely, assume that $\alpha\nleq\beta$. Then there exists a monoform object $H$ in $\mcA$ satisfying $\overline{H}=\alpha$ and $\beta\notin\ASupp H$. This shows that $H$ does not belong to $\mcX_{\alpha}$ and that $H$ belongs to $\mcX_{\beta}$. Hence we have $\mcX_{\alpha}\not\supset\mcX_{\beta}$. This also shows the map $\alpha\mapsto\mcX_{\alpha}$ is injective.
\end{proof}

The next proposition shows that the localization at an atom is a generalization of the localization of a commutative ring at a prime ideal. Although this is well known, we give a direct proof from the viewpoint of atoms.

\begin{Prop}\label{Prop:LocalizationAtAtomForCommutativeNoetherianRing}
	Let $R$ be a commutative ring, $\mfp$ a prime ideal of $R$, and $\alpha$ the corresponding atom $\overline{R/\mfp}$ in $\Mod R$. Then $(\Mod R)_{\alpha}$ is equivalent to $\Mod R_{\mfp}$.
\end{Prop}

\begin{proof}
	Let $F\colon\Mod R\to (\Mod R)_{\alpha}$ be the canonical functor and $G\colon(\Mod R)_{\alpha}\to\Mod R$ its right adjoint functor. By Theorem \ref{Thm:PropertyOfQuotientCategory} (\ref{item:ThmPropertyOfCanonicalFunctorForQuotientCategory}), $G$ is fully faithful. Since $\Mod R_{\mfp}$ can also be regarded as a full subcategory of $\Mod R$, it suffices to show that an $R$-module $M$ belongs to $\Mod R_{\mfp}$ if and only if $M$ is isomorphic to the image of some object in $(\Mod R)_{\alpha}$ by $G$. $M$ belongs to $\Mod R_{\mfp}$ if and only if any element of $R\setminus\mfp$ acts on $M$ as an isomorphism. Since $FG$ is isomorphic to the identity functor on $(\Mod R)_{\alpha}$, the $R$-module $M$ is isomorphic to $G(M')$ for some object $M'$ in $(\Mod R)_{\alpha}$ if and only if $M$ is isomorphic to $GF(M)$, and this is also equivalent to that $M$ and $E(M)/M$ have no nonzero submodule belonging to $\mcX_{\alpha}$ by Theorem \ref{Thm:PropertyOfQuotientCategory} (\ref{item:ThmDescriptionOfCompositeFunctorForQuotientCategory}). By Proposition \ref{Prop:AtomSupportAndAssociatedAtomsForCommutativeRing} (\ref{item:PropAtomSupportForCommutativeRing}), we have
	\begin{equation*}
		\mcX_{\alpha}=\{M\in\Mod R\mid\mfp\notin\Supp M\}=\{M\in\Mod R\mid M_{\mfp}=0\}.
	\end{equation*}
	Hence an $R$-module $N$ has a nonzero submodule belonging to $\mcX_{\alpha}$ if and only if there exist a nonzero element $z$ of $N$ and an element $s$ of $R\setminus\mfp$ such that $zs=0$.
	
	Assume that $M$ belongs to $\Mod R_{\mfp}$. Then $M$ has no nonzero submodule belonging to $\mcX_{\alpha}$. If $E(M)/M$ has a nonzero submodule belonging to $\mcX_{\alpha}$, then there exist $x\in E(M)\setminus M$ and $t\in R\setminus\mfp$ such that $xt\in M$. Since $t$ acts on $M$ as an isomorphism, there exists $y\in M$ such that $yt=xt$. We have $(x-y)t=0$ and $x-y\notin M$. Since $M$ is an essential submodule of $E(M)$, there exists $a\in R$ such that $0\neq (x-y)a\in M$. However, $(x-y)at=0$ implies $(x-y)a=0$. This is a contradiction. Therefore $E(M)/M$ has no nonzero submodule belonging to $\mcX_{\alpha}$.
	
	Conversely, assume that $M$ and $E(M)/M$ have no nonzero submodule belonging to $\mcX_{\alpha}$. Let $r\in R\setminus\mfp$, and let $f\colon M\to M$ and $\widetilde{f}\colon E(M)\to E(M)$ be the maps defined by the multiplication of $r$. Then $f$ is a monomorphism. Since we have $\Ker\widetilde{f}\cap M=\Ker f=0$, the map $\widetilde{f}$ is also a monomorphism. By the injectivity of $E(M)$, there exists a morphism $\widetilde{g}\colon E(M)\to E(M)$ such that $\widetilde{g}\widetilde{f}$ is the identity morphism on $E(M)$. Since we have $\widetilde{f}\widetilde{g}=\widetilde{g}\widetilde{f}$, the map $\widetilde{f}$ is an isomorphism. For any $w\in M$, we have $w=\widetilde{g}\widetilde{f}(w)=\widetilde{g}(w)r$. Since $E(M)/M$ has no nonzero submodule belonging to $\mcX_{\alpha}$, we have $\widetilde{g}(w)\in M$. This shows that $f$ is an isomorphism. Therefore $M$ belongs to $\Mod R_{\mfp}$.
\end{proof}

%%%%%%%%%%%%%%%%%%%%%%%%%%%%%%%%%%%%%%%%%%%%%%%%%%%%%%%%%%%%%%%%%%%%%%%%%%%%%%%%
\section{Construction of Grothendieck categories}
\label{sec:ConstructionOfGrothendieckCategories}
%%%%%%%%%%%%%%%%%%%%%%%%%%%%%%%%%%%%%%%%%%%%%%%%%%%%%%%%%%%%%%%%%%%%%%%%%%%%%%%%

In this section, we introduce systematic methods to construct Grothendieck categories from colored quivers. As an application, we show that any partially ordered set appears as the atom spectrum of some Grothendieck category.

%%%%%%%%%%%%%%%%%%%%%%%%%%%%%%%%%%%%%%%%%%%%%%%%%%%%%%%%%%%%%%%%%%%%%%%%%%%%%%%%
\subsection{Known results in commutative ring theory}
\label{subsec:KnownResultsInCommutativeRingTheory}
%%%%%%%%%%%%%%%%%%%%%%%%%%%%%%%%%%%%%%%%%%%%%%%%%%%%%%%%%%%%%%%%%%%%%%%%%%%%%%%%

In this subsection, we recall motivating facts in commutative ring theory. We do not use these facts in later discussion.

A nonempty topological space $X$ is called \emph{irreducible} if for any closed subsets $Y_{1}$ and $Y_{2}$ of $X$, we have $Y_{1}\cup Y_{2}\neq X$ unless $Y_{1}=X$ or $Y_{2}=X$.

\begin{Def}\label{Def:SpectralSpace}
	A topological space $X$ is called a \emph{spectral space} if it satisfies the following conditions.
	\begin{enumerate}
		\item $X$ is a Kolmogorov compact space.
		\item The intersection of any finite family of open compact subsets of $X$ is also compact.
		\item The family of all the open compact subsets of $X$ is an open basis of $X$.
		\item Any irreducible closed subset of $X$ has a generic point.
	\end{enumerate}
\end{Def}

Hochster \cite{Hochster} showed that the notion of spectral spaces characterizes the topological spaces appearing as the prime spectra of commutative rings. For a commutative ring $R$, we regard $\Spec R$ as a topological space with the Zariski topology.

\begin{Thm}[{Hochster \cite[Theorem 6 and Proposition 10]{Hochster}}]\label{Thm:HochsterTheorem}
	Let $X$ be a topological space. Then $X$ is homeomorphic to $\Spec R$ for some commutative ring $R$ if and only if $X$ is a spectral space. This is also equivalent to that $X$ is an inverse limit of finite Kolmogorov spaces in the category of topological spaces.
\end{Thm}

Speed \cite{Speed} pointed out that a similar characterization for partially ordered sets follows from Theorem \ref{Thm:HochsterTheorem}. Note that there exists a bijection between homeomorphism classes of finite Kolmogorov spaces and isomorphism classes of finite partially ordered sets (Proposition \ref{Prop:BijectionBetweenKolmogorovAlexandroffSpacesAndPartiallyOrderedSets} (\ref{item:PropBijectionBetweenFiniteKolmogorovSpacesAndFinitePartiallyOrderedSets})). For a commutative ring $R$, we regard $\Spec R$ also as a partially ordered set with the inclusion relation.

\begin{Cor}[{Hochster \cite[Proposition 10]{Hochster}} and {Speed \cite[Corollary 1]{Speed}}]\label{Cor:SpectralPosetIsInverseLimitOfFinitePoset}
	Let $P$ be a partially ordered set. Then $P$ is isomorphic to the partially ordered set $\Spec R$ for some commutative ring $R$ if and only if $P$ is an inverse limit of finite partially ordered sets in the category of partially ordered sets.
\end{Cor}

If we assume that the commutative ring $R$ is noetherian, then the topological spaces and the partially ordered sets appearing as the prime spectra of $R$ are strongly restricted. In the case of finite partially ordered sets, the following result is known.

\begin{Thm}[{de Souza Doering and Lequain \cite[Theorem B]{deSouzaDoeringLequain}}]\label{Thm:FinitePosetOfCommutativeNoetherianRing}
	Let $P$ be a finite partially ordered set. Then $P$ is isomorphic to the partially ordered set $\Spec R$ for some commutative noetherian ring $R$ if and only if there does not exist a chain of the form $x<y<z$ in $P$.
\end{Thm}

\begin{proof}
	It is well known that for any commutative noetherian ring $R$ whose prime spectrum is a finite set, there does not exist a chain of the form $x<y<z$ in $\Spec R$, that is, the Krull dimension of $R$ is at most one (\cite[Theorem 144]{Kaplansky}). The converse follows from \cite[Theorem B]{deSouzaDoeringLequain}.
\end{proof}

%%%%%%%%%%%%%%%%%%%%%%%%%%%%%%%%%%%%%%%%%%%%%%%%%%%%%%%%%%%%%%%%%%%%%%%%%%%%%%%%
\subsection{Our results}
\label{subsec:OurResults}
%%%%%%%%%%%%%%%%%%%%%%%%%%%%%%%%%%%%%%%%%%%%%%%%%%%%%%%%%%%%%%%%%%%%%%%%%%%%%%%%

From the viewpoint of study of Grothendieck categories, it is natural to ask which partially ordered sets appear as the atom spectra of Grothendieck categories. We answer this question by constructing Grothendieck categories from colored quivers and by considering the quotient categories by some localizing subcategories.

We fix a field $K$ until the end of this paper.

\begin{Def}\label{Def:ValuedQuiverAndColoredQuiver}
	Let $K$ be a field.
	\begin{enumerate}
		\item\label{item:DefQuiver} We call a quintuple $Q=(Q_{0},Q_{1},s,t,m)$ a $K$\emph{-valued quiver} if $Q_{0}$ and $Q_{1}$ are sets, $s$ and $t$ are maps $Q_{1}\to Q_{0}$, and $m$ is a map $Q_{1}\to K$. Elements of $Q_{0}$ (resp.\ $Q_{1}$) are called \emph{vertices} (resp.\ \emph{arrows}) in $Q$. For an arrow $r$ in $Q$, we call $s(r)$ (resp.\ $t(r)$, $m(r)$) the \emph{source} (resp.\ \emph{target}, \emph{value}) of $r$.
		\item\label{item:DefColoredQuiver} We call a triple $\Gamma=(Q,C,u)$ a $K$\emph{-valued colored quiver} if $Q$ is a $K$-valued quiver, $C$ is a set, and $u$ is a map $Q_{1}\to C$. For an arrow $r$ in $Q$, we call $u(r)$ the \emph{color} of $r$.
		\item\label{item:DefFiniteArrowsCondition} We say that a $K$-valued colored quiver $\Gamma=(Q,C,u)$ satisfies the \emph{finite arrow condition} if for any vertex $v$ in $Q$ and $c\in C$, there exist only finitely many arrows in $Q$ with the source $v$, the color $c$, and the nonzero value.
	\end{enumerate}
\end{Def}

We construct a Grothendieck category from a $K$-valued colored quiver satisfying the finite arrow condition as follows.

\begin{Def}\label{Def:GrothendieckCategoryAssociatedWithColoredQuiver}
	Let $K$ be a field, and let $\Gamma=(Q,C,u)$ be a $K$-valued colored quiver satisfying the finite arrow condition, where $Q=(Q_{0},Q_{1},s,t,m)$. Denote the free $K$-algebra on $C$ by $S_{C}=K\left<s_{c}\mid c\in C\right>$. Define a $K$-vector space $M_{\Gamma}$ by $M_{\Gamma}=\bigoplus_{v\in Q_{0}}F_{v}$, where $F_{v}=x_{v}K$ is a one-dimensional $K$-vector space generated by an element $x_{v}$. Regard $M_{\Gamma}$ as a right $S_{C}$-module by defining the action of $s_{c}\in S_{C}$ as follows: for each vertex $v$ in $Q$,
	\begin{equation*}
		x_{v}\cdot s_{c}=\sum_{r}x_{t(r)}m(r),
	\end{equation*}
	where $r$ runs over all the arrows in $Q$ with the source $v$ and the color $c$. The number of such arrows $r$ is finite since $\Gamma$ satisfies the finite arrow condition. Denote by $\mcA_{\Gamma}$ the smallest prelocalizing subcategory (Definition \ref{Def:PrelocalizingSubcategory}) of $\Mod S_{C}$ containing $M_{\Gamma}$.
\end{Def}

\begin{ex}\label{ex:GrothendieckCategoryAssociatedWithFiniteColoredQuiver}
	Let $\Gamma=(Q,C,u)$ be the colored quiver defined as follows: let $Q=(Q_{0},Q_{1},s,t,m)$, where $Q_{0}=\{v,w\}$, $Q_{1}=\{r\}$, $s(r)=v$, $t(r)=w$, and $m(r)=q\in K$. Let $C=\{c\}$ and $u(r)=c$. Then we have $S_{C}=K\left<s_{c}\right>=K[s_{c}]$. As a $K$-vector space, $M_{\Gamma}=x_{v}K\oplus x_{w}K$. The action of $S_{C}$ on $M_{\Gamma}$ is described as $x_{v}s_{c}=x_{w}q$ and $x_{w}s_{c}=0$. Hence we have $M_{\Gamma}\cong K[s_{c}]/(s_{c}^{2})$. The full subcategory $\mcA_{\Gamma}$ of $\Mod S_{C}$ is equivalent to $\Mod K[x]/(x^{2})$.
\end{ex}

\begin{rem}\label{rem:ConnectionWithPathAlgebra}
	The module $M_{\Gamma}$ in Definition \ref{Def:GrothendieckCategoryAssociatedWithColoredQuiver} can be described by using path algebra if the sets $Q_{0}$ and $Q_{1}$ are finite. The reader may refer to \cite{AssemSimsonSkowronski} for the terms of path algebras and representations of quivers.
	
	For the $K$-valued colored quiver $\Gamma$ given in Definition \ref{Def:GrothendieckCategoryAssociatedWithColoredQuiver} such that $Q_{0}$ and $Q_{1}$ are finite, we consider the path algebra $KQ'$ of the quiver $Q'=(Q_{0},Q_{1},s,t)$. Define a $K$-linear representation $M$ of $Q'$ by letting $M_{v}=K$ for each $v\in Q_{0}$ and the $K$-linear map $M_{s(r)}\to M_{t(r)}$ be the identity map on $K$ for each $r\in Q_{1}$. We regard $M$ as a $KQ'$-module. Define a $K$-algebra homomorphism $S_{C}\to KQ'$ by $s_{c}\mapsto\sum_{r}r\cdot m(r)$ for each $c\in C$, where $r$ runs over all $r\in Q_{1}$ satisfying $u(r)=c$. Then $M$ can be regarded as an $S_{C}$-module through this $K$-algebra homomorphism. The $S_{C}$-module $M$ is isomorphic to $M_{\Gamma}$.
\end{rem}

In Definition \ref{Def:GrothendieckCategoryAssociatedWithColoredQuiver}, the full subcategory $\mcA_{\Gamma}$ is a Grothendieck category. By Proposition \ref{Prop:AtomSpectraOfLocalizingSubcategory}, Proposition \ref{Prop:AtomSupportAndAssociatedAtomsAndShortExactSequence} (\ref{item:PropAtomSupportAndShortExactSequence}), and Proposition \ref{Prop:AtomSupportAndAssociatedAtomsAndDirectSum} (\ref{item:PropAtomSupportAndDirectSum}), it holds that $\ASpec\mcA_{\Gamma}$ is homeomorphic to the open subset $\ASupp M_{\Gamma}$ of $\ASpec(\Mod S_{C})$.

If $M_{\Gamma}$ is the sum of its noetherian submodules, then $\mcA_{\Gamma}$ is contained by the locally noetherian Grothendieck category $\Noeth\mcA$ by Proposition \ref{Prop:LocallyNoetherianGrothendieckCategoryFromGrothendieckCategory}. Hence $\mcA_{\Gamma}$ is also a locally noetherian Grothendieck category.

\begin{rem}\label{rem:ValuedColoredQuiverWithMultipleArrows}
	\emph{Multiple arrows} in a $K$-valued quiver are arrows with a common source and a common target. For the purpose of constructing Grothendieck categories as above, we only have to consider $K$-valued colored quivers which have no multiple arrows with a common color by the following argument.
	
	For a $K$-valued colored quiver $\Gamma=(Q,C,u)$ satisfying the finite arrow condition, we can define another $K$-valued colored quiver $\Gamma'=(Q',C',u')$ satisfying the finite arrow condition as follows: let $Q=(Q_{0},Q_{1},s,t,m)$ and $Q'=(Q'_{0},Q'_{1},s',t',m')$.
	\begin{enumerate}
		\item Let $Q'_{0}=Q_{0}$ and $C'=C$.
		\item Let $Q'_{1}=\{r_{v,w,c}\}_{v,w\in Q_{0},\,c\in C}$.
		\item For each $v,w\in Q_{0}$ and $c\in C$, let $s(r_{v,w,c})=v$, $t(r_{v,w,c})=w$, $u(r_{v,w,c})=c$, and
		\begin{equation*}
			m(r_{v,w,c})=\sum_{r}m(r),
		\end{equation*}
		where $r$ runs over all the arrows in $Q$ with the source $v$, the target $w$, and the color $c$.
	\end{enumerate}
	Then we have $\mcA_{\Gamma'}=\mcA_{\Gamma}$, and $\Gamma'$ has no multiple arrows with a common color.
	
	We also note that $\mcA_{\Gamma}$ does not change if we replace $C$ with another set containing $C$.
\end{rem}

The construction in Definition \ref{Def:GrothendieckCategoryAssociatedWithColoredQuiver} yields various kinds of Grothendieck categories. For example, we can obtain the categories of modules over arbitrary algebras over fields.

\begin{Prop}\label{Prop:AnyAlgebraIsRealizedByValuedColoredQuiver}
	Let $K$ be a field and $R$ a $K$-algebra. Then there exists a $K$-valued colored quiver $\Gamma$ such that $\mcA_{\Gamma}$ is equivalent to $\Mod R$.
\end{Prop}

\begin{proof}
	Let $B\subset R$ be a $K$-basis of $R$. Define a $K$-valued colored quiver $\Gamma=(Q,C,u)$, where $Q=(Q_{0},Q_{1},s,t,m)$, as follows.
	\begin{enumerate}
		\item Let $Q_{0}=\{v_{b}\}_{b\in B}$ and $C=\{c_{b}\}_{b\in B}$.
		\item Let $Q_{1}=\{r_{b,b',b''}\}_{b,b',b''\in B}$.
		\item For each $b,b',b''\in B$, let $s(r_{b,b',b''})=v_{b}$, $t(r_{b,b',b''})=v_{b'}$, and $u(r_{b,b',b''})=c_{b''}$.
		\item Define a family $\{q_{b,b',b''}\}_{b,b',b''\in B}$ of elements of $K$ by
		\begin{equation*}
			bb''=\sum_{b'\in B}b'q_{b,b',b''}
		\end{equation*}
		for each $b,b''\in B$. For each $b,b',b''\in B$, let $m(r_{b,b',b''})=q_{b,b',b''}$.
	\end{enumerate}
	Then we have the surjective $K$-algebra homomorphism $S_{C}\to R$ defined by $c_{b}\mapsto b$ for each $b\in B$. It induces a fully faithful functor $F\colon\Mod R\to\Mod S_{C}$. The image of $\Mod R$ by $F$ is closed under subobjects, quotient objects, and direct sums up to isomorphism, and $F(R)$ is a generator of it. By the construction of $\Gamma$, it holds that $M_{\Gamma}$ is isomorphic to $F(R)$. Therefore $\Mod R$ is isomorphic to $\mcA_{\Gamma}$.
\end{proof}

From now on, we only consider $K$-valued colored quivers $\Gamma=(Q,C,u)$, where $Q=(Q_{0},Q_{1},s,t,m)$, satisfying the following conditions.
\begin{enumerate}
	\item For any arrow $r$ in $Q$, we have $m(r)=1\in K$.
	\item $\Gamma$ has no multiple arrows with a common color.
	\item $\Gamma$ satisfies the finite arrow condition.
\end{enumerate}
We call such $K$-valued colored quivers simply \emph{colored quivers}. The colored quiver $\Gamma$ is also denoted by $(Q_{0},Q_{1},C,s,t,u)$ or simply by $(Q_{0},Q_{1},C)$.

\begin{Def}\label{Def:IsomorphicColoredQuiversAndColoredFullSubquiver}\leavevmode
	\begin{enumerate}
		\item\label{item:DefIsomorphicColoredQuivers} Let $\Gamma^{i}=(Q^{i}_{0},Q^{i}_{1},C^{i},s^{i},t^{i},u^{i})$ be a colored quiver for $i=1,2$. We say that $\Gamma^{1}$ is \emph{isomorphic} to $\Gamma^{2}$ if we have
		\begin{equation*}
			C^{1}\supset\{u^{1}(r)\mid r\in Q^{1}_{1}\}=\{u^{2}(r')\mid r'\in Q^{2}_{1}\}\subset C^{2},
		\end{equation*}
		and there exist bijections $\varphi_{0}\colon Q^{1}_{0}\to Q^{2}_{0}$ and $\varphi_{1}\colon Q^{1}_{1}\to Q^{2}_{1}$ such that for any $r\in Q^{1}_{1}$, we have $\varphi_{0}(s^{1}(r))=s^{2}(\varphi_{1}(r))$, $\varphi_{0}(t^{1}(r))=t^{2}(\varphi_{1}(r))$, and $u^{1}(r)=u^{2}(\varphi_{1}(r))$.
		\item\label{item:DefColoredFullSubquiver} Let $\Gamma=(Q_{0},Q_{1},C,s,t,u)$ be a colored quiver. A colored quiver $\Gamma'=(Q'_{0},Q'_{1},C',s',t',u')$ is called a \emph{colored full subquiver} of $\Gamma$ if it satisfies the following conditions.
		\begin{enumerate}
			\item $Q'_{0}$ is a subset of $Q_{0}$.
			\item $Q'_{1}=\{r\in Q_{1}\mid s(r),t(r)\in Q'_{0}\}$.
			\item $\{u(r)\mid r\in Q'_{1}\}\subset C'$.
			\item $s'(r)=s(r)$, $t'(r)=t(r)$, and $u'(r)=u(r)$ for each $r\in Q'_{1}$.
		\end{enumerate}
	\end{enumerate}
\end{Def}

For a colored quiver $\Gamma=(Q_{0},Q_{1},C)$, certain subsets of $Q_{0}$ induce submodules and quotient modules of $M_{\Gamma}$ as follows.

\begin{rem}\label{rem:SubmoduleAssociatedToColoredFullSubquiver}
	Let $\Gamma=(Q_{0},Q_{1},C,s,t,u)$ be a colored quiver and $V$ a subset of $Q_{0}$ such that for any $r\in Q_{1}$ with $s(r)\in V$, we have $t(r)\in V$. Define colored full subquivers $\Gamma'=(Q'_{0},Q'_{1},C')$ and $\Gamma''=(Q''_{0},Q''_{1},C'')$ as follows.
	\begin{enumerate}
		\item $Q'_{0}=V$ and $Q''_{0}=Q_{0}\setminus V$.
		\item $Q'_{1}=\{r\in Q_{1}\mid s(r),t(r)\in Q'_{0}\}$, $Q''_{1}=\{r\in Q_{1}\mid s(r),t(r)\in Q''_{0}\}$.
		\item $C'=C''=C$.
	\end{enumerate}
	Then the exact sequence
	\begin{equation*}
		0\to\bigoplus_{v\in V}x_{v}K\to\bigoplus_{v\in Q_{0}}x_{v}K\to\bigoplus_{v\in Q_{0}\setminus V}x_{v}K\to 0
	\end{equation*}
	of $K$-vector spaces can be regarded as an exact sequence
	\begin{equation*}
		0\to M_{\Gamma'}\to M_{\Gamma}\to M_{\Gamma''}\to 0
	\end{equation*}
	in $\Mod S_{C}$, and this gives an identification of $M_{\Gamma'}$ (resp.\ $M_{\Gamma''}$) with the submodule (resp.\ quotient module) of $M_{\Gamma}$. By Proposition \ref{Prop:AtomSupportAndAssociatedAtomsAndShortExactSequence} (\ref{item:PropAtomSupportAndShortExactSequence}), we have
	\begin{equation*}
		\ASupp M_{\Gamma}=\ASupp M_{\Gamma'}\cup\ASupp M_{\Gamma''}.
	\end{equation*}
\end{rem}

\begin{Not}\label{Not:ColoredQuiver}
	We describe a colored quiver $\Gamma=(Q_{0},Q_{1},C,s,t,u)$ by a diagram as follows: each $v\in Q_{0}$ is represented by the symbol $v$, and each $r\in Q_{1}$ is represented by an arrow from $s(r)$ to $t(r)$ labeled by $u(r)$. For example, the diagram
	\begin{equation*}
		\xymatrix{
			v\ar[r]^-{c} & w
		}
	\end{equation*}
	is the colored quiver $\Gamma'=(Q'_{0},Q'_{1},C',s',t',u')$ defined by $Q'_{0}=\{v,w\}$, $Q'_{1}=\{r\}$, $C'=\{c\}$, $s'(r)=v$, $t'(r)=w$, and $u'(r)=c$.
\end{Not}

\begin{ex}\label{ex:AtomSpectrumOfGrothendieckCategoryAssociatedWithFiniteColoredQuiver}
	Let $\Gamma$ be the colored quiver
	\begin{equation*}
		\xymatrix{
			v_{1}\ar[r]^-{c_{1,2}} & v_{2}\ar[r]^-{c_{2,3}} & v_{3}.
		}
	\end{equation*}
	Then we have $S_{C}=K\left<s_{c_{1,2}},s_{c_{2,3}}\right>$, $M_{\Gamma}=x_{v_{1}}K\oplus x_{v_{2}}K\oplus x_{v_{3}}K$ as a $K$-vector space, and $x_{v_{i}}s_{c_{j,j+1}}=x_{v_{j+1}}\delta_{i,j}$ for each $i\in\{1,2,3\}$ and $j\in\{1,2\}$, where $\delta_{i,j}$ is the Kronecker delta. Hence $M_{\Gamma}$ is described as
	\begin{equation*}
		M_{\Gamma}=x_{v_{1}}S_{C}\cong\frac{S_{C}}{\Ann_{S_{C}}(x_{v_{1}})}=\frac{S_{C}}{s_{c_{2,3}}S_{C}+s_{c_{1,2}}s_{c_{1,2}}S_{C}+s_{c_{1,2}}s_{c_{2,3}}s_{c_{1,2}}S_{C}+s_{c_{1,2}}s_{c_{2,3}}s_{c_{2,3}}S_{C}}.
	\end{equation*}
	By definition, $M_{\Gamma}$ is an $S_{C}$-module belonging to the full subcategory $\mcA_{\Gamma}$ of $\Mod S_{C}$. We also define colored quivers $\Gamma_{1,2}$, $\Gamma_{2,3}$, $\Delta_{1}$, $\Delta_{2}$, and $\Delta_{3}$ as follows.
	\begin{gather*}
		\Gamma_{1,2}\colon
		\xymatrix{
			v_{1}\ar[r]^-{c_{1,2}} & v_{2},
		}\hspace{8mm}
		\Gamma_{2,3}\colon
		\xymatrix{
			v_{2}\ar[r]^-{c_{2,3}} & v_{3},
		}\\
		\Delta_{1}\colon
		\xymatrix{
			v_{1},
		}\hspace{8mm}
		\Delta_{2}\colon
		\xymatrix{
			v_{2},
		}\hspace{8mm}
		\Delta_{3}\colon
		\xymatrix{
			v_{3}.
		}
	\end{gather*}
	Then the $S_{C}$-modules $M_{\Gamma_{1,2}}$, $M_{\Gamma_{2,3}}$, $M_{\Delta_{1}}$, $M_{\Delta_{2}}$, and $M_{\Delta_{3}}$ also belong to $\mcA_{\Gamma}$ since they are subquotients of $M_{\Gamma}$ by Remark \ref{rem:SubmoduleAssociatedToColoredFullSubquiver}. For example, we have exact sequences
	\begin{equation*}
		0\to M_{\Delta_{3}}\to M\to M_{\Gamma_{1,2}}\to 0
	\end{equation*}
	and
	\begin{equation*}
		0\to M_{\Delta_{2}}\to M_{\Gamma_{1,2}}\to M_{\Delta_{1}}\to 0.
	\end{equation*}
	Since $M_{\Delta_{1}}$, $M_{\Delta_{2}}$, and $M_{\Delta_{3}}$ are isomorphic simple $S_{C}$-modules, we have
	\begin{align*}
		\ASpec\mcA_{\Gamma}
		&=\ASupp M_{\Gamma}\\
		&=\ASupp M_{\Delta_{1}}\cup\ASupp M_{\Delta_{2}}\cup\ASupp M_{\Delta_{3}}\\
		&=\{\overline{M_{\Delta_{1}}}=\overline{M_{\Delta_{2}}}=\overline{M_{\Delta_{3}}}\}
	\end{align*}
	by Proposition \ref{Prop:AtomSupportAndAssociatedAtomsAndShortExactSequence} (\ref{item:PropAtomSupportAndShortExactSequence}).
\end{ex}

The next example shows the way to distinguish simple modules corresponding to different vertices.

\begin{ex}\label{ex:AtomSpectrumOfGrothendieckCategoryAssociatedWithFiniteColoredQuiverWithLoop}
	Define colored quivers $\Gamma,\Gamma_{1,2}$, $\Gamma_{2,3}$, $\Delta_{1}$, $\Delta_{2}$, and $\Delta_{3}$ as follows.
	\begin{gather*}
		\Gamma\colon
		\xymatrix{
			v_{1}\ar@(dl,dr)_-{c_{1}}\ar[r]^-{c_{1,2}} & v_{2}\ar@(dl,dr)_-{c_{2}}\ar[r]^-{c_{2,3}} & v_{3}\ar@(dl,dr)_-{c_{3}}
		},\\
		\Gamma_{1,2}\colon
		\xymatrix{
			v_{1}\ar@(dl,dr)_-{c_{1}}\ar[r]^-{c_{1,2}} & v_{2}\ar@(dl,dr)_-{c_{2}}
		},\hspace{8mm}
		\Gamma_{2,3}\colon
		\xymatrix{
			v_{2}\ar@(dl,dr)_-{c_{2}}\ar[r]^-{c_{2,3}} & v_{3}\ar@(dl,dr)_-{c_{3}}
		},\\
		\Delta_{1}\colon
		\xymatrix{
			v_{1}\ar@(dl,dr)_-{c_{1}}
		},\hspace{8mm}
		\Delta_{2}\colon
		\xymatrix{
			v_{2}\ar@(dl,dr)_-{c_{2}}
		},\hspace{8mm}
		\Delta_{3}\colon
		\xymatrix{
			v_{3}\ar@(dl,dr)_-{c_{3}}
		}.
	\end{gather*}
	Then $M_{\Gamma_{1,2}}$, $M_{\Gamma_{2,3}}$, $M_{\Delta_{1}}$, $M_{\Delta_{2}}$, and $M_{\Delta_{3}}$ are subquotients of $M_{\Gamma}$. By a similar argument to that in Example \ref{ex:AtomSpectrumOfGrothendieckCategoryAssociatedWithFiniteColoredQuiver}, we have
	\begin{equation*}
		\ASpec\mcA_{\Gamma}=\{\overline{M_{\Delta_{1}}},\overline{M_{\Delta_{2}}},\overline{M_{\Delta_{3}}}\},
	\end{equation*}
	where $\overline{M_{\Delta_{1}}}$, $\overline{M_{\Delta_{2}}}$, and $\overline{M_{\Delta_{3}}}$ are distinct atoms in $\mcA_{\Gamma}$. The topology on $\ASpec\mcA_{\Gamma}$ is discrete since we have $\ASupp M_{\Delta_{i}}=\{\overline{M_{\Delta_{i}}}\}$ for each $i=1,2,3$.
\end{ex}

By using an infinite colored quiver, we can construct a Grothendieck category whose atom spectrum is not discrete.

\begin{ex}\label{ex:AtomSpectrumOfGrothendieckCategoryAssociatedWithInfiniteColoredQuiver}
	Let $\Gamma$ be the colored quiver
	\begin{equation*}
		\Gamma\colon
		\xymatrix{
			v_{0}\ar[r]^-{c_{0,1}} & v_{1}\ar[r]^-{c_{1,2}} & \cdots.
		}
	\end{equation*}
	Then we have $S_{C}=K\left<s_{c_{i,i+1}}\mid i\in\mbZ_{\geq 0}\right>$, $M_{\Gamma}=\bigoplus_{i=0}^{\infty}x_{v_{i}}K$ as a $K$-vector space, and $x_{v_{i}}s_{c_{j,j+1}}=x_{v_{j+1}}\delta_{i,j}$ for each $i,j\in\mbZ_{\geq 0}$.
	
	Let $L$ be a nonzero $S_{C}$-submodule of $M_{\Gamma}$ and $x$ a nonzero element of $L$. Then we have $x=\sum_{i\in I}x_{v_{i}}q_{i}$, where $I$ is a nonempty finite subset of $\mbZ_{\geq 0}$ and $q_{i}\in K\setminus\{0\}$ for each $i\in I$. Let $j$ be the smallest element of $I$. For each $i\in I\setminus\{j\}$, we have
	\begin{equation*}
		xs_{c_{j,j+1}}s_{c_{j+1,j+2}}\cdots s_{c_{i-1,i}}\frac{1}{q_{j}}=x_{v_{i}},
	\end{equation*}
	and hence $x_{v_{i}}\in xS_{C}$. This implies
	\begin{equation*}
		x_{v_{j}}=\left(x-\sum_{i\in I\setminus\{j\}}x_{v_{i}}q_{i}\right)\frac{1}{q_{j}}\in xS_{C},
	\end{equation*}
	and $xS_{C}\subset x_{v_{j}}S_{C}$. Since we have
	\begin{equation*}
		x_{v_{j}}s_{c_{j,j+1}}s_{c_{j+1,j+2}}\cdots s_{c_{i-1,i}}=x_{v_{i}},
	\end{equation*}
	it holds that $xS_{C}=x_{v_{j}}S_{C}$. Therefore $L$ is a sum of submodules of the form $L_{\geq j}=x_{v_{j}}S_{C}$ for some $j\in\mbZ_{\geq 0}$. Note that $L_{\geq j}=\bigoplus_{i=j}^{\infty}x_{v_{i}}K$ as a $K$-vector space. Since we have $L_{\geq 0}\supset L_{\geq 1}\supset\cdots$, there exists $j\in\mbZ_{\geq 0}$ such that $L=L_{\geq j}$.
	
	Therefore $M_{\Gamma}$ is a noetherian $S_{C}$-module, and hence $\mcA_{\Gamma}$ is a locally noetherian Grothendieck category. For any nonzero submodule $L$ of $M_{\Gamma}$, the length of $L$ is infinite, and $M_{\Gamma}/L$ has finite length. This shows that $M_{\Gamma}$ is monoform. Let $\Delta$ be the colored quiver consisting of a single vertex and no arrow. We regard $M_{\Delta}$ as an $S_{C}$-module with $M_{\Delta}s_{c_{i,i+1}}=0$ for each $i\in\mbZ_{\geq 0}$. For each $j\in\mbZ_{\geq 0}$, since $M_{\Gamma}/L_{\geq j}=\bigoplus_{i=0}^{j-1}x_{v_{i}}K$ as a $K$-vector space, we have
	\begin{equation*}
		\ASupp\frac{M_{\Gamma}}{L_{\geq j}}=\{\overline{M_{\Delta}}\},
	\end{equation*}
	by a similar argument to that in Example \ref{ex:AtomSpectrumOfGrothendieckCategoryAssociatedWithFiniteColoredQuiver}.
	
	Let $H$ be a monoform subquotient of $M_{\Gamma}$ which is not atom-equivalent to $M_{\Gamma}$. Then $H$ is a submodule of $M_{\Gamma}/L_{\geq j}$ for some $j\in\mbZ_{\geq 0}$. Since we have
	\begin{equation*}
		\overline{H}\in\ASupp\frac{M_{\Gamma}}{L_{\geq j}}=\{\overline{M_{\Delta}}\},
	\end{equation*}
	it follows that $\ASpec\mcA_{\Gamma}=\ASupp M_{\Gamma}=\{\overline{M_{\Gamma}},\overline{M_{\Delta}}\}$. Since any nonzero submodule of $M_{\Gamma}$ has a quotient module isomorphic to $M_{\Delta}$, we have $\overline{M_{\Gamma}}<\overline{M_{\Delta}}$ by Proposition \ref{Prop:TopologicalCharacterizationOfPartialOrder}.
\end{ex}

In the next definition, we introduce an operation to combine colored quivers. The new colored quiver is determined by an underlying colored quiver $\Omega$ and a family of colored quivers to be combined.

\begin{Def}\label{Def:SubstitutionOfColoredQuivers}
	Let $\Omega=(\Omega_{0},\Omega_{1},\Xi,\sigma,\tau,\mu)$ be a colored quiver, and let $\{\Gamma^{\omega}=(Q^{\omega}_{0},Q^{\omega}_{1},C^{\omega},s^{\omega},t^{\omega},u^{\omega})\}_{\omega\in\Omega_{0}}$ be a family of colored quivers indexed by the set $\Omega_{0}$ of vertices in $\Omega$. Define a colored quiver $\Omega(\{\Gamma^{\omega}\}_{\omega\in\Omega_{0}})=(Q_{0},Q_{1},C,s,t,u)$ as follows.
	\begin{enumerate}
		\item Let $Q_{0}=\coprod_{\omega\in\Omega_{0}}Q^{\omega}_{0}$.
		\item Let
		\begin{align*}
			Q_{1}&=\left(\coprod_{\omega\in\Omega_{0}}Q^{\omega}_{1}\right)\amalg\left(\coprod_{\rho\in\Omega_{1}}\{r^{\rho}_{v,v'}\mid v\in Q^{\sigma(\rho)}_{0},\ v'\in Q^{\tau(\rho)}_{0}\}\right),\\
			C&=\left(\bigcup_{\omega\in\Omega_{0}}C^{\omega}\right)\amalg\left(\bigcup_{\rho\in\Omega_{1}}\{c^{\mu(\rho)}_{v,v'}\mid v\in Q^{\sigma(\rho)}_{0},\ v'\in Q^{\tau(\rho)}_{0}\}\right).
		\end{align*}
		\item
		\begin{enumerate}
			\item For each $r\in Q^{\omega}_{1}\subset Q_{1}$, let $s(r)=s^{\omega}(r)$, $t(r)=t^{\omega}(r)$, and $u(r)=u^{\omega}(r)$.
			\item For each $r=r^{\rho}_{v,v'}\in Q_{1}$, let $s(r)=v$, $t(r)=v'$, and $u(r)=c^{\mu(\rho)}_{v,v'}$.
		\end{enumerate}
	\end{enumerate}
	
	In the case where $\Omega$ has no arrow, the colored quiver $\Omega(\{\Gamma^{\omega}\}_{\omega\in\Omega_{0}})$ is called the \emph{disjoint union} of $\{\Gamma^{\omega}\}_{\omega\in\Omega_{0}}$ and is denoted by $\coprod_{\omega\in\Omega_{0}}\Gamma^{\omega}$.
\end{Def}

\begin{ex}\label{ex:SubstitutionOfColoredQuivers}
	Let $\Omega$ be the colored quiver
	\begin{equation*}
		\xymatrix{
			\omega_{1}\ar[r]^-{(a)} & \omega_{2}\ar[r]^-{(a)} & \omega_{3}\ar[r]^-{(b)} & \omega_{4}
		}.
	\end{equation*}
	Let $\Gamma$ be the colored quiver
	\begin{equation*}
		\xymatrix{
			v\ar[d]_-{c} \\
			w
		}
	\end{equation*}
	and $\Gamma_{\omega_{1}}=\Gamma_{\omega_{2}}=\Gamma_{\omega_{3}}=\Gamma_{\omega_{4}}=\Gamma$. Then $\Omega(\{\Gamma^{\omega}\}_{\omega\in\Omega_{0}})$ is the colored quiver
	\begin{equation*}
		\xymatrix @R=12mm @C=12mm {
			v_{1}\ar[d]_-{c}\ar[r]^-{c^{(a)}_{v,v}}\ar[dr]|(0.32){c^{(a)}_{v,w}} & 
			v_{2}\ar[d]_-{c}\ar[r]^-{c^{(a)}_{v,v}}\ar[dr]|(0.32){c^{(a)}_{v,w}} & 
			v_{3}\ar[d]_-{c}\ar[r]^-{c^{(b)}_{v,v}}\ar[dr]|(0.32){c^{(b)}_{v,w}} & 
			v_{4}\ar[d]_-{c} \\
			w_{1}\ar[r]_-{c^{(a)}_{w,w}}\ar[ur]|(0.32){c^{(a)}_{w,v}} & 
			w_{2}\ar[r]_-{c^{(a)}_{w,w}}\ar[ur]|(0.32){c^{(a)}_{w,v}} & 
			w_{3}\ar[r]_-{c^{(b)}_{w,w}}\ar[ur]|(0.32){c^{(b)}_{w,v}} & 
			w_{4}\rlap{\,.}
		}
	\end{equation*}
\end{ex}

For a disjoint union $\Gamma$ of colored quivers, the atom spectrum of the Grothendieck category $\mcA_{\Gamma}$ is described as follows.

\begin{Prop}\label{Prop:AtomSpectrumOfGrothendieckCategoryAssociatedToDisjointUnion}
	Let $\{\Gamma^{\lambda}\}_{\lambda\in\Lambda}$ be a family of colored quivers, and let $\Gamma=(Q_{0},Q_{1},C)$ be the disjoint union $\coprod_{\lambda\in\Lambda}\Gamma^{\lambda}$. Then we have
	\begin{equation*}
		\ASpec\mcA_{\Gamma}=\bigcup_{\lambda\in\Lambda}\ASpec\mcA_{\Gamma^{\lambda}}
	\end{equation*}
	as a subset of $\ASpec(\Mod S_{C})$. For any $\lambda\in\Lambda$, the subset $\mcA_{\Gamma^{\lambda}}$ of $\ASpec\mcA_{\Gamma}$ is open.
\end{Prop}

\begin{proof}
	Since we have $M_{\Gamma}=\bigoplus_{\lambda\in\Lambda}M_{\Gamma^{\lambda}}$, it holds that
	\begin{equation*}
		\ASpec\mcA_{\Gamma}=\ASupp M_{\Gamma}=\bigcup_{\lambda\in\Lambda}\ASupp M_{\Gamma^{\lambda}}=\bigcup_{\lambda\in\Lambda}\ASpec\mcA_{\Gamma^{\lambda}}
	\end{equation*}
	by Proposition \ref{Prop:AtomSupportAndAssociatedAtomsAndDirectSum} (\ref{item:PropAtomSupportAndDirectSum}). By Proposition \ref{Prop:OpenBasisOfAtomSpectrum}, the subset $\ASupp M_{\Gamma^{\lambda}}$ of $\ASpec\mcA_{\Gamma}$ is open for any $\lambda\in\Lambda$.
\end{proof}

\begin{Not}\label{Not:SubstitutionOfColoredQuivers}
	We describe the colored quiver $\Omega(\{\Gamma^{\omega}\}_{\omega\in\Omega_{0}})$ in Definition \ref{Def:SubstitutionOfColoredQuivers} as the diagram of $\Omega$ by replacing each vertex $\omega$ by $\Gamma^{\omega}$ and each arrow to a \emph{bold arrow}. For example, the colored quiver $\Omega(\{\Gamma^{\omega}\}_{\omega\in\Omega_{0}})$ in Example \ref{ex:SubstitutionOfColoredQuivers} is described as
	\begin{equation*}
		\xymatrix{
			\Gamma\ar@3[r]^-{(a)} & \Gamma\ar@3[r]^-{(a)} & \Gamma\ar@3[r]^-{(b)} & \Gamma
		}.
	\end{equation*}
	
	In the case where the colors of arrows in $\Omega$ are omitted in a diagram, we assume that distinct arrows in $\Omega$ have distinct colors. For example,
	\begin{equation*}
		\xymatrix{
			\Gamma\ar@3[r] & \Gamma\ar@3[r] & \Gamma\ar@3[r] & \Gamma
		}
	\end{equation*}
	means the colored quiver
	\begin{equation*}
		\xymatrix{
			\Gamma\ar@3[r]^-{(a)} & \Gamma\ar@3[r]^-{(b)} & \Gamma\ar@3[r]^-{(c)} & \Gamma
		}.
	\end{equation*}
\end{Not}

The following results contain the key idea behind the construction of a Grothendieck category whose atom spectrum is isomorphic to a given partially ordered set.

\begin{Lem}\label{Lem:AddingMinimalAtom}
	Let $\{\Gamma^{i}=(Q^{i}_{0},Q^{i}_{1},C^{i})\}_{i=0}^{\infty}$ be a family of colored quivers with $Q^{i}_{0}\neq\emptyset$ for each $i\in\mbZ_{\geq 0}$, and let $\Gamma=(Q_{0},Q_{1},C)$ be the colored quiver
	\begin{equation*}
		\xymatrix{
			\Gamma^{0}\ar@3[r] & \Gamma^{1}\ar@3[r] & \cdots
		}.
	\end{equation*}
	\begin{enumerate}
		\item\label{item:LemDescriptionOfElementOfModuleOfCountableColoredQuiver} Let $x$ be a nonzero element of $M_{\Gamma}$ of the form
		\begin{equation*}
			x=\sum_{i=0}^{\infty}\sum_{v\in V_{i}}x_{v}q_{v},
		\end{equation*}
		where $V_{i}$ is a finite subset of $Q^{i}_{0}$ for each $i\in\mbZ_{\geq 0}$, and $q_{v}\in K\setminus\{0\}$ for each $v\in V_{i}$. Take the smallest $j\in\mbZ_{\geq 0}$ satisfying $V_{j}\neq\emptyset$, and let
		\begin{equation*}
			x'=\sum_{v\in V_{j}}x_{v}q_{v}.
		\end{equation*}
		Then we have $x'S_{C}=xS_{C}$.
		\item\label{item:LemDescriptionOfSubmoduleOfModuleOfCountableColoredQuiver} For any nonzero proper submodule $L$ of $M_{\Gamma}$, there exist $j\in\mbZ_{\geq 0}$ and a proper $S_{C^{j}}$-submodule $\widehat{L}$ of $M_{\Gamma^{j}}$ such that
		\begin{equation*}
			L=\widehat{L}\oplus\left(\bigoplus_{i=j+1}^{\infty}M_{\Gamma^{i}}\right)
		\end{equation*}
		as a $K$-vector space.
		\item\label{item:LemSerialityOfModuleOfCountableColoredQuiver} If $M_{\Gamma^{i}}$ is noetherian for any $i\in\mbZ_{\geq 0}$, then $M_{\Gamma}$ is also noetherian.
	\end{enumerate}
\end{Lem}

\begin{proof}
	For each $i\in\mbZ_{\geq 0}$, $v\in Q^{i}_{0}$, and $w\in Q^{i+1}_{0}$, denote by $c(v,w)$ the color of the arrow from $v$ to $w$ in $\Gamma$.
	
	(\ref{item:LemDescriptionOfElementOfModuleOfCountableColoredQuiver}) Let $v\in Q^{l}_{0}$, where $l\in\mbZ_{\geq 0}$ with $j<l$. Take $v'\in V_{j}$ and $w_{i}\in Q^{i}_{0}$ for each $i\in\{j+1,j+2,\ldots,l-1\}$. Then we have
	\begin{align*}
		x_{v}
		&=xs_{c(v',w_{j+1})}s_{c(w_{j+1},w_{j+2})}\cdots s_{c(w_{l-1},v)}\frac{1}{q_{v'}}\\
		&=x's_{c(v',w_{j+1})}s_{c(w_{j+1},w_{j+2})}\cdots s_{c(w_{l-1},v)}\frac{1}{q_{v'}}.
	\end{align*}
	By a similar argument to that in Example \ref{ex:AtomSpectrumOfGrothendieckCategoryAssociatedWithInfiniteColoredQuiver}, we have $x'S_{C}=xS_{C}$.
	
	(\ref{item:LemDescriptionOfSubmoduleOfModuleOfCountableColoredQuiver}) In the setting of (\ref{item:LemDescriptionOfElementOfModuleOfCountableColoredQuiver}), we have
	\begin{equation*}
		xS_{C}=x'S_{C}=x'S_{C^{j}}\oplus\left(\bigoplus_{i=j+1}^{\infty}M_{\Gamma^{i}}\right)
	\end{equation*}
	as $K$-vector spaces. Since $L$ is a sum of submodules of $M_{\Gamma}$ of the form $xS_{C}$, the assertion follows.
	
	(\ref{item:LemSerialityOfModuleOfCountableColoredQuiver}) Let $L_{0}\subset L_{1}\subset\cdots$ be an ascending chain of submodules of $M_{\Gamma}$. By (\ref{item:LemDescriptionOfSubmoduleOfModuleOfCountableColoredQuiver}), there exist $j,l\in\mbZ_{\geq 0}$ such that for any $l'\in\mbZ_{\geq 0}$ with $l\leq l'$, we have
	\begin{equation*}
		L_{l'}=\widehat{L_{l'}}\oplus\left(\bigoplus_{i=j+1}^{\infty}M_{\Gamma^{i}}\right)
	\end{equation*}
	as a $K$-vector space, where $\widehat{L_{l'}}$ is an $S_{C^{j}}$-submodule of $M_{\Gamma^{j}}$. Since we have $\widehat{L_{l}}\subset\widehat{L_{l+1}}\subset\cdots$, there exists $m\in\mbZ_{\geq 0}$ with $l\leq m$ such that $\widehat{L_{m}}=\widehat{L_{m+1}}=\cdots$ by the noetherianness of $M_{\Gamma^{j}}$. Therefore we have $L_{m}=L_{m+1}=\cdots$.
\end{proof}

\begin{Prop}\label{Prop:AddingMinimalAtom}
	Let $\{\Gamma^{i}=(Q^{i}_{0},Q^{i}_{1},C^{i})\}_{i=0}^{\infty}$ be a family of colored quivers with $Q^{i}_{0}\neq\emptyset$ for each $i\in\mbZ_{\geq 0}$, and let $\Gamma=(Q_{0},Q_{1},C)$ be the colored quiver
	\begin{equation*}
		\xymatrix{
			\Gamma^{0}\ar@3[r] & \Gamma^{1}\ar@3[r] & \cdots
		}.
	\end{equation*}
	\begin{enumerate}
		\item\label{item:PropMonoformnessOfModuleOfCountableColoredQuiver} $M_{\Gamma}$ is a monoform $S_{C}$-module, and we have
		\begin{equation*}
			\ASpec\mcA_{\Gamma}=\{\overline{M_{\Gamma}}\}\amalg\left(\bigcup_{i=0}^{\infty}\ASpec\mcA_{\Gamma^{i}}\right)
		\end{equation*}
		as a subset of $\ASpec(\Mod S_{C})$.
		\item\label{item:PropTopologyOnAtomSpectrumOfGrothendieckCategoryOfCountableColoredQuiver} For any $i\in\mbZ_{\geq 0}$, the subset $\ASpec\mcA_{\Gamma^{i}}$ of $\ASpec\mcA_{\Gamma}$ is open. The set
		\begin{equation*}
			\Biggl\{\{\overline{M_{\Gamma}}\}\amalg\Biggl(\bigcup_{j=i}^{\infty}\ASpec\mcA_{\Gamma^{j}}\Biggr)\Biggm| i\in\mbZ_{\geq 0}\Biggr\}
		\end{equation*}
		is a basis of open neighborhoods of $\overline{M_{\Gamma}}$ in $\ASpec\mcA_{\Gamma}$.
		\item\label{item:PropPartialOrderOnAtomSpectrumOfGrothendieckCategoryOfCountableColoredQuiver} Let $\alpha$ and $\beta$ be atoms in $\mcA_{\Gamma}$. Then we have $\alpha\leq\beta$ if and only if one of the following conditions holds.
		\begin{enumerate}
			\item There exists $i\in\mbZ_{\geq 0}$ such that $\alpha,\beta\in\ASpec\mcA_{\Gamma^{i}}$, and $\alpha\leq\beta$ in $\ASpec\mcA_{\Gamma^{i}}$.
			\item $\alpha=\beta=\overline{M_{\Gamma}}$.
			\item $\alpha=\overline{M_{\Gamma}}$, and $\beta\in\ASpec\mcA_{\Gamma^{i}}$ for infinitely many $i\in\mbZ_{\geq 0}$.
		\end{enumerate}
		In particular, $\overline{M_{\Gamma}}$ is minimal in $\ASpec\mcA_{\Gamma}$.
	\end{enumerate}
\end{Prop}

\begin{proof}
	(\ref{item:PropMonoformnessOfModuleOfCountableColoredQuiver}) Let $L$ be a nonzero submodule of $M_{\Gamma}$. Then by Lemma \ref{Lem:AddingMinimalAtom} (\ref{item:LemDescriptionOfSubmoduleOfModuleOfCountableColoredQuiver}), there exists an infinite sequence $c_{0},c_{1},\ldots$ of colors appearing in bold arrows such that
	\begin{equation*}
		Ls_{c_{0}}s_{c_{1}}\cdots s_{c_{i}}\neq 0
	\end{equation*}
	for each $i\in\mbZ_{\geq 0}$. On the other hand, any submodule of $M_{\Gamma}/L$ does not satisfy this property. This shows that $M_{\Gamma}$ is a monoform $S_{C}$-module.
	
	For any $i\in\mbZ_{\geq 0}$, we have $\ASpec\mcA_{\Gamma^{i}}=\ASupp M_{\Gamma^{i}}\subset\ASupp M_{\Gamma}=\ASpec\mcA_{\Gamma}$ since $M_{\Gamma^{i}}$ is a subquotient of $M_{\Gamma}$ by Remark \ref{rem:SubmoduleAssociatedToColoredFullSubquiver}. Let $H$ be a monoform subquotient of $M_{\Gamma}$ which is not atom-equivalent to $M_{\Gamma}$. Then there exists a nonzero submodule $L$ of $M_{\Gamma}$ such that $H$ is a submodule of $M_{\Gamma}/L$. There exists $i\in\mbZ_{\geq 0}$ such that $M_{\Gamma}/L$ is a quotient module of the quotient module $M_{\Gamma'}$ of $M_{\Gamma}$, where $\Gamma'$ is the colored full subquiver
	\begin{equation*}
		\xymatrix{
			\Gamma^{0}\ar@3[r] & \cdots\ar@3[r] & \Gamma^{i}
		}
	\end{equation*}
	of $\Gamma$. By Proposition \ref{Prop:AtomSupportAndAssociatedAtomsAndShortExactSequence} (\ref{item:PropAtomSupportAndShortExactSequence}), we have
	\begin{equation*}
		\overline{H}\in\ASupp(M_{\Gamma}/L)\subset\ASupp M_{\Gamma'}=\ASupp M_{\Gamma^{0}}\cup\cdots\cup\ASupp M_{\Gamma^{i}}.
	\end{equation*}
	Hence we deduce the conclusion.
	
	(\ref{item:PropTopologyOnAtomSpectrumOfGrothendieckCategoryOfCountableColoredQuiver}) By Proposition \ref{Prop:OpenBasisOfAtomSpectrum}, the subset $\ASpec\mcA_{\Gamma^{i}}=\ASupp M_{\Gamma^{i}}\subset\ASpec\mcA$ is open for each $i\in\mbZ_{\geq 0}$. By the definition of the topology on $\ASpec\mcA_{\Gamma}$, the set
	\begin{equation*}
		\{\ASupp L\mid L\text{ is a nonzero submodule of }M_{\Gamma}\}
	\end{equation*}
	is a basis of open neighborhoods of $\overline{M_{\Gamma}}$ in $\ASpec\mcA_{\Gamma}$. Hence the claim follows from Lemma \ref{Lem:AddingMinimalAtom} (\ref{item:LemDescriptionOfSubmoduleOfModuleOfCountableColoredQuiver}).
	
	(\ref{item:PropPartialOrderOnAtomSpectrumOfGrothendieckCategoryOfCountableColoredQuiver}) 	Assume that $\alpha\in\ASpec\mcA_{\Gamma^{i}}$ for some $i\in\mbZ_{\geq 0}$. Since $\ASpec\mcA_{\Gamma^{i}}$ is an open subset of $\ASpec\mcA_{\Gamma}$ by (\ref{item:PropTopologyOnAtomSpectrumOfGrothendieckCategoryOfCountableColoredQuiver}), we have $\alpha\leq\beta$ in $\ASpec\mcA_{\Gamma}$ if and only if $\beta\in\ASpec\mcA_{\Gamma^{i}}$, and $\alpha\leq\beta$ in $\ASpec\mcA_{\Gamma^{i}}$ by Proposition \ref{Prop:TopologicalCharacterizationOfPartialOrder}.
	
	Assume that $\alpha=\overline{M_{\Gamma}}$. Then $\alpha\leq\beta$ if and only if $\beta$ belongs to all the open subsets of $\ASpec\mcA_{\Gamma}$ containing $\alpha$ by Proposition \ref{Prop:TopologicalCharacterizationOfPartialOrder}. By (\ref{item:PropTopologyOnAtomSpectrumOfGrothendieckCategoryOfCountableColoredQuiver}), this is equivalent to that
	\begin{equation*}
		\beta\in\{\overline{M_{\Gamma}}\}\amalg\left(\bigcap_{i=0}^{\infty}\,\bigcup_{j=i}^{\infty}\ASpec\mcA_{\Gamma^{j}}\right).
	\end{equation*}
	Therefore the assertion holds.
\end{proof}

In Proposition \ref{Prop:ExistenceOfMaximalAtom}, it is shown that the atom spectrum of any locally noetherian Grothendieck category satisfies the ascending chain condition. The next theorem is a partial answer to the question which partially ordered sets appear as the atom spectra of locally noetherian Grothendieck categories. For an element $p$ of a partially ordered set $P$, denote by $V(p)$ the subset $\{p'\in P\mid p\leq p'\}$ of $P$. Recall that an Alexandroff space is a topological space $X$ such that the intersection of any family of open subsets of $X$ is also open.

\begin{Thm}\label{Thm:PosetOfLocallyNoetherianGrothendieckCategory}
	Let $P$ be a partially ordered set satisfying the ascending chain condition. Assume that $V(p)$ is a countable set for each $p\in P$. Then there exists a \emph{locally noetherian} Grothendieck category $\mcA$ such that $\ASpec\mcA$ is an Alexandroff space and is isomorphic to $P$ as a partially ordered set.
\end{Thm}

\begin{proof}
	Since $P$ satisfies the ascending chain condition, by recursion, we can construct a family $\{\Gamma^{p}=(Q^{p}_{0},Q^{p}_{1},C^{p})\}_{p\in P}$ of colored quivers satisfying the following conditions:
	\begin{enumerate}
		\item For any maximal element $p$ of $P$, $\Gamma^{p}$ is the colored quiver
		\begin{equation*}
			\xymatrix{
				v_{p}\ar@(dl,dr)_-{c_{p}}
			}.
		\end{equation*}
		\item For any nonmaximal element $p$ of $P$, $\Gamma^{p}$ is the colored quiver
		\begin{equation*}
			\xymatrix{
				\Gamma^{p_{0}}\ar@3[r]^-{(p;0,0)} &
				\Gamma^{p_{0}}\ar@3[r]^-{(p;1,0)} &
				\Gamma^{p_{1}}\ar@3[r]^-{(p;1,1)} &
				\Gamma^{p_{0}}\ar@3[r]^-{(p;2,0)} &
				\Gamma^{p_{1}}\ar@3[r]^-{(p;2,1)} &
				\Gamma^{p_{2}}\ar@3[r]^-{(p;2,2)} &
				\Gamma^{p_{0}}\ar@3[r]^-{(p;3,0)} &
				\cdots,
			}
		\end{equation*}
		where $\{p_{0},p_{1},\ldots\}=V(p)\setminus\{p\}$. The label on each bold arrow is of the form $(p;i,j)$, where $i,j\in\mbZ_{\geq 0}$ with $i\leq j$.
	\end{enumerate}
	For any $p\in P$, by Proposition \ref{Prop:AddingMinimalAtom} (\ref{item:PropMonoformnessOfModuleOfCountableColoredQuiver}), the $S_{C^{p}}$-module $M_{\Gamma^{p}}$ is monoform, and we have
	\begin{equation*}
		\ASpec\mcA_{\Gamma^{p}}=\{\overline{M_{\Gamma^{p}}}\}\amalg\bigcup_{p'\in V(p)\setminus\{p\}}\ASpec\mcA_{\Gamma^{p'}}.
	\end{equation*}
	By Proposition \ref{Prop:AddingMinimalAtom} (\ref{item:PropTopologyOnAtomSpectrumOfGrothendieckCategoryOfCountableColoredQuiver}), the set $\ASpec\mcA_{\Gamma^{p}}$ is the only open subset of $\ASpec\mcA_{\Gamma^{p}}$ containing $\overline{M_{\Gamma^{p}}}$. By Proposition \ref{Prop:AddingMinimalAtom} (\ref{item:PropPartialOrderOnAtomSpectrumOfGrothendieckCategoryOfCountableColoredQuiver}) and by recursion, $\ASpec \mcA_{\Gamma^{p}}$ is isomorphic to $V(p)$ as a partially ordered set.
	
	Let $\Gamma=\coprod_{p\in P}\Gamma^{p}$. Then $\mcA_{\Gamma}$ is a locally noetherian Grothendieck category since we have $M_{\Gamma}=\bigoplus_{p\in P}M_{\Gamma^{p}}$, where $M_{\Gamma^{p}}$ is a noetherian submodule of $M_{\Gamma}$. By Proposition \ref{Prop:AtomSpectrumOfGrothendieckCategoryAssociatedToDisjointUnion}, we have
	\begin{equation*}
		\ASpec\mcA_{\Gamma}=\bigcup_{p\in P}\ASpec\mcA_{\Gamma^{p}}.
	\end{equation*}
	For any $p\in P$, the subset $\ASpec\mcA_{\Gamma^{p}}$ of $\ASpec\mcA_{\Gamma}$ is smallest among open subsets of $\ASpec\mcA_{\Gamma}$ containing $p$. Therefore $\ASpec \mcA_{\Gamma}$ is an Alexandroff space and is isomorphic to $P$ as a partially ordered set.
\end{proof}

In particular, any finite partially ordered set appears as the atom spectrum of some locally noetherian Grothendieck category.

\begin{Cor}\label{Cor:FinitePosetOfLocallyNoetherianGrothendieckCategory}\leavevmode
	\begin{enumerate}
		\item\label{item:CorFiniteTopologicalSpaceOfLocallyNoetherianGrothendieckCategory} Let $X$ be a finite Kolmogorov space. Then there exists a locally noetherian Grothendieck category $\mcA$ such that $\ASpec\mcA$ is homeomorphic to $X$.
		\item\label{item:CorFinitePosetOfLocallyNoetherianGrothendieckCategory} Let $P$ be a finite partially ordered set. Then there exists a locally noetherian Grothendieck category $\mcA$ such that $\ASpec\mcA$ is isomorphic to $P$ as a partially ordered set.
	\end{enumerate}
\end{Cor}

\begin{proof}
	(\ref{item:CorFiniteTopologicalSpaceOfLocallyNoetherianGrothendieckCategory}) Let $P$ be the corresponding finite partially ordered set to $X$ by Proposition \ref{Prop:BijectionBetweenKolmogorovAlexandroffSpacesAndPartiallyOrderedSets} (\ref{item:PropBijectionBetweenFiniteKolmogorovSpacesAndFinitePartiallyOrderedSets}). By Theorem \ref{Thm:PosetOfLocallyNoetherianGrothendieckCategory}, there exists a locally noetherian Grothendieck category $\mcA$ such that $\ASpec\mcA$ is an Alexandroff space and is isomorphic to $P$ as a partially ordered set. Since both finite Kolmogorov spaces $\ASpec\mcA$ and $X$ correspond to the finite partially ordered set $P$, they are homeomorphic.
	
	(\ref{item:CorFinitePosetOfLocallyNoetherianGrothendieckCategory}) This follows directly from Theorem \ref{Thm:PosetOfLocallyNoetherianGrothendieckCategory}.
\end{proof}

\begin{rem}\label{rem:PosetOfPolynomialRing}
	There exists a locally noetherian Grothendieck category whose atom spectrum does not satisfy the additional condition in Theorem \ref{Thm:PosetOfLocallyNoetherianGrothendieckCategory}. Consider the locally noetherian Grothendieck category $\Mod\mbC[x]$. Then we have
	\begin{equation*}
		\ASpec(\Mod\mbC[x])=\{\overline{\mbC[x]}\}\cup\{\overline{S_{a}}\mid a\in\mbC\},
	\end{equation*}
	where $S_{a}=\mbC[x]/(x-a)$. Therefore $V(\overline{\mbC[x]})=\ASpec(\Mod\mbC[x])$ is uncountable.
\end{rem}

In the proof of Theorem \ref{Thm:PosetOfLocallyNoetherianGrothendieckCategory}, we constructed a locally noetherian Grothendieck category by recursion since the partially ordered set satisfies the ascending chain condition. For a general partially ordered set, we need different ideas in order to construct a Grothendieck category. The following lemma provides an important idea.

\begin{Lem}\label{Lem:GrothendieckCategoryOfCountableFamilyOfColoredQuiversWithSkippingArrows}
	Let $\Gamma'=(Q'_{0},Q'_{1},C',s',t',u')$ be a colored quiver. Define a colored quiver $\Gamma=(Q_{0},Q_{1},C,s,t,u)$ as follows.
	\begin{enumerate}
		\item Let $Q_{0}=\mbZ\times Q'_{0}$.
		\item Let $C={^{0}}C\amalg{^{1}}C\amalg{^{2}}C$ and $Q_{1}={^{0}}Q_{1}\amalg{^{1}}Q_{1}\amalg{^{2}}Q_{1}$, where
		\begin{align*}
			{^{0}}C&=C',\\
			{^{1}}C&=\{{^{1}}c_{v,w}\mid v,w\in Q'_{0}\},\\
			{^{2}}C&=\{{^{2}}c^{i}_{v,w}\mid i\in\mbZ,\ v,w\in Q'_{0}\}
		\end{align*}
		and
		\begin{align*}
			{^{0}}Q_{1}&=\mbZ\times Q'_{1},\\
			{^{1}}Q_{1}&=\{{^{1}}r^{i}_{v,w}\mid i\in\mbZ,\ v,w\in Q'_{0}\},\\
			{^{2}}Q_{1}&=\{{^{2}}r^{i}_{v,w}\mid i\in\mbZ,\ v,w\in Q'_{0}\}.
		\end{align*}
		\item
		\begin{enumerate}
			\item For $r=(i,r')\in{^{0}}Q_{1}$, let $s(r)=(i,s'(r'))$, $t(r)=(i,t'(r'))$, and $u(r)=u'(r')$.
			\item For $r={^{1}}r^{i}_{v,w}\in{^{1}}Q_{1}$, let $s(r)=(i,v)$, $t(r)=(i-1,w)$, and $u(r)={^{1}}c_{v,w}$.
			\item For $r={^{2}}r^{i}_{v,w}\in{^{2}}Q_{1}$, let $s(r)=(i,v)$, $t(r)=(i-2,w)$, and $u(r)={^{2}}c^{i}_{v,w}$.
		\end{enumerate}
	\end{enumerate}
	Let $x$ be a nonzero element of $M_{\Gamma}$ of the form
	\begin{equation*}
		x=\sum_{i\in\mbZ}\sum_{v\in V_{i}}x_{(i,v)}q_{(i,v)},
	\end{equation*}
	where $V_{i}$ is a finite subset of $Q'_{0}$ for each $i\in\mbZ_{\geq 0}$, and $q_{(i,v)}\in K\setminus\{0\}$ for each $(i,v)\in Q_{0}$. Take the largest $j\in\mbZ$ satisfying $V_{j}\neq\emptyset$, and let
	\begin{equation*}
		x'=\sum_{v\in V_{j}}x_{(j,v)}q_{j,v}.
	\end{equation*}
	Then we have $x'S_{C}=xS_{C}$.
\end{Lem}

Note that $\Gamma$ is illustrated as
\begin{equation*}
	\xymatrix{
		\cdots\ar@3[r]^-{(1)}\ar@3@/_5mm/[rr]_-{(2;2)} & \Gamma'\ar@3[r]^-{(1)}\ar@3@/_5mm/[rr]_-{(2;1)} & \Gamma'\ar@3[r]^-{(1)}\ar@3@/_5mm/[rr]_-{(2;0)} & \Gamma'\ar@3[r]^-{(1)} & \cdots.
	}
\end{equation*}

\begin{proof}
	Let
	\begin{equation*}
		L=\sum_{\substack{i\in\mbZ\\i\leq j-2}}\sum_{v\in Q'_{0}}x_{(i,v)}S_{C}
	\end{equation*}
	and take $v'\in V_{j}$. For any $w\in V_{j-1}$, by letting $c={^{1}}c_{v',w}$, we have
	\begin{equation*}
		xs_{c}\frac{1}{q_{(j,v')}}=x_{(j-1,w)}+\sum_{\substack{i\in\mbZ\\i\leq j-1}}\sum_{v\in V_{i}}x_{(i-1,w)}\frac{q_{(i,v)}}{q_{(j,v')}}\delta_{v,v'}.
	\end{equation*}
	It follows that $x_{(j-1,w)}\in xS_{C}+L$, and hence $x'\in xS_{C}+L$. Since we have
	\begin{equation*}
		x's_{c}\frac{1}{q_{(j,v')}}=x_{(j-1,w)},
	\end{equation*}
	we also have $x\in x'S_{C}+L$. Therefore we have $x'S_{C}+L=xS_{C}+L$. Since $\Gamma$ has the set ${^{2}}Q_{1}$ of arrows with distinct colors, we can deduce the conclusion by a similar argument to that in the proof of Lemma \ref{Lem:AddingMinimalAtom} (\ref{item:LemDescriptionOfElementOfModuleOfCountableColoredQuiver}).
\end{proof}

We show the main construction theorem of Grothendieck categories.

\begin{Thm}\label{Thm:PosetOfGrothendieckCategory}\leavevmode
	\begin{enumerate}
		\item\label{item:ThmTopologicalSpaceOfGrothendieckCategory} Let $X$ be a Kolmogorov Alexandroff space. Then there exists a Grothendieck category $\mcA$ such that $\ASpec\mcA$ is homeomorphic to $X$.
		\item\label{item:ThmPosetOfGrothendieckCategory} Let $P$ be a partially ordered set. Then there exists a Grothendieck category $\mcA$ such that $\ASpec\mcA$ is isomorphic to $P$ as a partially ordered set.
	\end{enumerate}
\end{Thm}

\begin{proof}
	By Proposition \ref{Prop:BijectionBetweenKolmogorovAlexandroffSpacesAndPartiallyOrderedSets} (\ref{item:PropBijectionBetweenKolmogorovAlexandroffSpacesAndPartiallyOrderedSets}), it suffices to show that for any partially ordered set $P$, there exists a Grothendieck category $\mcA$ such that $\ASpec\mcA$ is an Alexandroff space and is isomorphic to $P$ as a partially ordered set.
	
	Index $P=\{p_{\theta}\}_{\theta\in\Theta}$ by a well-ordered set $\Theta$. For a totally ordered set $T$, define another totally ordered set $\widehat{T}$ by $\widehat{T}=T\cup\{\pm\infty\}$, where $\infty$ (resp.\ $-\infty$) is supposed to be larger (resp.\ smaller) than any element of $T$. Regard
	\begin{equation*}
		\widetilde{E}=\Theta\times\prod_{j=1}^{\infty}(\widehat{\mbZ}\times\widehat{\Theta})=\Theta\times\widehat{\mbZ}\times\widehat{\Theta}\times\widehat{\mbZ}\times\widehat{\Theta}\times\cdots
	\end{equation*}
	as a totally ordered set with respect to the lexicographic order. Define a totally ordered set $E$ by
	\begin{equation*}
		E=\Biggl\{(\theta_{0},i_{1},\theta_{1},\ldots,i_{l},\theta_{l},\infty,\infty,\ldots)\Biggm|
		\begin{matrix}
			l\in\mbZ_{\geq 0},\ \theta_{0},\ldots,\theta_{l}\in\Theta,\ i_{1},\ldots,i_{l}\in\mbZ,\\
			p_{\theta_{0}}<\cdots<p_{\theta_{l}}
		\end{matrix}
		\Biggr\}
	\end{equation*}
	as a subset of $\widetilde{E}$. For each $\theta\in\Theta$, denote by $E^{\theta}$ the subset of $E$ consisting of elements $e=(\theta_{0},i_{1},\theta_{1},\ldots,i_{l},\theta_{l},\infty,\infty,\ldots)$ of $E$ with $\theta_{0}=\theta$. Let
	\begin{equation*}
		\mcF^{\theta}=\Biggl\{(\theta_{0},i_{1},\theta_{1},\ldots,i_{l-1},\theta_{l-1},i_{l})\Biggm|
		\begin{matrix}
			l\in\mbZ_{\geq 0},\ \theta_{0},\ldots,\theta_{l-1}\in\Theta,\ i_{1},\ldots,i_{l}\in\mbZ,\\
			p_{\theta_{0}}<\cdots<p_{\theta_{l-1}}<p_{\theta}
		\end{matrix}
		\Biggr\}.
	\end{equation*}
	
	Define a colored quiver $\Gamma=(Q_{0},Q_{1},C,s,t,u)$ as follows.
	\begin{enumerate}
		\item Let $Q_{0}=\{v_{e}\}_{e\in E}$.
		\item Let $C={^{0}}C\amalg{^{1}}C\amalg{^{2}}C\amalg{^{\infty}}C$ and $Q_{1}={^{0}}Q_{1}\amalg{^{1}}Q_{1}\amalg{^{2}}Q_{1}\amalg{^{\infty}}Q_{1}$, where
		\begin{align*}
			{^{0}}C^{\theta}&=\{{^{0}}c^{\theta}_{e,e'}\mid\theta',\theta''\in\Theta,\ e\in E^{\theta'},\ e'\in E^{\theta''},\ \theta''<\theta',\ p_{\theta}<p_{\theta'},\ p_{\theta}<p_{\theta''}\},\\
			{^{1}}C^{\theta}&=\{{^{1}}c^{\theta}_{e,e'}\mid\theta',\theta''\in\Theta,\ e\in E^{\theta'},\ e'\in E^{\theta''},\ p_{\theta}<p_{\theta'},\ p_{\theta}<p_{\theta''}\},\\
			{^{2}}C^{\theta}&=\{{^{2}}c^{\theta,i}_{e,e'}\mid\theta'\in\Theta,\ i\in\mbZ,\ e\in E^{\theta'},\ e'\in E^{\theta''},\ p_{\theta}<p_{\theta'},\ p_{\theta}<p_{\theta''}\},\\
			{^{\infty}}C^{\theta}&=\{{^{\infty}}c^{\theta,i}_{e}\mid\theta'\in\Theta,\ i\in\mbZ,\ e\in E^{\theta'},\ p_{\theta}<p_{\theta'}\}\cup\{{^{\infty}}c^{\theta,\infty}\}
		\end{align*}
		for each $\theta\in\Theta$,
		\begin{align*}
			{^{0}}C&=\coprod_{\theta\in\Theta}{^{0}}C^{\theta},\\
			{^{1}}C&=\coprod_{\theta\in\Theta}{^{1}}C^{\theta},\\
			{^{2}}C&=\coprod_{\theta\in\Theta}{^{2}}C^{\theta},\\
			{^{\infty}}C&=\coprod_{\theta\in\Theta}{^{\infty}}C^{\theta},
		\end{align*}
		and
		\begin{align*}
			{^{0}}Q_{1}&=\coprod_{\theta\in\Theta}\{{^{0}}r^{c}_{f,i}\mid c\in{^{0}}C^{\theta},\ f\in\mcF^{\theta},\ i\in\mbZ\},\\
			{^{1}}Q_{1}&=\coprod_{\theta\in\Theta}\{{^{1}}r^{c}_{f,i}\mid c\in{^{1}}C^{\theta},\ f\in\mcF^{\theta},\ i\in\mbZ\},\\
			{^{2}}Q_{1}&=\coprod_{\theta\in\Theta}\{{^{2}}r^{c}_{f}\mid c\in{^{2}}C^{\theta},\ f\in\mcF^{\theta}\},\\
			{^{\infty}}Q_{1}&=\coprod_{\theta\in\Theta}\{{^{\infty}}r^{c}_{f}\mid c\in{^{\infty}}C^{\theta},\ f\in\mcF^{\theta}\}.
		\end{align*}
		\item
		\begin{enumerate}
			\item For $r={^{0}}r^{c}_{f,i}\in {^{0}}Q_{1}$ with $c={^{0}}c^{\theta}_{e,e'}$, let $s(r)=v_{(f,\theta,i,e)}$, $t(r)=v_{(f,\theta,i,e')}$, and $u(r)=c$.
			\item For $r={^{1}}r^{c}_{f,i}\in {^{1}}Q_{1}$ with $c={^{1}}c^{\theta}_{e,e'}$, let $s(r)=v_{(f,\theta,i,e)}$, $t(r)=v_{(f,\theta,i-1,e')}$, and $u(r)=c$.
			\item For $r={^{2}}r^{c}_{f}\in {^{2}}Q_{1}$ with $c={^{2}}c^{\theta,i}_{e,e'}$, let $s(r)=v_{(f,\theta,i,e)}$, $t(r)=v_{(f,\theta,i-2,e')}$, and $u(r)=c$.
			\item
			\begin{enumerate}
				\item For $r={^{\infty}}r^{c}_{f}\in {^{\infty}}Q_{1}$ with $c={^{\infty}}c^{\theta,i}_{e}$, let $s(r)=v_{(f,\theta,\infty,\infty,\ldots)}$, $t(r)=v_{(f,\theta,i,e)}$, and $u(r)=c$.
				\item For $r={^{\infty}}r^{c}_{f}\in {^{\infty}}Q_{1}$ with $c={^{\infty}}c^{\theta,\infty}$, let $s(r)=t(r)=v_{(f,\theta,\infty,\infty,\ldots)}$, and $u(r)=c$.
			\end{enumerate}
		\end{enumerate}
	\end{enumerate}
	For any $\theta\in\Theta$, denote by $\Gamma^{\theta}$ the colored full subquiver of $\Gamma$ corresponding to the subset $\{v_{e}\}_{e\in E^{\theta}}$ of $Q_{0}$. Since we have
	\begin{equation*}
		E^{\theta}=\{(\theta,\infty,\infty,\ldots)\}\amalg\coprod_{i\in\mbZ}\coprod_{\substack{\theta'\in\Theta\\p_{\theta}<p_{\theta'}}}E^{\theta'},
	\end{equation*}
	$\Gamma^{\theta}$ is illustrated as follows.
	\begin{equation*}
		\Gamma^{\theta}\colon
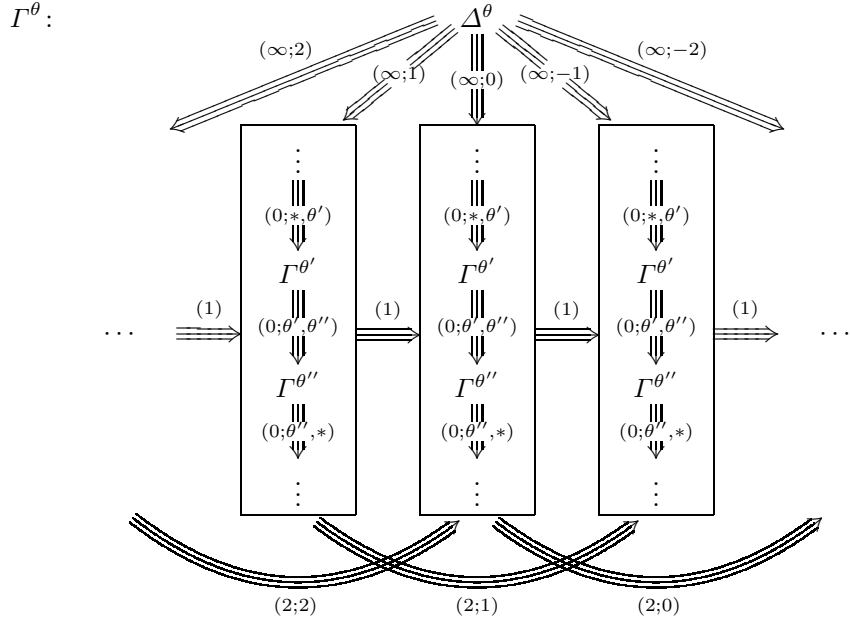
\xymatrix @R=3mm @W=13mm @L=2mm {
			& & \Delta^{\theta}\ar@3@<-2mm>[dddll]_-{(\infty;2)}\ar@3@<-1mm>[dddl]|-{(\infty;1)}\ar@3[ddd]|-{(\infty;0)}\ar@3@<1mm>[dddr]|-{(\infty;-1)}\ar@3@<2mm>[dddrr]^-{(\infty;-2)} & & \\
			& & & & \\
			& & & & \\
			\save[].[dd].[dddd].[dddddd]!C="g0"\frm{}\restore\phantom{\vdots} &
			\save[].[dd].[dddd].[dddddd]!C="g1"*[F]\frm{}\restore\vdots\ar@3[dd]|-{(0;*,\theta')} &
			\save[].[dd].[dddd].[dddddd]!C="g2"*[F]\frm{}\restore\vdots\ar@3[dd]|-{(0;*,\theta')} &
			\save[].[dd].[dddd].[dddddd]!C="g3"*[F]\frm{}\restore\vdots\ar@3[dd]|-{(0;*,\theta')} &
			\save[].[dd].[dddd].[dddddd]!C="g4"\frm{}\restore\phantom{\vdots} \\
			& & & & \\
			&
			\Gamma^{\theta'}\ar@3[dd]|-{(0;\theta',\theta'')} &
			\Gamma^{\theta'}\ar@3[dd]|-{(0;\theta',\theta'')} &
			\Gamma^{\theta'}\ar@3[dd]|-{(0;\theta',\theta'')} &
			\\
			\cdots & & & & \cdots
			\\
			&
			\Gamma^{\theta''}\ar@3[dd]|-{(0;\theta'',*)} &
			\Gamma^{\theta''}\ar@3[dd]|-{(0;\theta'',*)} &
			\Gamma^{\theta''}\ar@3[dd]|-{(0;\theta'',*)} &
			\\
			& & & & \\
			\ar@3@<-1mm>@/_10mm/[rr]_-{(2;2)} &
			\vdots\ar@3@<-1mm>@/_10mm/[rr]_-{(2;1)} &
			\vdots\ar@3@<-1mm>@/_10mm/[rr]_-{(2;0)} &
			\vdots &
			\ar@3@<-2mm>^{(1)}"g0";"g1"
			\ar@3@<-2mm>^{(1)}"g1";"g2"
			\ar@3@<-2mm>^{(1)}"g2";"g3"
			\ar@3@<-2mm>^{(1)}"g3";"g4"
		}
	\end{equation*}
	where $\Delta^{\theta}$ is the colored quiver
	\begin{equation*}
		\xymatrix{
			v_{(\theta,\infty,\infty,\ldots)}\ar@(dl,dr)_-{(\infty;\infty)}
		}.
	\end{equation*}
	Note that for any $e,e'\in E^{\theta}$, we have $x_{v_{e}}S_{C}\subset x_{v_{e'}}S_{C}$ if and only if $e\leq e'$.
	
	Let $x$ be a nonzero element of $M_{\Gamma^{\theta}}$. Then there exist $e(0),\ldots,e(n)\in E^{\theta}$ satisfying $e(0)>\cdots>e(n)$ and
	\begin{equation*}
		x=\sum_{j=0}^{n}x_{v_{e(j)}}q_{j},
	\end{equation*}
	where $q_{0},\ldots,q_{n}\in K\setminus\{0\}$. We show that $xS_{C}=x_{v_{e(0)}}S_{C}$ by induction on $n$. Assume that $n\geq 1$. For each $j=0,\ldots,n$, let
	\begin{equation*}
		e(j)=(\theta,i^{(j)}_{1},\theta^{(j)}_{1},i^{(j)}_{2},\theta^{(j)}_{2},\ldots).
	\end{equation*}
	Take the smallest $m\in\mbZ_{\geq 1}$ satisfying $(i^{(0)}_{m},\theta^{(0)}_{m})\neq (i^{(n)}_{m},\theta^{(n)}_{m})$. Then by Lemma \ref{Lem:GrothendieckCategoryOfCountableFamilyOfColoredQuiversWithSkippingArrows} and Lemma \ref{Lem:AddingMinimalAtom} (\ref{item:LemDescriptionOfElementOfModuleOfCountableColoredQuiver}), there exists $n'\in\{0,\ldots,n\}$ such that $(i^{(0)}_{m},\theta^{(0)}_{m})=\cdots=(i^{(n')}_{m},\theta^{(n')}_{m})$, and
	\begin{equation*}
		xS_{C}=\left(\sum_{j=0}^{n'}x_{v_{e(j)}}q_{j}\right)S_{C}.
	\end{equation*}
	Let $\theta'=\theta^{(0)}_{m}=\cdots=\theta^{(n')}_{m}$. Then $\sum_{j=0}^{n'}x_{v_{e(j)}}q_{j}$ can be regarded as an element of $M_{\Gamma^{\theta'}}$, and hence by the induction hypothesis, we obtain $xS_{C}=x_{v_{e(0)}}q_{0}S_{C}=x_{v_{e(0)}}S_{C}$.
	
	By a similar argument to that in the proof of Proposition \ref{Prop:AddingMinimalAtom} (\ref{item:PropMonoformnessOfModuleOfCountableColoredQuiver}), it follows that $M_{\Gamma^{\theta}}$ is a monoform $S_{C}$-module. The $S_{C}$-module $M_{\Delta^{\theta}}$ is simple. We show that
	\begin{equation*}
		\ASupp M_{\Gamma^{\theta}}=\{\overline{M_{\Gamma^{\theta'}}}\mid\theta'\in\Theta,\ p_{\theta}<p_{\theta'}\}\cup\{\overline{M_{\Delta^{\theta'}}}\mid\theta'\in\Theta,\ p_{\theta}<p_{\theta'}\}.
	\end{equation*}
	
	For each $\widetilde{e}=(\theta_{0},i_{1},\theta_{1},\ldots)\in\widetilde{E}$ $(=\Theta\times\prod_{j=1}^{\infty}(\widehat{\mbZ}\times\widehat{\Theta}))$ with $\theta_{0}=\theta$, define a submodule $L_{<\widetilde{e}}$ of $M_{\Gamma^{\theta}}$ by
	\begin{equation*}
		L_{<\widetilde{e}}=\sum_{\substack{e\in E^{\theta}\\e<\widetilde{e}}}x_{v_{e}}S_{C}.
	\end{equation*}
	Since $E$ is a totally ordered set, and any subset of $E^{\theta}$ has a supremum in $\widetilde{E}$, any nonzero proper submodule $L$ of $M_{\Gamma^{\theta}}$ is of one of the following forms.
	\begin{enumerate}
		\item $L=x_{v_{e}}S_{C}$ for some $e\in E^{\theta}\setminus\{(\theta,\infty,\infty,\ldots)\}$.
		\item $L=L_{<e}$ for some $e\in E^{\theta}$.
		\item $L=L_{<\widetilde{e}}$ for some $\widetilde{e}=(\theta,i_{1},\theta_{1},\ldots,i_{l},\theta_{l},i_{l+1},\infty,-\infty,-\infty,\ldots)\in\widetilde{E}$, where $l\in\mbZ_{\geq 0}$, $\theta_{1},\ldots,\theta_{l}\in\Theta$, $i_{1},\ldots,i_{l+1}\in\mbZ$, and $p_{\theta}<p_{\theta_{1}}<\cdots<p_{\theta_{l}}$.
		\item $L=L_{<\widetilde{e}}$ for some $\widetilde{e}=(\theta,i_{1},\theta_{1},\ldots,i_{l},\theta_{l},i_{l+1},\theta_{l+1},-\infty,-\infty,\ldots)\in\widetilde{E}$, where $l\in\mbZ_{\geq 0}$, $\theta_{1},\ldots,\theta_{l+1}\in\Theta$, $i_{1},\ldots,i_{l+1}\in\mbZ$, and $p_{\theta}<p_{\theta_{1}}<\cdots<p_{\theta_{l+1}}$.
		\item $L=L_{<\widetilde{e}}$ for some $\widetilde{e}=(\theta,i_{1},\theta_{1},i_{2},\theta_{2},\ldots)\in\widetilde{E}$, where $\theta_{1},\theta_{2},\ldots\in\Theta$, $i_{1},i_{2},\ldots\in\mbZ$, and $p_{\theta}<p_{\theta_{1}}<\cdots$.
	\end{enumerate}
	In the first four cases, $M_{\Gamma^{\theta}}/L$ has an essential submodule isomorphic to $M_{\Gamma^{\theta'}}$ or $M_{\Delta^{\theta'}}$ for some $\theta'\in\Theta$ with $p_{\theta}<p_{\theta'}$. Hence by Proposition \ref{Prop:AssociatedAtomsOfEssentialSubobject} and Proposition \ref{Prop:NumberOfAssociatedAtoms} (\ref{item:PropNumberOfAssociatedAtomsOfUniformObject}), we have $\AAss(M_{\Gamma^{\theta}}/L)=\AAss M_{\Gamma^{\theta'}}=\{\overline{M_{\Gamma^{\theta'}}}\}$ or $\AAss(M_{\Gamma^{\theta}}/L)=\AAss M_{\Delta^{\theta'}}=\{\overline{M_{\Delta^{\theta'}}}\}$.
	
	In the fifth case, we show that any nonzero submodule $L'/L$ of $M_{\Gamma^{\theta}}/L$ is not monoform. By Proposition \ref{Prop:PropertyOfMonoformObject} (\ref{item:PropSubobjectOfMonoformObjectIsMonoform}), we can assume that $L'=x_{v_{e}}S_{C}$, where $e=(\theta,i'_{1},\theta'_{1},\ldots,i'_{l},\theta'_{l},\infty,\infty,\ldots)\in E^{\theta}$ with $\widetilde{e}<e$. Define $\widetilde{e}',\widetilde{e}''\in\widetilde{E}$ by
	\begin{equation*}
		\widetilde{e}'=(\theta,i_{1},\theta_{1},\ldots,i_{l},\theta_{l},i_{l+1}+1,\theta_{l+1},i_{l+2},\theta_{l+2},\ldots)
	\end{equation*}
	and
	\begin{equation*}
		\widetilde{e}''=(\theta,i_{1},\theta_{1},\ldots,i_{l},\theta_{l},i_{l+1}+2,\theta_{l+1},i_{l+2},\theta_{l+2},\ldots).
	\end{equation*}
	Then $L_{<\widetilde{e}'}/L_{<\widetilde{e}}$ is isomorphic to $L_{<\widetilde{e}''}/L_{<\widetilde{e}'}$ since these $S_{C}$-modules are annihilated by $s_{c}$ for any $c={^{2}c^{\theta_{l},i}_{e',e''}}\in{^{2}C}$, where $i$, $e'$, and $e''$ are arbitrary. This shows that $L'/L$ is not monoform.
	
	Hence we have
	\begin{equation*}
		\ASpec\mcA_{\Gamma^{\theta}}=\ASupp M_{\Gamma^{\theta}}=\{\overline{M_{\Gamma^{\theta'}}}\mid\theta'\in\Theta,\ p_{\theta}\leq p_{\theta'}\}\cup\{\overline{M_{\Delta^{\theta'}}}\mid\theta'\in\Theta,\ p_{\theta}\leq p_{\theta'}\},
	\end{equation*}
	and $\ASupp L=\ASupp M_{\Gamma^{\theta}}$ for any nonzero submodule $L$ of $M_{\Gamma}$. Since we have $\Gamma=\coprod_{\theta\in\Theta}\Gamma^{\theta}$, it holds that
	\begin{equation*}
		\ASpec\mcA_{\Gamma}=\bigcup_{\theta\in\Theta}\ASpec\mcA_{\Gamma^{\theta}}=\{\overline{M_{\Gamma^{\theta}}}\mid\theta\in\Theta\}\cup\{\overline{M_{\Delta^{\theta}}}\mid\theta\in\Theta\}
	\end{equation*}
	by Proposition \ref{Prop:AtomSpectrumOfGrothendieckCategoryAssociatedToDisjointUnion}. By a similar argument to that in the proof of Theorem \ref{Thm:PosetOfLocallyNoetherianGrothendieckCategory}, the topological space $\ASpec\mcA$ is Alexandroff. It is easy to verify the following properties.
	\begin{enumerate}
		\item For any distinct $\theta,\theta'\in\Theta$, we have $\overline{M_{\Gamma^{\theta}}}\neq\overline{M_{\Gamma^{\theta'}}}$ and $\overline{M_{\Delta^{\theta}}}\neq\overline{M_{\Delta^{\theta'}}}$.
		\item For any $\theta,\theta'\in\Theta$, we have $\overline{M_{\Gamma^{\theta}}}=\overline{M_{\Delta^{\theta'}}}$ if and only if $\theta=\theta'$, and $p_{\theta}=p_{\theta'}$ is maximal in $P$.
		\item For any $\theta\in\Theta$, we have
		\begin{equation*}
			V(\overline{M_{\Gamma^{\theta}}})=\{\overline{M_{\Gamma^{\theta'}}}\mid\theta'\in\Theta,\ p_{\theta}\leq p_{\theta'}\}\cup\{\overline{M_{\Delta^{\theta'}}}\mid\theta'\in\Theta,\ p_{\theta}<p_{\theta'}\}
		\end{equation*}
		and
		\begin{equation*}
			V(\overline{M_{\Delta^{\theta}}})=\{\overline{M_{\Delta^{\theta}}}\}.
		\end{equation*}
	\end{enumerate}
	
	Let $\Phi=\{\overline{M_{\Delta^{\theta}}}\mid p_{\theta}\text{ is not maximal in }P\}$. Since $M_{\Delta^{\theta}}$ is a simple $S_{C}$-module for any $\theta\in\Theta$, the subset $\Phi$ of $\ASpec\mcA_{\Gamma}$ is open. Hence we have $\ASupp(\ASupp^{-1}\Phi)=\Phi$ by Lemma \ref{Lem:OpenSubclassOfAtomSpectrumAndAtomSupport}. Let $\mcA=\mcA_{\Gamma}/\ASupp^{-1}\Phi$. Then by Theorem \ref{Thm:AtomSpectraOfQuotientCategory}, the topological space $\ASpec\mcA$ is Alexandroff, and is isomorphic to $P$ as a partially ordered set.
\end{proof}

%%%%%%%%%%%%%%%%%%%%%%%%%%%%%%%%%%%%%%%%%%%%%%%%%%%%%%%%%%%%%%%%%%%%%%%%%%%%%%%%
\section{Examples of Grothendieck categories}
\label{sec:ExamplesOfGrothendieckCategories}
%%%%%%%%%%%%%%%%%%%%%%%%%%%%%%%%%%%%%%%%%%%%%%%%%%%%%%%%%%%%%%%%%%%%%%%%%%%%%%%%

In this section, by using methods in section \ref{sec:ConstructionOfGrothendieckCategories}, we introduce several Grothendieck categories which have remarkable structures.

For a commutative noetherian ring $R$ and an $R$-module $M$, we have
\begin{equation*}
	\Supp M=\{\mfp\in\Spec R\mid\mfq\subset\mfp\text{ for some }\mfq\in\Ass M\}.
\end{equation*}
In particular, the set $\Supp M$ is determined by $\Ass M$ and the partial order on $\Spec R$. The following proposition shows strongly that we cannot show this kind of result for a general locally noetherian Grothendieck categories.

\begin{Prop}\label{Prop:NoConnectionBetweenAAssAndASupp}
	There exist a locally noetherian Grothendieck category $\mcA$ and an object $M$ in $\mcA$ such that $\ASpec\mcA=\{\alpha,\beta\}$, $\alpha<\beta$, $\AAss M=\{\beta\}$, and $\ASupp M=\{\alpha,\beta\}$.
\end{Prop}

\begin{proof}
	Let $\Delta$ be a colored quiver consisting of a vertex with no arrow, $\Gamma'$ the colored quiver
	\begin{equation*}
		\xymatrix{
			\Delta\ar@3[r] & \Delta\ar@3[r] & \cdots,
		}
	\end{equation*}
	and $\Gamma=(Q_{0},Q_{1},C)$ the colored quiver
	\begin{equation*}
		\xymatrix{
			\Gamma'\ar@3[r] & \Delta
		}.
	\end{equation*}
	Note that $\Gamma'$ is the colored quiver denoted by $\Gamma$ in Example \ref{ex:AtomSpectrumOfGrothendieckCategoryAssociatedWithInfiniteColoredQuiver}. By Remark \ref{rem:SubmoduleAssociatedToColoredFullSubquiver}, we have an exact sequence
	\begin{equation*}
		0\to M_{\Delta}\to M_{\Gamma}\to M_{\Gamma'}\to 0.
	\end{equation*}
	By Proposition \ref{Prop:AddingMinimalAtom} (or by the argument in Example \ref{ex:AtomSpectrumOfGrothendieckCategoryAssociatedWithInfiniteColoredQuiver}), the $S_{C}$-module $M_{\Gamma'}$ is monoform, and $\ASupp M_{\Gamma'}=\{\overline{M_{\Gamma'}},\overline{M_{\Delta}}\}$ with $\overline{M_{\Gamma'}}<\overline{M_{\Delta}}$. By Proposition \ref{Prop:AtomSupportAndAssociatedAtomsAndShortExactSequence} (\ref{item:PropAtomSupportAndShortExactSequence}), we have
	\begin{equation*}
		\ASpec\mcA_{\Gamma}=\ASupp M_{\Gamma}=\ASupp M_{\Gamma'}\cup\ASupp M_{\Delta}=\{\overline{M_{\Gamma'}},\overline{M_{\Delta}}\}.
	\end{equation*}
	Since $M_{\Gamma}$ has an essential submodule isomorphic to $M_{\Delta}$, we have $\AAss M_{\Gamma}=\AAss M_{\Delta}=\{\overline{M_{\Delta}}\}$ by Proposition \ref{Prop:AssociatedAtomsOfEssentialSubobject}.
\end{proof}

For a locally noetherian Grothendieck category $\mcA$ satisfying $\ASpec\mcA=\ASupp M$ for some noetherian object $M$ in $\mcA$ (that is, $\ASpec\mcA$ is compact by Proposition \ref{Prop:CharacterizationOfOpenCompactSubset}), there exists a minimal atom in $\mcA$ by Proposition \ref{Prop:ExistenceOfMinimalAtomInOpenCompactSubset}. This does not necessarily hold for a general locally noetherian Grothendieck category.

\begin{Prop}\label{Prop:LocallyNoetherianGrothendieckCategoryWithNoMinimalAtom}
	There exists a nonzero locally noetherian Grothendieck category with no minimal atom.
\end{Prop}

\begin{proof}
	By Theorem \ref{Thm:PosetOfLocallyNoetherianGrothendieckCategory}, there exists a locally noetherian Grothendieck category $\mcA$ such that $\ASpec\mcA$ is an Alexandroff space and is isomorphic to the partially ordered set $\{p_{i}\mid i\in\mbZ_{\geq 0}\}$ with $p_{0}>p_{1}>\cdots$.
\end{proof}

For a commutative noetherian ring $R$, the height of any prime ideal is finite by Krull's height theorem. For a locally noetherian Grothendieck category $\mcA$, the atom spectrum of $\mcA$ does not necessarily satisfy the descending chain condition even in the case where $\ASpec\mcA=\ASupp M$ for some noetherian object $M$ in $\mcA$.

\begin{Prop}\label{Prop:LocallyNoetherianGrothendieckCategoryWithNoDescendingChainCondition}
	There exists a locally noetherian Grothendieck category $\mcA$ and a noetherian object $M$ in $\mcA$ such that $\ASpec\mcA=\ASupp M$, but $\ASpec\mcA$ does not satisfy the descending chain condition as a partially ordered set.
\end{Prop}

\begin{proof}
	By Theorem \ref{Thm:PosetOfLocallyNoetherianGrothendieckCategory}, there exists a locally noetherian Grothendieck category $\mcA$ such that $\ASpec\mcA$ is an Alexandroff space and is isomorphic to the partially ordered set $\{p_{i}\mid i\in\mbZ_{\geq 0}\}\cup\{p_{\infty}\}$ with $p_{0}>p_{1}>\cdots$ and $p_{i}>p_{\infty}$ for each $i\in\mbZ_{\geq 0}$. Let $H$ be a noetherian monoform object in $\mcA$ such that $\overline{H}$ is the atom in $\mcA$ corresponding to $p_{\infty}$. Then by Proposition \ref{Prop:TopologicalCharacterizationOfPartialOrder}, we have $\ASpec\mcA=\ASupp H$.
\end{proof}

For a commutative ring $R$, the topological space $\ASpec(\Mod R)$ is Alexandroff by Proposition \ref{Prop:TopologyOnAtomSpectrumOfCommutativeRing}. Hence the set of maximal (resp.\ minimal) elements of $\ASpec(\Mod R)$ is an open (resp.\ closed) subset of $\ASpec(\Mod R)$ by Proposition \ref{Prop:TopologicalCharacterizationOfMaximalAtomAndMinimalAtom}. However, this does not necessarily hold for a Grothendieck category $\mcA$ even in the case where $\mcA$ is locally noetherian.

\begin{Prop}\label{Prop:MaximalAtomsDoNotFormOpenSubsetMinimalAtomsDoNotFormClosedSubset}\leavevmode
	\begin{enumerate}
		\item\label{item:PropMaximalAtomsDoNotFormOpenSubset} There exists a locally noetherian Grothendieck category $\mcA$ such that the set of maximal elements of $\ASpec\mcA$ is not an open subset of $\ASpec\mcA$.
		\item\label{item:PropMinimalAtomsDoNotFormClosedSubset} There exists a locally noetherian Grothendieck category $\mcA$ such that the set of minimal elements of $\ASpec\mcA$ is not a closed subset of $\ASpec\mcA$.
	\end{enumerate}
\end{Prop}

\begin{proof}
	(\ref{item:PropMaximalAtomsDoNotFormOpenSubset}) For each $i\in\mbZ_{\geq 0}$, let $\Delta^{i}$ be the colored quiver
	\begin{equation*}
		\xymatrix{
			v_{i}\ar@(dl,dr)_-{c_{i}}
		},
	\end{equation*}
	$\Gamma^{i}$ the colored quiver
	\begin{equation*}
		\xymatrix{
			\Delta^{i}\ar@3[r] & \Delta^{i}\ar@3[r] & \cdots,
		}
	\end{equation*}
	and $\Gamma$ the colored quiver
	\begin{equation*}
		\xymatrix{
			\Gamma^{0}\ar@3[r] & \Gamma^{1}\ar@3[r] & \cdots.
		}
	\end{equation*}
	Then by Proposition \ref{Prop:AddingMinimalAtom}, we have
	\begin{equation*}
		\ASpec\mcA_{\Gamma}=\{\overline{M_{\Gamma}}\}\cup\{\overline{M_{\Gamma^{i}}}\mid i\in\mbZ_{\geq 0}\}\cup\{\overline{M_{\Delta^{i}}}\mid i\in\mbZ_{\geq 0}\}
	\end{equation*}
	with $\overline{M_{\Gamma^{i}}}<\overline{M_{\Delta^{i}}}$ for each $i\in\mbZ_{\geq 0}$. The set of maximal elements of $\ASpec\mcA_{\Gamma}$ is
	\begin{equation*}
		\Psi=\{\overline{M_{\Gamma}}\}\cup\{\overline{M_{\Delta^{i}}}\mid i\in\mbZ_{\geq 0}\}.
	\end{equation*}
	Since any nonzero submodule of $M_{\Gamma}$ has a subquotient isomorphic to $M_{\Gamma^{i}}$ for some $i\in\mbZ_{\geq 0}$, the subset $\Psi$ of $\ASpec\mcA_{\Gamma}$ is not open.
	
	(\ref{item:PropMinimalAtomsDoNotFormClosedSubset}) For each $i\in\mbZ_{\geq 0}$, let $\Delta^{i}$ be the colored quiver
	\begin{equation*}
		\xymatrix{
			v_{i}\ar@(dl,dr)_-{c_{i}}
		},
	\end{equation*}
	$\Gamma^{i}$ the colored quiver
	\begin{equation*}
		\xymatrix{
			\Delta^{i}\ar@3[r] & \Delta^{i+1}\ar@3[r] & \cdots
		},
	\end{equation*}
	and $\Gamma$ the colored quiver
	\begin{equation*}
		\xymatrix{
			\Gamma^{0}\ar@3[r] & \Gamma^{1}\ar@3[r] & \cdots
		}.
	\end{equation*}
	Then by Proposition \ref{Prop:AddingMinimalAtom}, we have
	\begin{equation*}
		\ASpec\mcA_{\Gamma}=\{\overline{M_{\Gamma}},\overline{M_{\Gamma^{0}}}\}\cup\{\overline{M_{\Delta^{i}}}\mid i\in\mbZ_{\geq 0}\}
	\end{equation*}
	with $\overline{M_{\Gamma}}<\overline{M_{\Gamma^{0}}}$ since it holds that $\overline{M_{\Gamma^{0}}}=\overline{M_{\Gamma^{1}}}=\cdots$. The set of minimal elements of $\ASpec\mcA_{\Gamma}$ is
	\begin{equation*}
		\Psi=\{\overline{M_{\Gamma}}\}\cup\{\overline{M_{\Delta^{i}}}\mid i\in\mbZ_{\geq 0}\}.
	\end{equation*}
	Since any nonzero submodule of $M_{\Gamma^{0}}$ has a subquotient isomorphic to $M_{\Delta^{i}}$ for some $i\in\mbZ_{\geq 0}$, the the subset $\Psi$ of $\ASpec\mcA_{\Gamma}$ is not closed.
\end{proof}

The \emph{injective spectrum} of a Grothendieck category $\mcA$ is the set of isomorphism classes of indecomposable injective objects in $\mcA$. It was investigated by \cite{Gabriel}, \cite{Herzog}, \cite{Krause}, and \cite{Pappacena}, for example. It is known that there exists a canonical injection from the atom spectrum of $\mcA$ to the injective spectrum of $\mcA$ (see \cite[Lemma 5.8]{Kanda1}). In the rest of this section, we construct a Grothendieck category which has empty atom spectrum but has nonempty injective spectrum.

\begin{Thm}\label{Thm:GrothendieckCategoryWithNoAtom}
	There exists a Grothendieck category $\mcA$ such that $\mcA$ has no atom but has at least one indecomposable injective object.
\end{Thm}

\begin{proof}
	Define a totally ordered set $\widehat{\mbZ}$ by $\widehat{\mbZ}=\mbZ\cup\{-\infty\}$, where $-\infty$ is supposed to be smaller than any element of $\mbZ$. Define a totally ordered set $E$ by
	\begin{equation*}
		E=\{(i_{0},\ldots,i_{l},-\infty,-\infty,\ldots)\mid l\in\mbZ_{\geq 0},\ i_{0},\ldots,i_{l}\in\mbZ_{\geq 0},\ i_{0}<\cdots<i_{l}\}
	\end{equation*}
	as a subset of $\widetilde{E}=\prod_{j=0}^{\infty}\widehat{\mbZ}$ with the lexicographic order. For each $i\in\mbZ_{\geq 0}$, denote by $E^{i}$ the subset of $E$ consisting of elements $e=(i_{0},\ldots,i_{l},-\infty,-\infty,\ldots)$ of $E$ with $i_{0}=i$. Let
	\begin{equation*}
		\mcF^{i}=\{(i_{0},\ldots,i_{l})\mid l\in\mbZ_{\geq 0},\ i_{0},\ldots,i_{l}\in\mbZ_{\geq 0},\ i_{0}<\cdots<i_{l}<i\}.
	\end{equation*}
	
	Define a colored quiver $\Gamma=(Q_{0},Q_{1},C,s,t,u)$ as follows.
	\begin{enumerate}
		\item Let $Q_{0}=\{v_{e}\}_{e\in E}$.
		\item Let $C={^{1}}C\amalg{^{\infty}}C$ and $Q_{1}={^{1}}Q_{1}\amalg{^{\infty}}Q_{1}$, where
		\begin{align*}
			{^{1}}C^{i}&=\{{^{1}}c_{e,e'}\mid e\in E^{i},\ e'\in E^{i+1}\},\\
			{^{\infty}}C^{i}&=\{{^{\infty}}c^{i}_{e}\mid i'\in\mbZ_{\geq 0},\ e\in E^{i'},\ i<i'\}\cup\{{^{\infty}}c^{i}_{-\infty}\}
		\end{align*}
		for each $i\in\mbZ_{\geq 0}$,
		\begin{align*}
			{^{1}}C&=\coprod_{i=0}^{\infty}{^{1}}C^{i},\\
			{^{\infty}}C&=\coprod_{i=0}^{\infty}{^{\infty}}C^{i},
		\end{align*}
		and
		\begin{align*}
			{^{1}}Q_{1}&=\coprod_{i=0}^{\infty}\{{^{1}}r^{c}_{f}\mid c\in{^{1}}C^{i},\ f\in\mcF^{i}\},\\
			{^{\infty}}Q_{1}&=\coprod_{i=0}^{\infty}\{{^{\infty}}r^{c}_{f}\mid c\in{^{\infty}}C^{i},\ f\in\mcF^{i}\}.
		\end{align*}
		\item
		\begin{enumerate}
			\item For $r={^{1}}r^{c}_{f}\in {^{1}}Q_{1}$ with $c={^{1}}c_{e,e'}$, let $s(r)=v_{(f,e)}$, $t(r)=v_{(f,e')}$, and $u(r)=c$.
			\item
			\begin{enumerate}
				\item For $r={^{\infty}}r^{c}_{f}\in {^{\infty}}Q_{1}$ with $c={^{\infty}}c^{i}_{e}$, let $s(r)=v_{(f,i,-\infty,-\infty,\ldots)}$, $t(r)=v_{(f,i,e)}$, and $u(r)=c$.
				\item For $r={^{\infty}}r^{c}_{f}\in {^{\infty}}Q_{1}$ with $c={^{\infty}}c^{i}_{-\infty}$, let $s(r)=t(r)=v_{(f,i,-\infty,-\infty,\ldots)}$, and $u(r)=c$.
			\end{enumerate}
		\end{enumerate}
	\end{enumerate}
	For any $i\in\mbZ_{\geq 0}$, denote by $\Gamma^{i}$ the colored full subquiver of $\Gamma$ corresponding to the subset $\{v_{e}\}_{e\in E^{i}}$ of $Q_{0}$. Since we have
	\begin{equation*}
		E^{i}=\{(i,-\infty,-\infty,\ldots)\}\amalg\coprod_{i'=i}^{\infty}E^{i'},
	\end{equation*}
	$\Gamma^{i}$ is illustrated as follows.
	\begin{equation*}
		\Gamma^{i}\colon
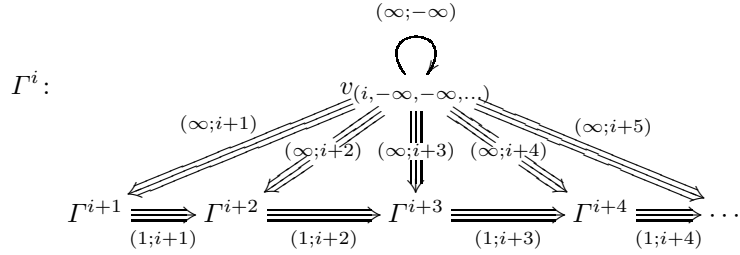
\xymatrix @R=12mm @L=2mm {
			& & v_{(i,-\infty,-\infty,\ldots)}\ar@(ul,ur)^-{(\infty;-\infty)}\ar@3@<-1mm>[dll]_-{(\infty;i+1)}\ar@3[dl]|-{(\infty;i+2)}\ar@3[d]|-{(\infty;i+3)}\ar@3[dr]|-{(\infty;i+4)}\ar@3@<1mm>[drr]^-{(\infty;i+5)} & & \\
			\Gamma^{i+1}\ar@3[r]_-{(1;i+1)} & \Gamma^{i+2}\ar@3[r]_-{(1;i+2)} & \Gamma^{i+3}\ar@3[r]_-{(1;i+3)} & \Gamma^{i+4}\ar@3[r]_-{(1;i+4)} & \cdots
		}
	\end{equation*}
	In this diagram, the colors of some arrows appearing in bold arrows are identified with the colors of some arrows in the colored full subquiver $\Gamma^{i+j}$ for some $j\in\mbZ_{\geq 1}$. For example, the colored full subquiver
	\begin{equation*}
		\xymatrix{
			\Gamma^{i+2}\ar@3[r]^-{(1;i+2)} & \Gamma^{i+3}\ar@3[r]^-{(1;i+3)} & \cdots
		}
	\end{equation*}
	of $\Gamma^{i}$ appearing in the above diagram is isomorphic to a colored full subquiver of $\Gamma^{i+1}$.
	
	For any $e,e'\in E^{i}$, we have $x_{v_{e}}S_{C}\subset x_{v_{e'}}S_{C}$ if and only if $e\geq e'$. For each $\widetilde{e}=(i_{0},i_{1},\ldots)\in\widetilde{E}$ with $i_{0}=i$, define a submodule $L_{>\widetilde{e}}$ of $M_{\Gamma^{i}}$ by
	\begin{equation*}
		L_{>\widetilde{e}}=\sum_{\substack{e\in E^{i}\\\widetilde{e}<e}}x_{v_{e}}S_{C}.
	\end{equation*}
	By a similar argument to that in the proof of Theorem \ref{Thm:PosetOfGrothendieckCategory}, for any nonzero element $x$ of $M_{\Gamma^{i}}$, there exists $e\in E^{i}$ such that $xS_{C}=x_{v_{e}}S_{C}$. Furthermore, any nonzero proper submodule $L$ of $M_{\Gamma^{i}}$ is of one of the following forms.
	\begin{enumerate}
		\item $L=x_{v_{e}}S_{C}$ for some $e=(i_{0},\ldots,i_{l},-\infty,-\infty,\ldots)\in E^{i}$, where $l\geq 1$.
		\item $L=L_{>\widetilde{e}}$ for some $\widetilde{e}=(i_{0},i_{1},\ldots)\in\widetilde{E}$, where $i_{0},i_{1},\ldots\in\mbZ_{\geq 0}$ and $i_{0}<i_{1}<\cdots$.
	\end{enumerate}
	In the first case, if $i_{l}=i_{l-1}+1$, then $M_{\Gamma^{i}}/L$ has an essential submodule isomorphic to the simple $S_{C}$-module defined by the colored quiver
	\begin{equation*}
		\xymatrix{
			v_{(i_{0},\ldots,i_{l-1},-\infty,\ldots)}\ar@(dl,dr)_-{{^{\infty}}c^{i_{l-1}}_{-\infty}}
		}.
	\end{equation*}
	
	If $i_{l}\neq i_{l-1}+1$, then $M_{\Gamma^{i}}/L$ has an essential submodule isomorphic to $M_{\Gamma^{i+j}}$ for some $j\in\mbZ_{\geq 1}$.
	
	In the second case, if there exists $l\in\mbZ_{\geq 1}$ such that $\widetilde{e}=(i_{0},\ldots,i_{l},i_{l}+1,i_{l}+2,\ldots)$, then the essential submodule $(x_{v_{e}}S_{C})/L$ of $M_{\Gamma^{i}}$/L, where $e=(i_{0},\ldots,i_{l},-\infty,-\infty,\ldots)$, is isomorphic to the noetherian $S_{C}$-module defined by the colored quiver
	\begin{equation*}
		\xymatrix{
			v_{(i_{0},\ldots,i_{l},-\infty,\ldots)}\ar@(dl,dr)_-{^{\infty}c^{i_{l}}_{-\infty}}\ar[r]^-{^{\infty}c^{i_{l}}_{(i_{l}+1,-\infty,\ldots)}} & v_{(i_{0},\ldots,i_{l},i_{l}+1,-\infty,\ldots)}\ar@(dl,dr)_-{^{\infty}c^{i_{l}+1}_{-\infty}}\ar[r]^-{^{\infty}c^{i_{l}+1}_{(i_{l}+2,-\infty,\ldots)}} & \cdots
		}.
	\end{equation*}
	Otherwise, we show that any nonzero submodule $L'/L$ of $M_{\Gamma^{i}}/L$ is not monoform. By Proposition \ref{Prop:PropertyOfMonoformObject} (\ref{item:PropSubobjectOfMonoformObjectIsMonoform}), we can assume that $L'=x_{v_{e}}S_{C}$ for some $e=(i'_{0},\ldots,i'_{l},-\infty,-\infty,\ldots)\in E^{i}$ with $e<\widetilde{e}$. By the assumption, there exists $l'\in\mbZ_{\geq 1}$ satisfying $i_{l'}\neq i_{l'-1}+1$ and $l<l'$. Define $\widetilde{e}'\in\widetilde{E}$ and $e',e''\in E$ by
	\begin{align*}
		\widetilde{e}'&=(i_{0},\ldots,i_{l'-1},i_{l'}-1,i_{l'},i_{l'+1},\ldots),\\
		e'&=(i_{0},\ldots,i_{l'},-\infty,-\infty,\ldots),\\
		e''&=(i_{0},\ldots,i_{l'-1},i_{l'}-1,i_{l'},-\infty,-\infty,\ldots).
	\end{align*}
	Then we have
	\begin{equation*}
		e<e''<\widetilde{e}'<e'<\widetilde{e},
	\end{equation*}
	and $(x_{v_{e'}}S_{C})/L_{>\widetilde{e}}$ is isomorphic to $(x_{v_{e''}}S_{C})/L_{>\widetilde{e}'}$. This shows that $L'/L$ is not monoform.
	
	Since $M_{\Gamma^{i}}$ has no monoform submodule, any atom in $\mcA_{\Gamma}$ is represented by a noetherian $S_{C}$-module. Denote by $\mcX$ the smallest localizing subcategory of $\mcA_{\Gamma}$ containing all the noetherian $S_{C}$-modules. Let $\mcA=\mcA_{\Gamma}/\mcX$. Then by Theorem \ref{Thm:AtomSpectraOfQuotientCategory}, there exists no atom in $\mcA$.
	
	Since $M_{\Gamma^{0}}$ has no nonzero noetherian submodule, any nonzero submodule of $M_{\Gamma^{0}}$ does not belong to $\mcX$ by Proposition \ref{Prop:LocalizingSubcategoryGeneratedBySubcategory}. Let $I=E(M_{\Gamma^{0}})$. Since $M_{\Gamma^{0}}$ is uniform, the $S_{C}$-module $I$ is indecomposable and injective. Any nonzero submodule of $I$ does not belong to $\mcX$. Let $F\colon\mcA_{\Gamma}\to\mcA_{\Gamma}/\mcX$ be the canonical functor and $G\colon\mcA_{\Gamma}/\mcX\to\mcA_{\Gamma}$ its right adjoint. By \cite[Lemma 4.5.1 (2)]{Popescu}, the object $F(I)$ in $\mcA_{\Gamma}/\mcX=\mcA$ is injective. By Theorem \ref{Thm:PropertyOfQuotientCategory} (\ref{item:ThmDescriptionOfCompositeFunctorForQuotientCategory}), $GF(I)$ is isomorphic to $I$, and hence $F(I)$ is indecomposable.
\end{proof}

%%%%%%%%%%%%%%%%%%%%%%%%%%%%%%%%%%%%%%%%%%%%%%%%%%%%%%%%%%%%%%%%%%%%%%%%%%%%%%%%
\section{Acknowledgments}
%%%%%%%%%%%%%%%%%%%%%%%%%%%%%%%%%%%%%%%%%%%%%%%%%%%%%%%%%%%%%%%%%%%%%%%%%%%%%%%%

The author would like to express his deep gratitude to his supervisor Osamu Iyama for his elaborated guidance. The author thanks Shiro Goto for his valuable comments.

%%%%%%%%%%%%%%%%%%%%%%%%%%%%%%%%%%%%%%%%%%%%%%%%%%%%%%%%%%%%%%%%%%%%%%%%%%%%%%%%

%%%%%%%%%%%%%%%%%%%%%%%%%%%%%%%%%%%%%%%%%%%%%%%%%%%%%%%%%%%%%%%%%%%%%%%%%%%%%%%%

\end{document}